\pdfoutput=1
\documentclass[12pt]{amsart}
\usepackage{lmodern}
\usepackage{ifthen}
\usepackage{amsfonts}
\usepackage{amsxtra}
\usepackage{amssymb}
\usepackage{mathdots}
\usepackage{array}
\usepackage[margin=1.0in]{geometry}
\usepackage{xcolor}
\definecolor{cite}{rgb}{0.30,0.60,1.00}
\definecolor{url}{rgb}{0.00,0.00,0.80}
\definecolor{link}{rgb}{0.40,0.10,0.20}
\usepackage[colorlinks,linkcolor=link,urlcolor=url,citecolor=cite,pagebackref,breaklinks]{hyperref}
\usepackage{bbm}
\usepackage{mathtools}
\usepackage{mathrsfs}
\usepackage{appendix}
\usepackage[all]{xy}
\usepackage[lite,abbrev,msc-links,alphabetic]{amsrefs}

\usepackage{graphicx}
\usepackage{multirow}
\usepackage{pstricks}
\usepackage{pst-pdf}
\usepackage{enumitem}
\usepackage{pifont}
\usepackage{marvosym}
\usepackage{tensor}

\usepackage[OT2,T1]{fontenc}
\DeclareSymbolFont{cyrletters}{OT2}{wncyr}{m}{n}
\DeclareMathSymbol{\Sha}{\mathalpha}{cyrletters}{"58}

\usepackage{graphicx}

\makeatletter
\providecommand*{\Dashv}{%
  \mathrel{%
    \mathpalette\@Dashv\vDash
  }%
}
\newcommand*{\@Dashv}[2]{%
  \reflectbox{$\m@th#1#2$}%
}
\makeatother

\setitemize[0]{leftmargin=20pt}
\setenumerate[0]{leftmargin=30pt}

\numberwithin{equation}{section}

\theoremstyle{plain}
\newtheorem{proposition}{Proposition}[section]

\newtheorem{corollary}[proposition]{Corollary}
\newtheorem{lem}[proposition]{Lemma}
\newtheorem{theorem}[proposition]{Theorem}

\newtheorem{hypothesis}[proposition]{Hypothesis}

\theoremstyle{definition}
\newtheorem{definition}[proposition]{Definition}

\newtheorem{notation}[proposition]{Notation}

\newtheorem{assumption}[proposition]{Assumption}

\theoremstyle{remark}
\newtheorem{remark}[proposition]{Remark}
\newtheorem{example}[proposition]{Example}

\renewcommand{\b}[1]{\mathbf{#1}}
\renewcommand{\c}[1]{\mathcal{#1}}
\renewcommand{\d}[1]{\mathbb{#1}}
\newcommand{\f}[1]{\mathfrak{#1}}
\renewcommand{\r}[1]{\mathrm{#1}}
\newcommand{\s}[1]{\mathscr{#1}}
\renewcommand{\sf}[1]{\mathsf{#1}}
\renewcommand{\(}{\left(}
\renewcommand{\)}{\right)}
\newcommand{\res}{\mathbin{|}}
\newcommand{\ol}[1]{\overline{#1}{}}

\newcommand{\ul}{\underline}
\renewcommand{\leq}{\leqslant}
\renewcommand{\geq}{\geqslant}

\newcommand{\bu}{\b u}

\newcommand{\bw}{\b w}

\newcommand{\cA}{\c A}

\newcommand{\cC}{\c C}

\newcommand{\cH}{\c H}
\newcommand{\cI}{\c I}

\newcommand{\cL}{\c L}

\newcommand{\cN}{\c N}

\newcommand{\cS}{\c S}
\newcommand{\cT}{\c T}
\newcommand{\cU}{\c U}
\newcommand{\cV}{\c V}
\newcommand{\cW}{\c W}
\newcommand{\cX}{\c X}
\newcommand{\cY}{\c Y}
\newcommand{\cZ}{\c Z}

\newcommand{\dA}{\d A}

\newcommand{\dC}{\d C}

\newcommand{\dF}{\d F}
\newcommand{\dG}{\d G}

\newcommand{\dI}{\d I}

\newcommand{\dL}{\d L}

\newcommand{\dP}{\d P}
\newcommand{\dQ}{\d Q}
\newcommand{\dR}{\d R}
\newcommand{\dS}{\d S}
\newcommand{\dT}{\d T}

\newcommand{\dZ}{\d Z}

\newcommand{\fD}{\f D}
\newcommand{\fE}{\f E}
\newcommand{\fF}{\f F}

\newcommand{\fL}{\f L}

\newcommand{\fS}{\f S}
\newcommand{\fT}{\f T}

\newcommand{\fZ}{\f Z}

\newcommand{\ff}{\f f}
\newcommand{\fg}{\f g}

\newcommand{\fm}{\f m}

\newcommand{\fp}{\f p}

\newcommand{\rD}{\r D}

\newcommand{\rF}{\r F}

\newcommand{\rH}{\r H}
\newcommand{\rI}{\r I}

\newcommand{\rK}{\r K}

\newcommand{\rM}{\r M}

\newcommand{\rP}{\r P}

\newcommand{\rT}{\r T}
\newcommand{\rU}{\r U}

\newcommand{\rZ}{\r Z}

\newcommand{\rb}{\r b}

\newcommand{\rd}{\,\r d}
\newcommand{\re}{\r e}

\newcommand{\rh}{\r h}

\newcommand{\rs}{\r s}
\newcommand{\rt}{\r t}

\newcommand{\rv}{\r v}

\newcommand{\sI}{\s I}

\newcommand{\sO}{\s O}

\newcommand{\sS}{\s S}

\newcommand{\sfF}{\sf F}

\newcommand{\sfI}{\sf I}
\newcommand{\sfJ}{\sf J}

\newcommand{\sfL}{\sf L}
\newcommand{\sfM}{\sf M}

\newcommand{\sfP}{\sf P}
\newcommand{\sfQ}{\sf Q}

\newcommand{\sfT}{\sf T}

\newcommand{\sfV}{\sf V}
\newcommand{\sfW}{\sf W}

\newcommand{\sff}{\sf f}

\newcommand{\sfu}{\sf u}

\newcommand{\tR}{\mathtt{R}}
\newcommand{\tS}{\mathtt{S}}
\newcommand{\tT}{\mathtt{T}}
\newcommand{\tU}{\mathtt{U}}
\newcommand{\tV}{\mathtt{V}}

\newcommand{\tc}{\mathtt{c}}

\newcommand{\tw}{\mathtt{w}}

\newcommand{\biota}{\boldsymbol{\iota}}

\newcommand{\bbn}{\boldsymbol{n}}

\newcommand{\bbA}{\boldsymbol{A}}
\newcommand{\bbD}{\boldsymbol{D}}
\newcommand{\bbL}{\boldsymbol{L}}
\newcommand{\bbP}{\boldsymbol{P}}
\newcommand{\bbV}{\boldsymbol{V}}
\newcommand{\bbX}{\boldsymbol{X}}
\newcommand{\bbLambda}{\boldsymbol{\Lambda}}
\newcommand{\bbphi}{\boldsymbol{\phi}}

\newcommand{\pres}[2]{\prescript{#1}{}{{\hskip-0.1em\relax}#2}}

\newcommand{\ac}{\r{ac}}

\newcommand{\CF}{\mathbbm{1}}

\newcommand{\cusp}{\r{cusp}}

\newcommand{\Den}{\r{Den}}

\newcommand{\dr}{\r{dR}}
\newcommand{\dtm}{\r{det}\:}

\newcommand{\Exo}{\r{Exo}}
\newcommand{\fin}{\r{fin}}

\newcommand{\Groth}{\mathrm{Groth}}

\newcommand{\Herm}{\r{Herm}}

\newcommand{\id}{\r{id}}

\newcommand{\inert}{\r{int}}
\newcommand{\integ}{\r{int}}
\newcommand{\Int}{\r{Int}}

\newcommand{\length}{\r{length}}

\newcommand{\ram}{\r{ram}}
\newcommand{\rec}{\r{rec}}
\newcommand{\reg}{\r{reg}}
\renewcommand{\red}{\r{red}}
\newcommand{\Red}{\r{Red}}
\newcommand{\rel}{\r{rel}}
\newcommand{\RE}{\r{Re}\,}

\newcommand{\Sph}{\r{Sph}}

\newcommand{\spl}{\r{spl}}

\DeclareMathOperator{\Aut}{Aut}

\DeclareMathOperator{\BC}{BC}
\DeclareMathOperator{\CH}{CH}

\DeclareMathOperator{\coker}{coker}

\DeclareMathOperator{\Diff}{Diff}

\DeclareMathOperator{\End}{End}

\DeclareMathOperator{\Fil}{Fil}

\DeclareMathOperator{\Gal}{Gal}

\DeclareMathOperator{\GL}{GL}

\DeclareMathOperator{\Hom}{Hom}
\DeclareMathOperator{\IM}{im}

\DeclareMathOperator{\Ind}{Ind}

\DeclareMathOperator{\Isom}{Isom}

\DeclareMathOperator{\Ker}{ker}

\DeclareMathOperator{\Lie}{Lie}

\DeclareMathOperator{\Mat}{Mat}

\DeclareMathOperator{\Nm}{Nm}

\DeclareMathOperator{\OG}{O}
\DeclareMathOperator{\ord}{ord}

\DeclareMathOperator{\Res}{Res}

\DeclareMathOperator{\Sch}{Sch}

\DeclareMathOperator{\SO}{SO}
\DeclareMathOperator{\Sp}{Sp}

\DeclareMathOperator{\Spec}{Spec}
\DeclareMathOperator{\Spf}{Spf}

\DeclareMathOperator{\supp}{supp}
\DeclareMathOperator{\Sym}{Sym}

\DeclareMathOperator{\tr}{tr}
\DeclareMathOperator{\Tr}{Tr}

\DeclareMathOperator{\val}{val}

\DeclareMathOperator{\vol}{vol}

\begin{document}

\title{Chow groups and $L$-derivatives of automorphic motives for unitary groups, II.}

\author{Chao Li}
\address{Department of Mathematics, Columbia University, New York NY 10027, United States}
\email{chaoli@math.columbia.edu}

\author{Yifeng Liu}
\address{Department of Mathematics, Yale University, New Haven CT 06511, United States}
\email{yifeng.liu@yale.edu}

\date{\today}

\begin{abstract}
  In this article, we improve our main results from \cite{LL} in two direction: First, we allow ramified places in the CM extension $E/F$ at which we consider representations that are spherical with respect to a certain special maximal compact subgroup, by formulating and proving an analogue of the Kudla--Rapoport conjecture for exotic smooth Rapoport--Zink spaces. Second, we lift the restriction on the components at split places of the automorphic representation, by proving a more general vanishing result on certain cohomology of integral models of unitary Shimura varieties with Drinfeld level structures.
\end{abstract}

\maketitle

\tableofcontents

\section{Introduction}
\label{ss:introduction}

In \cite{LL}, we proved that for certain cuspidal automorphic representations $\pi$ on unitary groups of even ranks, if the central derivative $L'(1/2,\pi)$ is nonvanishing, then the $\pi$-nearly isotypic localization of the Chow group of a certain unitary Shimura variety over its reflex field does not vanish. This proved part of the Beilinson--Bloch conjecture for Chow groups and $L$-functions. Moreover, assuming the modularity of Kudla's generating functions of special cycles, we further proved the arithmetic inner product formula relating $L'(1/2,\pi)$ and the height of arithmetic theta liftings. In this article, we improve the main results from \cite{LL} in two directions: First, we allow ramified places in the CM extension $E/F$ at which we consider representations that are spherical with respect to a certain special maximal compact subgroup, by formulating and proving an analogue of the Kudla--Rapoport conjecture for exotic smooth Rapoport--Zink spaces. Second, we lift the restriction on the components at split places of the automorphic representation, by proving a more general vanishing result on certain cohomology of integral models of unitary Shimura varieties with Drinfeld level structures. However, for technical reasons, we will still assume $F\neq\dQ$ (see Remark \ref{re:final}).

\subsection{Main results}

Let $E/F$ be a CM extension of number fields with the complex conjugation $\tc$. Denote by $\tV_F^{(\infty)}$ and $\tV_F^\fin$ the set of archimedean and non-archimedean places of $F$, respectively; and $\tV_F^\spl$, $\tV_F^\inert$, and $\tV_F^\ram$ the subsets of $\tV_F^\fin$ of those that are split, inert, and ramified in $E$, respectively. For every $v\in\tV_F^\fin$, we denote by $q_v$ the residue cardinality of $F_v$.

Take an even positive integer $n=2r$. We equip $W_r\coloneqq E^n$ with the skew-hermitian form (with respect to the involution $\tc$) given by the matrix $\(\begin{smallmatrix}&1_r\\ -1_r &\end{smallmatrix}\)$. Put $G_r\coloneqq\rU(W_r)$, the unitary group of $W_r$, which is a quasi-split reductive group over $F$. For every $v\in\tV_F^\fin$, we denote by $K_{r,v}\subseteq G_r(F_v)$ the stabilizer of the lattice $O_{E_v}^n$, which is a special maximal compact subgroup.

We start from an informal discussion on the arithmetic inner product formula. Let $\pi$ be a tempered automorphic representation of $G_r(\dA_F)$, which by theta dichotomy, gives rise to a unique up to isomorphism hermitian space $V_\pi$ of rank $n$ over $\dA_E$. It is known that the hermitian space $V_\pi$ is \emph{coherent} (resp.\ \emph{incoherent}), that is, $V_\pi$ is (resp.\ is not) the base change of a hermitian space over $E$, if and only if the global root number $\varepsilon(\pi)$ equals $1$ (resp.\ $-1$). When $\varepsilon(\pi)=1$, we have the global theta lifting of $\pi$, which is a space of automorphic forms on $\rU(V_\pi)(\dA_F)$; and the famous Rallis' inner product formula \cite{Ral84} computes the Petersson inner product of the global theta lifting in terms of the central $L$-value $L(\tfrac{1}{2},\pi)$ of $\pi$. When $\varepsilon(\pi)=-1$, we have the arithmetic theta lifting of $\pi$, which is a space of algebraic cycles on the Shimura variety associated to $V_\pi$; and the conjectural \emph{arithmetic inner product formula} \cite{Liu11} computes the height of the arithmetic theta lifting in terms of the central $L$-derivative $L'(\tfrac{1}{2},\pi)$ of $\pi$. In our previous article \cite{LL}, we verify the arithmetic inner product formula, under certain hypotheses, when $E/F$ and $\pi$ satisfy certain local conditions (see \cite{LL}*{Assumption~1.3}). In particular, we want $\tV_F^\ram=\emptyset$, which forces $[F:\dQ]$ to be even; and we want the representation $\pi$ has to be either unramified or almost unramified at $v\in\tV_F^\inert$. Computing local root numbers, we have $\varepsilon(\pi_v)=(-1)^r$ if $v\in\tV_F^{(\infty)}$, $\varepsilon(\pi_v)=1$ if $v\in\tV_F^\spl$ or $\pi_v$ is unramified, $\varepsilon(\pi_v)=-1$ if ($v\in\tV_F^\inert$ and) $\pi_v$ is almost unramified. It follows that $\varepsilon(\pi)=(-1)^{r[F:\dQ]+|\tS_\pi|}$, where $\tS_\pi\subseteq\tV_F^\inert$ denotes the (finite) subset at which $\pi$ is almost unramified, which equals $(-1)^{|\tS_\pi|}$ as $[F:\dQ]$ is even. In this article, we improve our results so that $\tV_F^\ram$ can be nonempty hence $[F:\dQ]$ can be odd; and we will still have $\varepsilon(\pi)=(-1)^{r[F:\dQ]+|\tS_\pi|}$. To show the significance of such improvement, now we may have $\varepsilon(\pi)=-1$ but $\tS_\pi=\emptyset$, so that we can accommodate $\pi$ coming from certain explicit motives like symmetric power of elliptic curves (see Example \ref{ex:symmetric}).

Readers may read the introduction of \cite{LL} for more background. Now we describe in more details our setup and main results in the current article.

\begin{definition}\label{de:heart}
We define the subset $\tV_F^\heartsuit$ of $\tV_F^\spl\cup\tV_F^\inert$ consisting of $v$ satisfying that for every $v'\in\tV_F^{(p)}\cap\tV_F^\ram$, where $p$ is the underlying rational prime of $v$, the subfield of $\ol{F_v}$ generated by $F_v$ and the Galois closure of $E_{v'}$ is unramified over $F_v$.
\end{definition}

\begin{remark}\label{re:heart}
The purpose of this technical definition is that for certain places $v$ in $\tV_F^\spl\cup\tV_F^\inert$, we need to have a CM type of $E$ such that its reflex field does not contain more ramification over $p$ than $F_v$ does -- this is possible for $v\in\tV_F^\heartsuit$. Note that
\begin{itemize}
  \item the complement $(\tV_F^\spl\cup\tV_F^\inert)\setminus\tV_F^\heartsuit$ is finite;

  \item when $E$ is Galois, or contains an imaginary quadratic field, or satisfies $\tV_F^\ram=\emptyset$, we have $\tV_F^\heartsuit=\tV_F^\spl\cup\tV_F^\inert$.
\end{itemize}
\end{remark}

\begin{assumption}\label{st:main}
Suppose that $F\neq\dQ$, that $\tV_F^\spl$ contains all 2-adic places, and that every prime in $\tV_F^\ram$ is unramified over $\dQ$. We consider a cuspidal automorphic representation $\pi$ of $G_r(\dA_F)$ realized on a space $\cV_\pi$ of cusp forms, satisfying:
\begin{enumerate}
  \item For every $v\in\tV_F^{(\infty)}$, $\pi_v$ is the holomorphic discrete series representation of Harish-Chandra parameter $\{\tfrac{n-1}{2},\tfrac{n-3}{2},\dots,\tfrac{3-n}{2},\tfrac{1-n}{2}\}$ (see \cite{LL}*{Remark~1.4(1)}).

  \item For every $v\in\tV_F^\ram$, $\pi_v$ is spherical with respect to $K_{r,v}$, that is, $\pi_v^{K_{r,v}}\neq\{0\}$.

  \item For every $v\in\tV_F^\inert$, $\pi_v$ is either unramified or almost unramified (see \cite{LL}*{Remark~1.4(3)}) with respect to $K_{r,v}$; moreover, if $\pi_v$ is almost unramified, then $v$ is unramified over $\dQ$.

  \item For every $v\in\tV_F^\fin$, $\pi_v$ is tempered.\footnote{In fact, (4) is implied by (1). See \cite{LL}*{Remark~1.4(4)}.}

  \item We have $\tR_\pi\cup\tS_\pi\subseteq\tV_F^\heartsuit$ (Definition \ref{de:heart}), where
     \begin{itemize}
         \item $\tR_\pi\subseteq\tV_F^\spl$ denotes the (finite) subset for which $\pi_v$ is ramified,

         \item $\tS_\pi\subseteq\tV_F^\inert$ denotes the (finite) subset for which $\pi_v$ is almost unramified.
     \end{itemize}
\end{enumerate}
\end{assumption}

Comparing Assumption \ref{st:main} with \cite{LL}*{Assumption~1.3}, we have lifted the restriction that $\tV_F^\ram=\emptyset$ (by allowing $\pi_v$ to be a certain type of representations for $v\in\tV_F^\ram$), and also the restriction on $\pi_v$ for $v\in\tV_F^\spl$. Note that (5) is not really a new restriction since when $\tV_F^\ram=\emptyset$, it is automatic by Remark \ref{re:heart}.

Suppose that we are in Assumption \ref{st:main}. Denote by $L(s,\pi)$ the doubling $L$-function. Then we have $\varepsilon(\pi)=(-1)^{r[F:\dQ]+|\tS_\pi|}$ for the global (doubling) root number, so that the vanishing order of $L(s,\pi)$ at the center $s=\tfrac{1}{2}$ has the same parity as $r[F:\dQ]+|\tS_\pi|$. The cuspidal automorphic representation $\pi$ determines a hermitian space $V_\pi$ over $\dA_E$ of rank $n$ via local theta dichotomy (so that the local theta lifting of $\pi_v$ to $\rU(V_\pi)(F_v)$ is nontrivial for every place $v$ of $F$), unique up to isomorphism, which is totally positive definite and satisfies that for every $v\in\tV_F^\fin$, the local Hasse invariant $\epsilon(V_\pi\otimes_{\dA_F}F_v)=1$ if and only if $v\not\in\tS_\pi$.

Now suppose that $r[F:\dQ]+|\tS_\pi|$ is odd hence $\varepsilon(\pi)=-1$, which is equivalent to that $V_\pi$ is incoherent. In what follows, we take $V=V_\pi$ in the context of \cite{LL}*{Conjecture~1.1}, hence $H=\rU(V_\pi)$. Let $\tR$ be a finite subset of $\tV_F^\fin$. We fix a special maximal subgroup $L^\tR$ of $H(\dA_F^{\infty,\tR})$ that is the stabilizer of a lattice $\Lambda^\tR$ in $V\otimes_{\dA_F}\dA_F^{\infty,\tR}$ (see Notation \ref{st:h}(H6) for more details). For a field $\dL$, we denote by $\dT^\tR_\dL$ the (abstract) Hecke algebra $\dL[L^\tR\backslash H(\dA_F^{\infty,\tR})/L^\tR]$, which is a commutative $\dL$-algebra. When $\tR$ contains $\tR_\pi$, the cuspidal automorphic representation $\pi$ gives rise to a character
\[
\chi_\pi^\tR\colon\dT^\tR_{\dQ^\ac}\to\dQ^\ac,
\]
where $\dQ^\ac$ denotes the subfield of $\dC$ of algebraic numbers; and we put $\fm_\pi^\tR\coloneqq\Ker\chi_\pi^\tR$, which is a maximal ideal of $\dT^\tR_{\dQ^\ac}$.

In what follows, we will fix an arbitrary embedding $\biota\colon E\hookrightarrow\dC$ and denote by $\{X_L\}$ the system of unitary Shimura varieties of dimension $n-1$ over $\biota(E)$ indexed by open compact subgroups $L\subseteq H(\dA_F^\infty)$ (see Subsection \ref{ss:atl} for more details). The following is the first main theorem of this article.

\begin{theorem}\label{th:main}
Let $(\pi,\cV_\pi)$ be as in Assumption \ref{st:main} with $r[F:\dQ]+|\tS_\pi|$ odd, for which we assume \cite{LL}*{Hypothesis~6.6}. If $L'(\tfrac{1}{2},\pi)\neq 0$, that is, $\ord_{s=\frac{1}{2}}L(s,\pi)=1$, then as long as $\tR$ satisfies $\tR_\pi\subseteq\tR$ and $|\tR\cap\tV_F^\spl\cap\tV_F^\heartsuit|\geq 2$, the nonvanishing
\[
\varinjlim_{L_\tR}\(\CH^r(X_{L_\tR L^\tR})^0_{\dQ^\ac}\)_{\fm_\pi^\tR}\neq\{0\}
\]
holds, where the colimit is taken over all open compact subgroups $L_\tR$ of $H(F_\tR)$.
\end{theorem}

Our remaining results rely on Hypothesis \ref{hy:modularity} on the modularity of Kudla's generating functions of special cycles, hence are conditional at this moment.

\begin{theorem}\label{th:aipf}
Let $(\pi,\cV_\pi)$ be as in Assumption \ref{st:main} with $r[F:\dQ]+|\tS_\pi|$ odd, for which we assume \cite{LL}*{Hypothesis~6.6}. Assume Hypothesis \ref{hy:modularity} on the modularity of generating functions of codimension $r$.
\begin{enumerate}
  \item For every test vectors
     \begin{itemize}
       \item $\varphi_1=\otimes_v\varphi_{1v}\in\cV_{\pi}$ and $\varphi_2=\otimes_v\varphi_{2v}\in\cV_{\pi}$ such that for every $v\in\tV_F^{(\infty)}$, $\varphi_{1v}$ and $\varphi_{2v}$ have the lowest weight and satisfy $\langle\varphi_{1v}^\tc,\varphi_{2v}\rangle_{\pi_v}=1$,

       \item $\phi^\infty_1=\otimes_v\phi^\infty_{1v}\in\sS(V^r\otimes_{\dA_F}\dA_F^\infty)$ and $\phi^\infty_2=\otimes_v\phi^\infty_{2v}\in\sS(V^r\otimes_{\dA_F}\dA_F^\infty)$,
     \end{itemize}
     the identity
     \[
     \langle\Theta_{\phi^\infty_1}(\varphi_1),\Theta_{\phi^\infty_2}(\varphi_2)\rangle_{X,E}^\natural=
     \frac{L'(\tfrac{1}{2},\pi)}{b_{2r}(0)}\cdot C_r^{[F:\dQ]}
     \cdot\prod_{v\in\tV_F^\fin}\fZ^\natural_{\pi_v,V_v}(\varphi^\tc_{1v},\varphi_{2v},\phi_{1v}^\infty\otimes(\phi_{2v}^\infty)^\tc)
     \]
     holds. Here,
     \begin{itemize}
       \item $\Theta_{\phi^\infty_i}(\varphi_i)\in\varinjlim_L\CH^r(X_L)^0_\dC$ is the arithmetic theta lifting (Definition \ref{de:arithmetic_theta}), which is only well-defined under Hypothesis \ref{hy:modularity};

       \item $\langle\Theta_{\phi^\infty_1}(\varphi_1),\Theta_{\phi^\infty_2}(\varphi_2)\rangle_{X,E}^\natural$ is the normalized height pairing (Definition \ref{de:natural}),\footnote{Strictly speaking, $\langle\Theta_{\phi^\infty_1}(\varphi_1),\Theta_{\phi^\infty_2}(\varphi_2)\rangle_{X,E}^\natural$ relies on the choice of a rational prime $\ell$ and is a priori an element in $\dC\otimes_\dQ\dQ_\ell$. However, the above identity implicitly says that it belongs to $\dC$ and is independent of the choice of $\ell$.} which is constructed based on Beilinson's notion of height pairing;

       \item $b_{2r}(0)$ is defined in Notation \ref{st:f}(F4), which equals $L(M_r^\vee(1))$ where $M_r$ is the motive associated to $G_r$ by Gross \cite{Gr97}, and is in particular a positive real number;

       \item $C_r=(-1)^r2^{r(r-1)}\pi^{r^2}\frac{\Gamma(1)\cdots\Gamma(r)}{\Gamma(r+1)\cdots\Gamma(2r)}$, which is the exact value of a certain archimedean doubling zeta integral; and

       \item $\fZ^\natural_{\pi_v,V_v}(\varphi^\tc_{1v},\varphi_{2v},\phi_{1v}^\infty\otimes(\phi_{2v}^\infty)^\tc)$ is the normalized local doubling zeta integral \cite{LL}*{Section~3}, which equals $1$ for all but finitely many $v$.
     \end{itemize}

  \item In the context of \cite{LL}*{Conjecture~1.1}, take ($V=V_\pi$ and) $\tilde\pi^\infty$ to be the theta lifting of $\pi^\infty$ to $H(\dA_F^\infty)$. If $L'(\tfrac{1}{2},\pi)\neq 0$, that is, $\ord_{s=\frac{1}{2}}L(s,\pi)=1$, then
     \[
     \Hom_{H(\dA_F^\infty)}\(\tilde\pi^\infty,\varinjlim_{L}\CH^r(X_L)^0_\dC\)\neq\{0\}
     \]
     holds.
\end{enumerate}
\end{theorem}

\begin{remark}\label{re:aipf}
We have the following remarks concerning Theorem \ref{th:aipf}.
\begin{enumerate}
  \item Part (1) verifies the so-called \emph{arithmetic inner product formula}, a conjecture proposed by one of us \cite{Liu11}*{Conjecture~3.11}.

  \item The arithmetic inner product formula in part (1) is perfectly parallel to the classical Rallis inner product formula. In fact, suppose that $V$ is totally positive definite but \emph{coherent}. We have the classical theta lifting $\theta_{\phi^\infty}(\varphi)$ where we use standard Gaussian functions at archimedean places. Then the Rallis inner product formula in this case reads as
      \[
      \langle\theta_{\phi^\infty_1}(\varphi_1),\theta_{\phi^\infty_2}(\varphi_2)\rangle_H=
      \frac{L(\tfrac{1}{2},\pi)}{b_{2r}(0)}\cdot C_r^{[F:\dQ]}
      \cdot\prod_{v\in\tV_F^\fin}\fZ^\natural_{\pi_v,V_v}(\varphi^\tc_{1v},\varphi_{2v},\phi_{1v}^\infty\otimes(\phi_{2v}^\infty)^\tc),
      \]
      in which $\langle\;,\;\rangle_H$ denotes the Petersson inner product with respect to the \emph{Tamagawa measure} on $H(\dA_F)$.
\end{enumerate}
\end{remark}

In the case where $\tR_\pi=\emptyset$, we have a very explicit height formula for test vectors that are new everywhere.

\begin{corollary}\label{co:aipf}
Let $(\pi,\cV_\pi)$ be as in Assumption \ref{st:main} with $r[F:\dQ]+|\tS_\pi|$ odd, for which we assume \cite{LL}*{Hypothesis~6.6}. Assume Hypothesis \ref{hy:modularity} on the modularity of generating functions of codimension $r$. In the situation of Theorem \ref{th:aipf}(1), suppose further that
\begin{itemize}
  \item $\tR_\pi=\emptyset$;

  \item $\varphi_1=\varphi_2=\varphi\in\cV_\pi^{[r]\emptyset}$ (see Notation \ref{st:g}(G8) for the precise definition of the one-dimensional space $\cV_\pi^{[r]\emptyset}$ of holomorphic new forms) such that for every $v\in\tV_F$, $\langle\varphi_v^\tc,\varphi_v\rangle_{\pi_v}=1$; and

  \item $\phi^\infty_1=\phi^\infty_2=\phi^\infty$ such that for every $v\in\tV_F^\fin$, $\phi^\infty_v=\CF_{(\Lambda_v^\emptyset)^r}$.
\end{itemize}
Then the identity
\[
\langle\Theta_{\phi^\infty}(\varphi),\Theta_{\phi^\infty}(\varphi)\rangle_{X,E}^\natural=(-1)^r\cdot
\frac{L'(\tfrac{1}{2},\pi)}{b_{2r}(0)}\cdot |C_r|^{[F:\dQ]}
\cdot\prod_{v\in\tS_\pi}\frac{q_v^{r-1}(q_v+1)}{(q_v^{2r-1}+1)(q_v^{2r}-1)}
\]
holds.
\end{corollary}

\begin{remark}
Assuming the conjecture on the injectivity of the \'{e}tale Abel--Jacobi map, one can show that the cycle $\Theta_{\phi^\infty}(\varphi)$ is a primitive cycle of codimension $r$. By \cite{Bei87}*{Conjecture~5.5}, we expect that $(-1)^r\langle\Theta_{\phi^\infty}(\varphi),\Theta_{\phi^\infty}(\varphi)\rangle_{X,E}^\natural\geq 0$ holds, which, in the situation of Corollary \ref{co:aipf}, is equivalent to $L'(\tfrac{1}{2},\pi)\geq 0$.
\end{remark}

\begin{remark}
When $\tS_\pi=\emptyset$, Theorem \ref{th:main}, Theorem \ref{th:aipf}, and Corollary \ref{co:aipf} hold without \cite{LL}*{Hypothesis~6.6}. See Remark \ref{re:galois} for more details.
\end{remark}

\begin{example}\label{ex:symmetric}
Suppose that $E/F$ satisfies the conditions in Assumption \ref{st:main} and that $r\geq 2$. Consider an elliptic curve $A$ over $F$ without complex multiplication, satisfying that $\Sym^{2r-1}A$ hence $\Sym^{2r-1}A_E$ are modular. Let $\Pi$ be the cuspidal automorphic representation of $\GL_n(\dA_E)$ corresponding to $\Sym^{2r-1}A_E$, which satisfies $\Pi^\vee\simeq\Pi\circ\tc$. Then there exists a cuspidal automorphic representation $\pi$ of $G_r(\dA_F)$ as in Assumption \ref{st:main} with $\Pi$ its base change if and only if $A$ has good reduction at every $v\in\tV_F^\fin\setminus\tV_F^\spl$.\footnote{Note that, when $r\geq 2$, the $(2r-1)$-th symmetric power of an irreducible admissible representation of $\GL_2(E_v)$ can never be the base change of an almost unramified representation of $G_r(F_v)$ for $v\in\tV_F^\inert$.} Moreover, if this is the case, then we have $\tS_\pi=\emptyset$ hence $\varepsilon(\pi)=(-1)^{r[F:\dQ]}$; in particular, above results apply when both $r$ and $[F:\dQ]$ are odd.
\end{example}

\subsection{Two new ingredients}
\label{ss:ingredients}

The proofs of our main theorems follow the same line in \cite{LL}, with two new (main) ingredients, responsible for the two improvements we have mentioned at the beginning.

The first new ingredient is formulating and proving an analogue of the Kudla--Rapoport conjecture in the case where $E/F$ is ramified and the level structure is the one that gives the exotic smooth model. Here, $F$ is a $p$-adic field with $p$ odd. Let $\bbL$ be an $O_E$-lattice of a nonsplit (nondegenerate) hermitian space $\bbV$ over $E$ of (even) rank $n$. Then one can associate an intersection number $\Int(\bbL)$ of special divisors on a formally smooth relative Rapoport--Zink space classifying quasi-isogenies of certain unitary $O_F$-divisible groups, and also the derivative of the representation density function $\partial\Den(\bbL)$ given by $\bbL$. We show in Theorem \ref{th:kr} the formula
\[
\Int(\bbL)=\partial\Den(\bbL).
\]
This is parallel to the Kudla--Rapoport conjecture proved in \cite{LZ}, originally stated for the case where $E/F$ is unramified. The proof follows from the same strategy as in \cite{LZ}, namely, we write $\bbL=L^\flat+\langle x\rangle$ for a sublattice $L^\flat$ of $\bbL$ such that $V_{L^\flat}\coloneqq L^\flat\otimes_{O_F}F$ is nondegenerate, and regard $x$ as a variable. Thus, it motivates us to define a function $\Int_{L^\flat}$ on $\bbV\setminus V_{L^\flat}$ by the formula $\Int_{L^\flat}(x)=\Int(L^\flat+\langle x\rangle)$ and similarly for $\partial\Den_{L^\flat}$. For $\Int_{L^\flat}$, there is a natural decomposition $\Int_{L^\flat}=\Int_{L^\flat}^\rh+\Int_{L^\flat}^\rv$ according to the horizontal and vertical parts of the special cycle defined by $L^\flat$. In a parallel manner, we have the decomposition $\partial\Den_{L^\flat}=\partial\Den_{L^\flat}^\rh+\partial\Den_{L^\flat}^\rv$ by simply matching $\partial\Den_{L^\flat}^\rh$ with $\Int_{L^\flat}^\rh$. Thus, it suffices to show that $\Int_{L^\flat}^\rv=\partial\Den_{L^\flat}^\rv$. By some sophisticated induction argument on $L^\flat$, it suffices to show the following remarkable property for both $\Int_{L^\flat}^\rv$ and $\partial\Den_{L^\flat}^\rv$: they extend (uniquely) to compactly supported locally constant functions on $\bbV$, whose Fourier transforms are supported in the set $\{x\in \bbV\res (x,x)_{\bbV}\in O_F\}$. However, there are some new difficulties in our case:
\begin{itemize}
  \item The isomorphism class of an $O_E$-lattice is not determined by its fundamental invariants, and there is a parity constraint for the valuation of an $O_E$-lattice. This will make the induction argument on $L^\flat$ much more complicated than the one in \cite{LZ} (see Subsection \ref{ss:kr_proof}).

  \item The comparison of our relative Rapoport--Zink space to an (absolute) Rapoport--Zink space is not known. This is needed even to show that our relative Rapoport--Zink space is representable, and also in the $p$-adic uniformization of Shimura varieties. We solve this problem when $F/\dQ_p$ is unramified, which is the reason for us to assume that every prime in $\tV_F^\ram$ is unramified over $\dQ$ in Assumption \ref{st:main}. See Subsection \ref{ss:comparison}.

  \item Due to the parity constraint, the computation of $\Int_{L^\flat}^\rv$ can only be reduced to the case where $n=4$ (rather than $n=3$ in \cite{LZ}). After that, we have to compute certain intersection multiplicity, for which we use a new argument based on the linear invariance of the K-theoretic intersection of special divisors. See Lemma \ref{le:geometric3}.
\end{itemize}
Here comes three more remarks:
\begin{itemize}
  \item First, we need to extend the result of \cite{CY20} on a counting formula for $\partial\Den(\bbL)$ to hermitian spaces over a ramified extension $E/F$ (Lemma \ref{le:analytic2}).

  \item Second, we have found a simpler argument for the properties of $\partial\Den_{L^\flat}^\rv$ (Proposition \ref{pr:analytic}), which does not use any functional equation or induction formula. This argument is applicable to \cite{LZ} to give a new proof of the main result on the analytic side there. Also note that we prove the vanishing property in Proposition \ref{pr:analytic} directly, while in \cite{LZ} it is only deduced after proving $\Int_{L^\flat}^\rv=\partial\Den_{L^\flat}^\rv$.

  \item Finally, unlike the case in \cite{LZ}, the parity of the dimension of the hermitian space plays a crucial role in the exotic smooth case. In particular, we will not study the case where $\bbV$ has odd dimension.
\end{itemize}

The second new ingredient is a vanishing result on certain cohomology of integral models of unitary Shimura varieties with Drinfeld level structures. For $v\in\tV_F^\spl\cap\tV_F^\heartsuit$ with $p$ the underlying rational prime, we have a tower of integral models $\{\cX_m\}_{m\geq 0}$ defined by Drinfeld level structures (at $v$), with an action by $\dT^{\tR\cup\tV_F^{(p)}}_{\dQ^\ac}$ via Hecke correspondences. We show in Theorem \ref{th:split_tempered} that
\[
\rH^{2r}(\cX_m,\ol\dQ_\ell(r))_\fm=0
\]
with $\ell\neq p$ and $\fm\coloneqq\fm_\pi^\tR\cap\dS^{\tR\cup\tV_F^{(p)}}_{\dQ^\ac}$, where $\dS^{\tR\cup\tV_F^{(p)}}_{\dQ^\ac}$ is the subalgebra of $\dT^{\tR\cup\tV_F^{(p)}}_{\dQ^\ac}$ consisting of those supported at split places. We reduce this vanishing property to some other vanishing properties for cohomology of Newton strata of $\cX_m$, by using a key result of Mantovan \cite{Man08} saying that the closure of every refined Newton stratum is smooth. For the vanishing properties for Newton strata, we generalize an argument of \cite{TY07}*{Proposition~4.4}. However, since in our case, the representation $\pi_v$ has arbitrary level and our group has nontrivial endoscopy, we need a more sophisticated trace formula, which was provided in \cite{CS17}.

\subsection{Notation and conventions}

\begin{itemize}
  \item When we have a function $f$ on a product set $A_1\times\cdots\times A_m$, we will write $f(a_1,\dots,a_m)$ instead of $f((a_1,\dots,a_m))$ for its value at an element $(a_1,\dots,a_m)\in A_1\times\cdots\times A_m$.

  \item For a set $S$, we denote by $\CF_S$ the characteristic function of $S$.

  \item All rings (but not algebras) are commutative and unital; and ring homomorphisms preserve units.

  \item For a (formal) subscheme $Z$ of a (formal) scheme $X$, we denote by $\sI_Z$ the ideal sheaf of $Z$, which is a subsheaf of the structure sheaf $\sO_X$ of $X$.

  \item For a ring $R$, we denote by $\Sch_{/R}$ the category of schemes over $R$, by $\Sch'_{/R}$ the subcategory of locally Noetherian schemes over $R$, and when $R$ is discretely valued, by $\Sch_{/R}^\rv$ the subcategory of schemes on which uniformizers of $R$ are locally nilpotent.

  \item If a base ring is not specified in the tensor operation $\otimes$, then it is $\dZ$.

  \item For an abelian group $A$ and a ring $R$, we put $A_R\coloneqq A\otimes R$.

  \item For an integer $m\geq 0$, we denote by $0_m$ and $1_m$ the null and identity matrices of rank $m$, respectively. We also denote by $\tw_m$ the matrix $\(\begin{smallmatrix}&1_m\\ -1_m &\end{smallmatrix}\)$.

  \item We denote by $\tc\colon\dC\to\dC$ the complex conjugation. For an element $x$ in a complex space with a default underlying real structure, we denote by $x^\tc$ its complex conjugation.

  \item For a field $K$, we denote by $\ol{K}$ the abstract algebraic closure of $K$. However, for aesthetic reason, we will write $\ol\dQ_p$ instead of $\ol{\dQ_p}$ and will denote by $\ol\dF_p$ its residue field. On the other hand, we denote by $\dQ^\ac$ the algebraic closure of $\dQ$ \emph{inside} $\dC$.

  \item For a number field $K$, we denote by $\psi_K\colon K\backslash \dA_K\to\dC^\times$ the standard additive character, namely, $\psi_K\coloneqq\psi_\dQ\circ\Tr_{K/\dQ}$ in which $\psi_\dQ\colon\dQ\backslash\dA\to\dC^\times$ is the unique character such that $\psi_{\dQ,\infty}(x)=\re^{2\pi ix}$.

  \item Throughout the entire article, all parabolic inductions are unitarily normalized.
\end{itemize}

\subsubsection*{Acknowledgements}

We thank Xuhua~He and Yichao~Tian for helpful discussion. We thank Benedict~Gross for useful comments. The research of C.~L. is partially supported by the NSF grant DMS--1802269. The research of Y.~L. is partially supported by the NSF grant DMS--2000533.

\section{Intersection of special cycles at ramified places}
\label{ss:2}

Throughout this section, we fix a \emph{ramified} quadratic extension $E/F$ of $p$-adic fields with $p$ odd, with $\tc\in\Gal(E/F)$ the Galois involution. We fix a uniformizer $u\in E$ satisfying $u^\tc=-u$. Let $k$ be the residue field of $F$ and denote by $q$ the cardinality of $k$. Let $n=2r$ be an even positive integer.

In Subsection \ref{ss:kr}, we introduce our relative Rapoport--Zink space and state the main theorem (Theorem \ref{th:kr}) on the relation between intersection numbers and derivatives of representation densities. In Subsection \ref{ss:analytic}, we study derivatives of representation densities. In Subsection \ref{ss:bt}, we recall the Bruhat--Tits stratification on the relative Rapoport--Zink space from \cite{Wu} and deduce some consequences. In Subsection \ref{ss:linear}, we prove the linear invariance on the K-theoretic intersection of special divisors, following \cite{How19}. In Subsection \ref{ss:kr_proof1}, we prove Theorem \ref{th:kr} when $r=1$, which is needed for the proof when $r>1$. In Subsection \ref{ss:geometric}, we study intersection numbers. In Subsection \ref{ss:kr_proof}, we prove Theorem \ref{th:kr} for general $r$. In Subsection \ref{ss:comparison}, we compare our relative Rapoport--Zink space to certain (absolute) Rapoport--Zink space assuming $F/\dQ_p$ is unramified.

Here are two preliminary definitions for this section:
\begin{itemize}
  \item A \emph{hermitian $O_E$-module} is a finitely generated free $O_E$-module $L$ together with an $O_F$-bilinear pairing $(\;,\;)_L\colon L\times L\to E$ such that the induced $E$-valued pairing on $L\otimes_{O_F}F$ is a nondegenerate hermitian pairing (with respect to $\tc$). When we say that a hermitian $O_E$-module $L$ is contained in a hermitian $O_E$-module or a hermitian $E$-space $M$, we require that the restriction of the pairing $(\;,\;)_M$ to $L$ coincides with $(\;,\;)_L$.

  \item Let $X$ be an object of an additive category with a notion of dual.
      \begin{itemize}
        \item We say that a morphism $\sigma_X\colon X\to X^\vee$ is a \emph{symmetrization} if $\sigma_X$ is an isomorphism and the composite morphism $X\to X^{\vee\vee}\xrightarrow{\sigma_X^\vee}X^\vee$ coincides with $\sigma_X$.

        \item Given an action $\iota_X\colon O_E\to\End(X)$, we say that a morphism $\lambda_X\colon X\to X^\vee$ is \emph{$\iota_X$-compatible} if $\lambda_X\circ\iota_X(\alpha)=\iota_X(\alpha^\tc)^\vee\circ\lambda_X$ holds for every $\alpha\in O_E$.
      \end{itemize}
\end{itemize}

\subsection{A Kudla--Rapoport type formula}
\label{ss:kr}

We fix an embedding $\varphi_0\colon E\to\dC_p$ and let $\breve{E}$ be the maximal complete unramified extension of $\varphi_0(E)$ in $\dC_p$. We regard $E$ as a subfield of $\breve{E}$ via $\varphi_0$ hence identify the residue field of $\breve{E}$ with an algebraic closure $\ol{k}$ of $k$.

\begin{definition}\label{de:exotic}
Let $S$ be an object of $\Sch_{/O_{\breve{E}}}$. We define a category $\Exo_{(n-1,1)}(S)$ whose objects are triples $(X,\iota_X,\lambda_X)$ in which
\begin{itemize}
  \item $X$ is an $O_F$-divisible group\footnote{An $O_F$-divisible group is also called a strict $O_F$-module.} over $S$ of dimension $n$ and (relative) height $2n$;

  \item $\iota_X\colon O_E\to\End(X)$ is an action of $O_E$ on $X$ satisfying:
    \begin{itemize}
      \item (Kottwitz condition): the characteristic polynomial of $\iota_X(u)$ on the $\sO_S$-module $\Lie(X)$ is $(T-u)^{n-1}(T+u)\in\sO_S[T]$,

      \item (Wedge condition): we have
         \begin{align*}
         \bigwedge^2\(\iota_X(u)-u\res\Lie(X)\)&=0,
         \end{align*}

      \item (Spin condition): for every geometric point $s$ of $S$, the action of $\iota_X(u)$ on $\Lie(X_s)$ is nonzero;
    \end{itemize}

  \item $\lambda_X\colon X\to X^\vee$ is a $\iota_X$-compatible polarization such that $\Ker(\lambda_X)=X[\iota_X(u)]$.
\end{itemize}
A morphism (resp.\ quasi-morphism) from $(X,\iota_X,\lambda_X)$ to $(Y,\iota_Y,\lambda_Y)$ is an $O_E$-linear isomorphism (resp.\ quasi-isogeny) $\rho\colon X\to Y$ of height zero such that $\rho^*\lambda_Y=\lambda_X$.

When $S$ belongs to $\Sch_{/O_{\breve{E}}}^\rv$, we denote by $\Exo_{(n-1,1)}^\rb(S)$ the subcategory of $\Exo_{(n-1,1)}(S)$ consisting of $(X,\iota_X,\lambda_X)$ in which $X$ is supersingular.
\end{definition}

\begin{remark}\label{re:polarization}
Giving a $\iota_X$-compatible polarization $\lambda_X$ of $X$ satisfying $\Ker(\lambda_X)=X[\iota_X(u)]$ is equivalent to giving a $\iota_X$-compatible symmetrization $\sigma_X$ of $X$. In fact, since $\Ker(\lambda_X)=X[\iota_X(u)]$, there is a unique morphism $\sigma_X\colon X\to X^\vee$ satisfying $\lambda_X=\sigma_X\circ\iota_X(u)$, which is in fact an isomorphism, satisfying
\[
\sigma_X^\vee=\iota_X(u^{-1})^\vee\circ\lambda_X^\vee
=-\iota_X(u^{-1})^\vee\circ\lambda_X=-\lambda_X\circ\iota_X(u^{-1,\tc})=\lambda_X\circ\iota_X(u^{-1})=\sigma_X,
\]
and is clearly $\iota_X$-compatible. Conversely, given a $\iota_X$-compatible symmetrization $\sigma_X$ of $X$, we may recover $\lambda_X$ as $\sigma_X\circ\iota_X(u)$. In what follows, we call $\sigma_X$ the symmetrization of $\lambda_X$.
\end{remark}

To define our relative Rapoport--Zink space, we fix an object $(\bbX,\iota_{\bbX},\lambda_{\bbX})\in\Exo_{(n-1,1)}^\rb(\ol{k})$.

\begin{definition}\label{de:rz}
We define a functor $\cN\coloneqq\cN_{(\bbX,\iota_{\bbX},\lambda_{\bbX})}$ on $\Sch_{/O_{\breve{E}}}^\rv$ such that for every object $S$ of $\Sch_{/O_{\breve{E}}}^\rv$, $\cN(S)$ consists of quadruples $(X,\iota_X,\lambda_X;\rho_X)$ in which
\begin{itemize}
  \item $(X,\iota_X,\lambda_X)$ is an object of $\Exo_{(n-1,1)}^\rb(S)$;

  \item $\rho_X$ is a quasi-morphism from $(X,\iota_X,\lambda_X)\times_S(S\otimes_{O_{\breve{E}}}\ol{k})$ to $(\bbX,\iota_{\bbX},\lambda_{\bbX})\otimes_{\ol{k}}(S\otimes_{O_{\breve{E}}}\ol{k})$ in the category $\Exo_{(n-1,1)}^\rb(S\otimes_{O_{\breve{E}}}\ol{k})$.
\end{itemize}
\end{definition}

\begin{hypothesis}\label{hy:rz}
The functor $\cN$ is (pro-)represented by a separated formal scheme over $\Spf O_{\breve{E}}$.
\end{hypothesis}

\begin{remark}\label{re:rz}
When $F$ is unramified over $\dQ_p$, Hypothesis \ref{hy:rz} is known. In fact, by Corollary \ref{co:arz}, $\cN$ is isomorphic to an absolute Rapoport--Zink space $\cN^\Phi$ which is known to be a separated formal scheme over $\Spf O_{\breve{E}}$ by \cite{RZ96}.
\end{remark}

In what follows, we will assume Hypothesis \ref{hy:rz}.

\begin{lem}\label{le:rz}
The functor $\cN$ is a separated formal scheme formally smooth over $\Spf O_{\breve{E}}$ of relative dimension $n-1$. Moreover, $\cN$ has two connected components.
\end{lem}

\begin{proof}
The formal smoothness of $\cN$ follow from the smoothness of its local model, which is \cite{RSZ17}*{Proposition~3.10}; and the dimension also follows. For the last assertion, our moduli functor $\cN$ is the disjoint union of $\cN_{(0,0)}$ and $\cN_{(0,1)}$ from \cite{Wu}*{Section~3.4}, each of which is connected by \cite{Wu}*{Theorem~5.18(2)}.\footnote{The article \cite{Wu} only studied the case $F=\dQ_p$. In fact, except for Hypothesis \ref{hy:rz}, all arguments hence results work for general $F$. This footnote applies to the proof of Proposition \ref{pr:bt} as well.}
\end{proof}

To study special cycles on $\cN$, we fix a triple $(X_0,\iota_{X_0},\lambda_{X_0})$ where
\begin{itemize}
  \item $X_0$ is a supersingular $O_F$-divisible group over $\Spec O_{\breve{E}}$ of dimension $1$ and height $2$;

  \item $\iota_{X_0}\colon O_E\to\End(X_0)$ is an $O_E$-action on $X_0$ such that the induced action on $\Lie(X_0)$ is given by $\varphi_0$;

  \item $\lambda_{X_0}\colon X_0\to X_0^\vee$ is a $\iota_{X_0}$-compatible principal polarization.
\end{itemize}
Note that $\iota_{X_0}$ induces an isomorphism $\iota_{X_0}\colon O_E\xrightarrow{\sim}\End_{O_E}(X_0)$. Put
\[
\bbV\coloneqq\Hom_{O_E}(X_0\otimes_{O_{\breve{E}}}\ol{k},\bbX)\otimes\dQ,
\]
which is a vector space over $E$ of dimension $n$. We have a pairing
\begin{align}\label{eq:rz_special}
(\;,\;)_{\bbV}\colon\bbV\times\bbV\to E
\end{align}
sending $(x,y)\in\bbV^2$ to the composition of quasi-homomorphisms
\[
X_0\xrightarrow{x}\bbX\xrightarrow{\lambda_{\bbX}}\bbX^\vee\xrightarrow{y^\vee}X_0^\vee\xrightarrow{u^{-2}\lambda_{X_0}^{-1}}X_0
\]
as an element in $\End_{O_E}(X_0)\otimes\dQ$ hence in $E$ via $\iota_{X_0}^{-1}$. It is known that $(\;,\;)_{\bbV}$ is a nondegenerate and nonsplit hermitian form on $\bbV$ \cite{RSZ17}*{Lemma~3.5}.\footnote{Readers may notice that we have an extra factor $u^{-2}$ in the definition of the hermitian form. This is because we want to ensure that $\cN(x)$ is nonempty if and only if $(x,x)_{\bbV}\in O_F$.}

\begin{definition}\label{de:rz_special}
For every nonzero element $x\in\bbV$, we define the \emph{special divisor} $\cN(x)$ of $\cN$ to be the maximal closed formal subscheme over which the quasi-homomorphism
\[
\rho_X^{-1}\circ x\colon(X_0\otimes_{O_{\breve{E}}}\ol{k})\otimes_k(S\otimes_{O_{\breve{E}}}\ol{k})\to X\times_S(S\otimes_{O_{\breve{E}}}\ol{k})
\]
lifts (uniquely) to a homomorphism $X_0\otimes_{O_{\breve{E}}}S\to X$.
\end{definition}

\begin{definition}\label{de:int}
For an $O_E$-lattice $\bbL$ of $\bbV$, the Serre intersection multiplicity
\[
\chi\(\sO_{\cN(x_1)}\overset{\dL}\otimes_{\sO_\cN}\cdots\overset{\dL}\otimes_{\sO_\cN}\sO_{\cN(x_n)}\)
\]
does not depend on the choice of a basis $\{x_1,\dots,x_n\}$ of $\bbL$ by Corollary \ref{co:linear2}, which we define to be $\Int(\bbL)$.
\end{definition}

\if false

\begin{remark}\label{re:int}
We have
\begin{enumerate}
  \item It is clear that $\Int(\bbL)$ is an intrinsic invariant of the (isomorphism class of the) hermitian $O_E$-module $\bbL$ such that $\bbL\otimes_{O_F}F$ is a nonsplit even dimensional hermitian space, independent of the choices of the object $(\bbX,\iota_{\bbX},\lambda_{\bbX})$ used to define $\cN$ and the object $(X_0,\iota_{X_0},\lambda_{X_0})$ used to define special cycles.

  \item Since $\cN$ has two connected components, say $\cN^+$ and $\cN^-$, we have $\Int(\bbL)=\Int^+(\bbL)+\Int^-(\bbL)$ in which $\Int^\pm(\bbL)$ denotes the intersection multiplicity on $\cN^\pm$. Then we have
      \[
      \Int^+(\bbL)=\Int^-(\bbL)=\tfrac{1}{2}\Int(\bbL).
      \]
      In fact, choose a normal basis (Remark \ref{re:analytic3}) $\{x_1,\dots,x_n\}$ of $\bbL$. Since $\bbV$ is nonsplit, there exists an anisotropic element in the basis, say $x_n$. Let $\theta$ the unique element in $\rU(\bbV)(F)$ satisfying $\theta(x_i)=1$ for $1\leq i\leq n-1$ and $\theta(x_n)=-x_n$. Then $\theta$ induces an automorphism of $\cN$, preserving $\cN(x_i)$ for $1\leq i\leq n$ but switching $\cN^+$ and $\cN^-$. Thus, we have $\Int^+(\bbL)=\Int^-(\bbL)$.
\end{enumerate}
\end{remark}

\fi

\begin{theorem}\label{th:kr}
Assume Hypothesis \ref{hy:rz}. For every $O_E$-lattice $\bbL$ of $\bbV$, we have
\[
\Int(\bbL)=\partial\Den(\bbL),
\]
where $\partial\Den(\bbL)$ is defined in Definition \ref{de:density2}.
\end{theorem}

By Remark \ref{re:rz}, this theorem is unconditional if $F$ is unramified over $\dQ_p$.

The strategy of proving this theorem described in Subsection \ref{ss:ingredients} motivates the following definition, which will be frequently used in the rest of Section \ref{ss:2}.

\begin{definition}\label{de:analytic}
We define $\flat(\bbV)$ to be the set of hermitian $O_E$-modules contained in $\bbV$ of rank $n-1$. In what follows, for $L^\flat\in\flat(\bbV)$, we put $V_{L^\flat}\coloneqq L^\flat\otimes_{O_F}F$ and write $V_{L^\flat}^\perp$ for the orthogonal complement of $V_{L^\flat}$ in $\bbV$.
\end{definition}

\begin{remark}\label{re:exotic}
Let $S$ be an object of $\Sch_{/O_{\breve{E}}}$. We have another category $\Exo_{(n,0)}(S)$ whose objects are triples $(X,\iota_X,\lambda_X)$ in which
\begin{itemize}
  \item $X$ is an $O_F$-divisible group over $S$ of dimension $n$ and (relative) height $2n$;

  \item $\iota_X\colon O_E\to\End(X)$ is an action of $O_E$ on $X$ such that $\iota_X(u)-u$ annihilates $\Lie(X)$;

  \item $\lambda_X\colon X\to X^\vee$ is a $\iota_X$-compatible polarization such that $\Ker(\lambda_X)=X[\iota_X(u)]$.
\end{itemize}
Morphisms are defined similarly as in Definition \ref{de:exotic}. The category $\Exo_{(n,0)}(S)$ is a connected groupoid.
\end{remark}

For later use, we fix a nontrivial additive character $\psi_F\colon F\to\dC^\times$ of conductor $O_F$. For a locally constant compactly supported function $\phi$ on a hermitian space $V$ over $E$, its Fourier transform $\widehat\phi$ is defined by
\[
\widehat\phi(x)=\int_V\phi(y)\psi_F(\Tr_{E/F}(x,y)_V)\rd y
\]
where $\r{d}y$ is the self-dual Haar measure on $V$.

\subsection{Fourier transform of analytic side}
\label{ss:analytic}

In this subsection, we study local densities of hermitian lattices. We first introduce some notion about $O_E$-lattices in hermitian spaces.

\begin{definition}\label{de:analytic1}
Let $V$ be a hermitian space over $E$ of dimension $m$, equipped with the hermitian form $(\;,\;)_V$.
\begin{enumerate}
  \item For a subset $X$ of $V$,
      \begin{itemize}
        \item we denote by $X^\integ$ the subset $\{x\in X\res (x,x)_V\in O_F\}$;

        \item we denote by $\langle X\rangle$ the $O_E$-submodule of $V$ generated by $X$; when $X=\{x,\dots\}$ is explicitly presented, we simply write $\langle x,\dots\rangle$ instead of $\langle\{x,\dots\}\rangle$.
      \end{itemize}

  \item For an $O_E$-lattice $L$ of $V$, we put
      \begin{align*}
      L^\vee
      &\coloneqq\{x\in V\res\Tr_{E/F}(x,y)_V\in O_F\text{ for every $y\in L$}\} \\
      &=\{x\in V\res (x,y)_V\in u^{-1}O_E\text{ for every $y\in L$}\}.
      \end{align*}
      We say that $L$ is
      \begin{itemize}
        \item \emph{integral} if $L\subseteq L^\vee$;

        \item \emph{vertex} if it is integral such that $L^\vee/L$ is annihilated by $u$; and

        \item \emph{self-dual} if $L=L^\vee$.
      \end{itemize}

  \item For an integral $O_E$-lattice $L$ of $V$, we define
      \begin{itemize}
        \item the \emph{fundamental invariants} of $L$ unique integers $0\leq a_1\leq\cdots\leq a_m$ such that $L^\vee/L\simeq O_E/(u^{a_1})\oplus\cdots\oplus O_E/(u^{a_m})$ as $O_E$-modules;

        \item the \emph{type} $t(L)$ of $L$ to be the number of nonzero elements in its fundamental invariants; and

        \item the \emph{valuation} of $L$ to be $\val(L)\coloneqq\sum_{i=1}^ma_i$; when $L$ is generated by a single element $x$, we simply write $\val(x)$ instead of $\val(\langle x\rangle)$.
      \end{itemize}
\end{enumerate}
The above notation and definitions make sense without specifying $V$, namely, they apply to hermitian $O_E$-modules.
\end{definition}

\begin{remark}\label{re:analytic2}
For an integral hermitian $O_E$-module $L$ of rank $m$ with fundamental invariants $(a_1,\dots,a_m)$, we have
\begin{enumerate}
  \item $L$ is vertex if and only if $a_m\leq 1$ and self-dual if and only if $a_m=0$;

  \item $t(L)$ and $\val(L)$ must have the same parity with $m$.
\end{enumerate}
\end{remark}

\begin{remark}\label{re:analytic3}
For a hermitian $O_E$-module $L$, we say that a basis $\{e_1,\dots,e_m\}$ of $L$ is a \emph{normal basis} if its moment matrix $T=((e_i,e_j)_L)_{i,j=1}^m$ is conjugate to
\[
\begin{pmatrix}
\beta_1 u^{2b_1}
\end{pmatrix}
\oplus\cdots\oplus
\begin{pmatrix}
\beta_s u^{2b_s}
\end{pmatrix}
\oplus
\begin{pmatrix}
0 & u^{2c_1-1} \\
-u^{2c_1-1} & 0
\end{pmatrix}
\oplus\cdots\oplus
\begin{pmatrix}
0 & u^{2c_t-1} \\
-u^{2c_t-1} & 0
\end{pmatrix}
\]
by a permutation matrix, for some $\beta_1,\dots,\beta_s\in O_F^\times$ and $b_1,\dots,b_s,c_1,\dots,c_t\in\dZ$. We have
\begin{enumerate}
  \item normal basis exists;

  \item the invariants $s,t$ and $b_1,\dots,b_s,c_1,\dots,c_t$ depend only on $L$;

  \item when $L$ is integral, the fundamental invariants of $L$ are the unique nondecreasing rearrangement of $(2b_1+1,\dots,2b_s+1,2c_1,2c_1,\dots,2c_t,2c_t)$.
\end{enumerate}
\end{remark}

\begin{definition}\label{de:density}
Let $M$ and $L$ be two hermitian $O_E$-modules. We denote by $\Herm_{L,M}$ the scheme of hermitian $O_E$-module homomorphisms from $L$ to $M$, which is a scheme of finite type over $O_F$. We define the \emph{local density} to be
\[
\Den(M,L)\coloneqq\lim_{N\to+\infty}\frac{\left|\Herm_{L,M}(O_F/(u^{2N}))\right|}{q^{N\cdot d_{L,M}}}
\]
where $d_{L,M}$ is the dimension of $\Herm_{L,M}\otimes_{O_F}F$.
\end{definition}

Denote by $H$ the standard hyperbolic hermitian $O_E$-module (of rank 2) given by the matrix $\(\begin{smallmatrix} 0 & u^{-1}\\ -u^{-1} & 0\end{smallmatrix}\)$. For an integer $s\geq 0$, put $H_s\coloneqq H^{\oplus s}$. Then $H_s$ is a self-dual hermitian $O_E$-module of rank $2s$. The following lemma is a variant of a result of Cho--Yamauchi \cite{CY20} when $E/F$ is ramified.

\begin{lem}\label{le:analytic1}
Let $L$ be a hermitian $O_E$-module of rank $m$. Then we have
\[
\Den(H_s,L)=\sum_{L\subseteq L'\subseteq L'^\vee}|L'/L|^{m-2s}
\prod_{s-\frac{m+t(L')}{2}<i\leq s}(1-q^{-2i})
\]
for every integer $s\geq m$, where the sum is taken over integral $O_E$-lattices of $L\otimes_{O_F}F$ containing $L$.
\end{lem}

\begin{proof}
Put $V\coloneqq L\otimes_{O_F}F$. For an integral $O_E$-lattice $L'$ of $V$, we equip the $k$-vector space $L'_k\coloneqq L'\otimes_{O_E}O_E/(u)$ with a $k$-valued pairing $\langle\;,\;\rangle_{L'_k}$ by the formula
\[
\langle x,y\rangle_{L'_k}\coloneqq u\cdot (x^\sharp,y^\sharp)_V\mod(u)
\]
where $x^\sharp$ and $y^\sharp$ are arbitrary lifts of $x$ and $y$, respectively. Then $L'_k$ becomes a symplectic space over $k$ of dimension $m$ whose radical has dimension $t(L')$. Similarly, we have $H_{s,k}$, which is a nondegenerate symplectic space over $k$ of dimension $2s$. We denote by $\Isom_{L'_k,H_{s,k}}$ the $k$-scheme of isometries from $L'_k$ to $H_{s,k}$.

By the same argument in \cite{CY20}*{Section~3.3}, we have
\[
\Den(H_s,L)=q^{-m(4s-m+1)/2}\cdot\sum_{L\subseteq L'\subseteq L'^\vee}|L'/L|^{m-2s}|\Isom_{L'_k,H_{s,k}}(k)|.
\]
Thus, it remains to show that
\begin{align}\label{eq:analytic1}
|\Isom_{L'_k,H_{s,k}}(k)|=q^{m(4s-m+1)/2}\prod_{s-\frac{m+t(L')}{2}<i\leq s}(1-q^{-2i}).
\end{align}

We fix a decomposition $L'_k=L_0\oplus L_1$ in which $L_0$ is nondegenerate and $L_1$ is the radical of $L'_k$. We have a morphism $\pi\colon\Isom_{L'_k,H_{s,k}}\to\Isom_{L_0,H_{s,k}}$ given by restriction, such that for every element $f\in\Isom_{L_0,H_{s,k}}(k)$, the fiber $\pi^{-1}f$ is isomorphic to $\Isom_{L_1,\IM(f)^\perp}$. As $\IM(f)^\perp$ is isomorphic to $H_{s-\frac{m-t(L')}{2},k}$, it suffices to show \eqref{eq:analytic1} in the two extremal cases: $t(L')=0$ and $t(L')=m$.

Suppose that $t(L')=0$, that is, $L'_k$ is nondegenerate. Note that $\Sp(H_{s,k})$ acts on $\Isom_{L'_k,H_{s,k}}$ transitively, with the stabilizer isomorphic to $\Sp(H_{s-\frac{m}{2},k})$. Thus, we have
\begin{align*}
|\Isom_{L'_k,H_{s,k}}(k)|&=\frac{|\Sp(H_{s,k})(k)|}{|\Sp(H_{s-\frac{m}{2},k})(k)|}
=\frac{q^{s^2}\prod_{i=1}^s(q^{2i}-1)}{q^{\(s-\frac{m}{2}\)^2}\prod_{i=1}^{s-\frac{m}{2}}(q^{2i}-1)} \\
&=q^{m(4s-m+1)/2}\prod_{s-\frac{m}{2}<i\leq s}(1-q^{-2i}),
\end{align*}
which confirms \eqref{eq:analytic1}.

Suppose that $t(L')=m$, that is, $L'_k$ is isotropic. Note that $\Sp(H_{s,k})$ acts on $\Isom_{L'_k,H_{s,k}}$ transitively, with the stabilizer $Q$ fits into a short exact sequence
\[
1 \to U_m \to Q \to \Sp(H_{s-m,k}) \to 1
\]
in which $U_m$ is a unipotent subgroup of $\Sp(H_{s,k})$ of Levi type $\GL_{m,k}\times\Sp(H_{s-m,k})$. Thus, we have
\begin{align*}
|\Isom_{L'_k,H_{s,k}}(k)|&=\frac{|\Sp(H_{s,k})(k)|}{|U_m(k)|\cdot|\Sp(H_{s-m,k})(k)|}
=\frac{q^{s^2}\prod_{i=1}^s(q^{2i}-1)}{q^{m(2s-2m)+\frac{m(m+1)}{2}}\cdot q^{\(s-m\)^2}\prod_{i=1}^{s-m}(q^{2i}-1)} \\
&=q^{m(4s-m+1)/2}\prod_{s-m<i\leq s}(1-q^{-2i}),
\end{align*}
which confirms \eqref{eq:analytic1}.

Thus, \eqref{eq:analytic1} is proved and the lemma follows.
\end{proof}

\if false
\begin{remark}\label{re:analytic1}
By a similar argument for Lemma \ref{le:analytic1}, we have
\[
\Den(H_{r+s},H_r)=\prod_{i=s+1}^{r+s}(1-q^{-2i})
\]
for every integer $s\geq 0$.
\end{remark}
\fi

Now we fix a hermitian space $\bbV$ over $E$ of dimension $n=2r$ that is \emph{nonsplit}.

\begin{definition}\label{de:density2}
For an $O_E$-lattice $\bbL$ of $\bbV$, define the \emph{(normalized) local Siegel series} of $\bbL$ to be the polynomial $\Den(X,\bbL)\in\dZ[X]$, which exists by Lemma \ref{le:analytic2} below, such that for every integer $s\geq 0$,
\[
\Den(q^{-s},\bbL)=\frac{\Den(H_{r+s},\bbL)}{\prod_{i=s+1}^{r+s}(1-q^{-2i})},
\]
where $\Den$ is defined in Definition \ref{de:density}. We then put
\[
\partial\Den(\bbL)\coloneqq-\left.\frac{\rd}{\rd X}\right|_{X=1}\Den(X,\bbL).
\]
\end{definition}

\begin{remark}
Since $\bbV$ is nonsplit, we have $\Den(1,\bbL)=\Den(H_r,\bbL)=0$.
\end{remark}

\begin{remark}\label{re:whittaker}
Let $\bbL$ be an $O_E$-lattice of $\bbV$. Let $T\in\GL_n(E)$ be a representing matrix of $\bbL$, and consider the $T$-th Whittaker function $W_T(s,1_{4r},\CF_{H_r^{2r}})$ of the Schwartz function $\CF_{H_r^{2r}}$ at the identity element $1_{4r}$. By \cite{KR14}*{Proposition~10.1},\footnote{In \cite{KR14}*{Proposition~10.1} and its proof, the lattice $L_{r,r}$ should be replaced by $H_r$.} we have
\[
W_T(s,1_{4r},\CF_{H_r^{2r}})=\Den(H_{r+s},\bbL)
\]
for every integer $s\geq 0$. Thus, we obtain
\[
\log q\cdot\partial\Den(\bbL)=\frac{W'_T(0,1_{4r},\CF_{H_r^{2r}})}{\prod_{i=1}^r(1-q^{-2i})}
\]
by Definition \ref{de:density2}.
\end{remark}

\begin{lem}\label{le:analytic2}
For every $O_E$-lattice $\bbL$ of $\bbV$, we have
\begin{align}\label{eq:analytic2}
\Den(X,\bbL)=\sum_{\bbL\subseteq L\subseteq L^\vee}X^{2\length_{O_E}(L/\bbL)}\prod_{i=0}^{\frac{t(L)}{2}-1}(1-q^{2i}X^2),
\end{align}
and
\begin{align}\label{eq:analytic3}
\partial\Den(\bbL)=2\sum_{\bbL\subseteq L\subseteq L^\vee}\prod_{i=1}^{\frac{t(L)}{2}-1}(1-q^{2i}),
\end{align}
where both sums are taken over integral $O_E$-lattices of $\bbV$ containing $\bbL$.\footnote{In \eqref{eq:analytic3}, when $t(L)=2$, we regard the empty product $\prod_{i=1}^{\frac{t(L)}{2}-1}(1-q^{2i})$ as $1$.}
\end{lem}

\begin{proof}
The identity \eqref{eq:analytic2} is a direct consequence of Lemma \ref{le:analytic1} and Definition \ref{de:density2}. The identity \eqref{eq:analytic3} is a consequence of \eqref{eq:analytic2}.
\end{proof}

\begin{definition}\label{de:density1}
Let $L^\flat$ be an element of $\flat(\bbV)$ (Definition \ref{de:analytic}). For $x\in\bbV\setminus V_{L^\flat}$, we put
\begin{align*}
\partial\Den_{L^\flat}(x)&\coloneqq\partial\Den(L^\flat+\langle x\rangle), \\
\partial\Den_{L^\flat}^\rh(x)&\coloneqq 2\sum_{\substack{L^\flat\subseteq L\subseteq L^\vee \\ t(L\cap V_{L^\flat})=1}}\CF_{L}(x), \\
\partial\Den_{L^\flat}^\rv(x)&\coloneqq \partial\Den_{L^\flat}(x)-\partial\Den_{L^\flat}^\rh(x).
\end{align*}
Here in the second formula, $L$ in the summation is an $O_E$-lattice of $\bbV$.
\end{definition}

\begin{remark}\label{re:analytic}
We have
\begin{enumerate}
  \item The summation in $\partial\Den_{L^\flat}^\rh(x)$ equals twice the number of integral $O_E$-lattices $L$ of $\bbV$ that contains $L^\flat+\langle x\rangle$ and such that $t(L\cap V_{L^\flat})=1$.

  \item There exists a compact subset $C_{L^\flat}$ of $\bbV$ such that $\partial\Den_{L^\flat}$, $\partial\Den_{L^\flat}^\rh$, and $\partial\Den_{L^\flat}^\rv$ vanish outside $C_{L^\flat}$ and are locally constant functions on $C_{L^\flat}\setminus V_{L^\flat}$.

  \item For an integral $O_E$-lattice $L$ of $\bbV$, if $t(L\cap V_{L^\flat})=1$, then $t(L)=2$ by Lemma \ref{le:analytic7}(1) below and the fact that $\bbV$ is nonsplit.

  \item By (3) and Lemma \ref{le:analytic2}, we have
      \[
      \partial\Den_{L^\flat}^\rv(x)=2\sum_{\substack{L^\flat\subseteq L\subseteq L^\vee \\ t(L\cap V_{L^\flat})>1}}\(\prod_{i=1}^{\frac{t(L)}{2}-1}(1-q^{2i})\)\CF_L(x)
      \]
      for $x\in\bbV\setminus V_{L^\flat}$.
\end{enumerate}
\end{remark}

The following is our main result of this subsection.

\begin{proposition}\label{pr:analytic}
Let $L^\flat$ be an element of $\flat(\bbV)$. Then $\partial\Den_{L^\flat}^\rv$ extends (uniquely) to a (compactly supported) locally constant function on $\bbV$, which we still denote by $\partial\Den_{L^\flat}^\rv$. Moreover, the support of $\widehat{\partial\Den_{L^\flat}^\rv}$ is contained in $\bbV^\integ$ (Definition \ref{de:analytic1}).
\end{proposition}

We need some lemma for preparation.

\begin{lem}\label{le:analytic7}
Let $L$ be an integral hermitian $O_E$-module of with fundamental invariants $(a_1,\dots,a_m)$.
\begin{enumerate}
  \item If $T=((e_i,e_j)_L)_{i,j=1}^m$ is the moment matrix of an arbitrary basis $\{e_1,\dots,e_m\}$ of $L$, then for every $1\leq i\leq m$, $a_1+\cdots+a_i-i$ equals the minimal $E$-valuation of the determinant of all $i$-by-$i$ minors of $T$.

  \item If $L=L'+\langle x\rangle$ for some (integral) hermitian $O_E$-module $L'$ contained in $L$ of rank $m-1$, then we have
      \[
      t(L)=
      \begin{dcases}
      t(L')+1, &\text{if $x'\in uL^{\prime\vee}+L'$,} \\
      t(L')-1, &\text{otherwise,}
      \end{dcases}
      \]
      where $x'$ is the unique element in $L^{\prime\vee}$ such that $(x',y)_L=(x,y)_L$ for every $y\in L'$.
\end{enumerate}
\end{lem}

\begin{proof}
Part (1) is simply the well-known method of computing the Smith normal form of $uT$ (over $O_E$) using ideals generated by determinants of minors. For (2), take a normal basis $\{x_1,\dots,x_{m-1}\}$ of $L$ (Remark \ref{re:analytic3}) such that $\langle x_1,\dots,x_{m-1-t(L')}\rangle$ is self-dual. Applying (1) to the basis $\{x_1,\dots,x_{m-1},x\}$ of $L$, we know that $t(L)=t(L')+1$ if $(x_i,x)_L\in O_E$ for every $m-t(L')\leq i\leq {m-1}$; otherwise, we have $t(L)=t(L')-1$. In particular, (2) follows.
\end{proof}

In the rest of this subsection, in order to shorten formulae, we put
\[
\mu(t)\coloneqq\prod_{i=1}^{\frac{t}{2}-1}(1-q^{2i})
\]
for every positive even integer $t$.

\begin{lem}\label{le:analytic8}
Take $L^\flat\in\flat(\bbV)$ that is integral. For every compact subset $X$ of $\bbV$ not contained in $V_{L^\flat}$, we denote by $\delta_X$ the maximal integer such that the image of $X$ under the projection map $\bbV\to V_{L^\flat}^\perp$ induced by the orthogonal decomposition $\bbV=V_{L^\flat}\oplus V_{L^\flat}^\perp$ is contained in $u^{\delta_X}(V_{L^\flat}^\perp)^\integ$. We denote by $\fL$ the set of $O_E$-lattices of $\bbV$ containing $L^\flat$, and by $\fE$ the set of triples $(L^{\flat\prime},\delta,\varepsilon)$ in which $L^{\flat\prime}$ is an $O_E$-lattice of $V_{L^\flat}$ containing $L^\flat$, $\delta\in\dZ$, and $\varepsilon\colon u^\delta(V_{L^\flat}^\perp)^\integ\to L^{\flat\prime}\otimes_{O_F}F/O_F$ is an $O_E$-linear map.
\begin{enumerate}
  \item The map $\fL\to\fE$ sending $L$ to the triple $(L\cap V_{L^\flat},\delta_L,\varepsilon_L)$ is a bijection, where $\varepsilon_L$ is the is the extension map $u^{\delta_X}(V_{L^\flat}^\perp)^\integ\to(L\cap V_{L^\flat})\otimes_{O_F}F/O_F$ induced by the short exact sequence
      \[
      0 \to L\cap V_{L^\flat} \to L \to u^{\delta_X}(V_{L^\flat}^\perp)^\integ \to 0.
      \]
      Moreover, $L$ is integral if and only if the following hold:
      \begin{itemize}
        \item $L\cap V_{L^\flat}$ is integral;

        \item the image of $\varepsilon$ is contained in $(L\cap V_{L^\flat})^\vee/(L\cap V_{L^\flat})$;

        \item $\varepsilon_L(x)+x\subseteq\bbV^\integ$ for every $x\in u^{\delta_X}(V_{L^\flat}^\perp)^\integ$.\footnote{For $(L^{\flat\prime},\delta,\varepsilon)\in\fE$, we regard $\varepsilon(x)+x$ as an $L^{\flat\prime}$-coset in $\bbV$ as long as we write $\varepsilon(x)+x\subseteq\Omega$ for a subset $\Omega$ of $\bbV$.}
      \end{itemize}

  \item For $L\in\fL$ that is integral and corresponds to $(L^{\flat\prime},\delta,\varepsilon)\in\fE$, we have
      \[
      t(L)=
      \begin{dcases}
      t(L^{\flat\prime})+1, &\text{if the image of $\varepsilon$ is contained in $(u(L^{\flat\prime})^\vee+L^{\flat\prime})/L^{\flat\prime}$,} \\
      t(L^{\flat\prime})-1, &\text{otherwise.}
      \end{dcases}
      \]

  \item For every fixed integral $O_E$-lattice $L^{\flat\prime}$ of $V_{L^\flat}$ containing $L^\flat$, the sum
      \[
      \sum_{\substack{L\subseteq L^\vee \\ L\cap V_{L^\flat}=L^{\flat\prime}}}q^{-\delta_L}|\mu(t(L))|
      \]
      is convergent; and if $t(L^{\flat\prime})>1$, then we have
      \[
      \sum_{\substack{L\subseteq L^\vee \\ L\cap V_{L^\flat}=L^{\flat\prime}\\ z\in L^\vee}}q^{-\delta_L}\mu(t(L))=0
      \]
      for every $z\in\bbV\setminus\bbV^\integ$.

  \item For every fixed integral $O_E$-lattice $L^{\flat\prime}$ of $V_{L^\flat}$ containing $L^\flat$ with $t(L^{\flat\prime})>1$, we have
      \[
      \sum_{\substack{L\subseteq L^\vee \\ L\cap V_{L^\flat}=L^{\flat\prime}\\ \delta_L=0}}\mu(t(L))=0.
      \]
\end{enumerate}
\end{lem}

\begin{proof}
For (1), the inverse map $\fE\to\fL$ is the one that sends $(L^{\flat\prime},\delta,\varepsilon)$ to the $O_E$-lattice $L$ generated by $L^{\flat\prime}$ and $\varepsilon_L(x)+x$ for every $x\in u^{\delta_X}(V_{L^\flat}^\perp)^\integ$. The rest of (1) is straightforward.

Part (2) is simply Lemma \ref{le:analytic7}(2).

Part (4) follows by applying (3) to generators $z$ of $O_E$-modules $u^{-1}(V_{L^\flat}^\perp)^\integ$ and $u^{-2}(V_{L^\flat}^\perp)^\integ$ and then taking the difference.

Now we prove (3), which is the most difficult one. For every $x\in\bbV$, we denote by $x'\in V_{L^\flat}$ the first component of $x$ with respect to the orthogonal decomposition $\bbV=V_{L^\flat}\oplus V_{L^\flat}^\perp$. Put
\begin{align*}
\Omega\coloneqq\{x\in\bbV^\integ\res x'\in(L^{\flat\prime})^\vee\},\qquad
\Omega^\circ \coloneqq\{x\in\bbV^\integ\res x'\in u(L^{\flat\prime})^\vee+L^{\flat\prime}\}.
\end{align*}
Note that both $\Omega$ and $\Omega^\circ$ are open compact subsets of $\bbV$ stable under the translation by $L^{\prime\flat}$. For an element $L\in\fL$ corresponding to $(L^{\flat\prime},\delta,\varepsilon)\in\fE$ from (1), $L$ is integral if and only $\varepsilon(x)+x\subseteq\Omega$ for every $x\in u^\delta(V_{L^\flat}^\perp)^\integ$. By (2), for such $L$,
\begin{align*}
t(L)=
\begin{dcases}
t(L^{\flat\prime})+1, &\text{if $\varepsilon(x)+x\subseteq\Omega^\circ$ for every
$x\in u^\delta(V_{L^\flat}^\perp)^\integ\setminus u^{\delta+1}(V_{L^\flat}^\perp)^\integ$,} \\
t(L^{\flat\prime})-1, &\text{if $\varepsilon(x)+x\subseteq\Omega\setminus\Omega^\circ$ for every
$x\in u^\delta(V_{L^\flat}^\perp)^\integ\setminus u^{\delta+1}(V_{L^\flat}^\perp)^\integ$.}
\end{dcases}
\end{align*}
Thus, we may replace the term corresponding to $L$ in the summation in (3) by an integration over the region $\bigcup_{x\in u^\delta(V_{L^\flat}^\perp)^\integ\setminus u^{\delta+1}(V_{L^\flat}^\perp)^\integ}(\varepsilon(x)+x)$ of $\Omega$. It follows that
\[
\sum_{\substack{L\subseteq L^\vee \\ L\cap V_{L^\flat}=L^{\flat\prime}}}q^{-\delta_L}|\mu(t(L))|
=\frac{1}{C}\(\int_{\Omega^\circ\setminus V_{L^\flat}}|\mu(t(L^{\flat\prime})+1)|\rd x+\int_{\Omega\setminus(\Omega^\circ\cup V_{L^\flat})}|\mu(t(L^{\flat\prime})-1)|\rd x\)
\]
which is convergent, where
\[
C=\vol(L^{\flat\prime})\cdot\vol((V_{L^\flat}^\perp)^\integ\setminus u(V_{L^\flat}^\perp)^\integ).
\]

Now we take an element $z\in\bbV\setminus\bbV^\integ$. We may assume $z'\in(L^{\flat\prime})^\vee$ since otherwise the summation in (3) is empty. Put
\[
\Omega_z\coloneqq\{x\in\Omega\res(x,z)_{\bbV}\in u^{-1}O_E\},\qquad
\Omega^\circ_z\coloneqq\{x\in\Omega^\circ\res(x,z)_{\bbV}\in u^{-1}O_E\},
\]
both stable under the translation by $L^{\prime\flat}$ since $z'\in(L^{\flat\prime})^\vee$. Similarly, we have
\begin{align*}
\sum_{\substack{L\subseteq L^\vee \\ L\cap V_{L^\flat}=L^{\flat\prime}\\ z\in L^\vee}}q^{-\delta_L}\mu(t(L))
&=\frac{1}{C}\(\int_{\Omega_z^\circ\setminus V_{L^\flat}}\mu(t(L^{\flat\prime})+1)\rd x+\int_{\Omega_z\setminus(\Omega_z^\circ\cup V_{L^\flat})}\mu(t(L^{\flat\prime})-1)\rd x\) \\
&=\frac{\mu(t(L^{\flat\prime})-1)}{C}\(\vol(\Omega_z\setminus\Omega_z^\circ)+\(1-q^{t(L^{\flat\prime})-1}\)\vol(\Omega_z^\circ)\)  \\
&=\frac{\mu(t(L^{\flat\prime})-1)}{C}\(\vol(\Omega_z)-q^{t(L^{\flat\prime})-1}\vol(\Omega_z^\circ)\),
\end{align*}
where we have used $t(L^{\flat\prime})>1$ in the second equality. Thus, it remains to show that
\begin{align}\label{eq:analytic13}
\vol(\Omega_z)=q^{t(L^{\flat\prime})-1}\vol(\Omega_z^\circ).
\end{align}

We fix an orthogonal decomposition $L^{\flat\prime}=L_0\oplus L_1$ in which $L_0$ is self-dual and $L_1$ is of both rank and type $t(L^{\flat\prime})$. Since both $\Omega_z$ and $\Omega_z^\circ$ depend only on the coset $z+L^{\flat\prime}$, we may assume $z'\in L_1^\vee$ and anisotropic. Let $V_2\subseteq\bbV$ be the orthogonal complement of $L_0+\langle z\rangle$. We claim
\begin{itemize}
  \item[($*$)] There exists an integral $O_E$-lattice $L_2$ of $V_2$ of of type $t(L^{\flat\prime})$ such that
     \begin{align}\label{eq:analytic14}
     (u^i L_2^\vee)^\integ=\{x\in V_2^\integ\res x'\in u^iL_1^\vee\}
     \end{align}
     for $i=0,1$.
\end{itemize}
Assuming ($*$), by construction, we have
\[
\{x\in\bbV\res(x,z)_{\bbV}\in u^{-1}O_E\}=L_0\otimes_{O_F}F\oplus\langle z\rangle^\vee\oplus V_2.
\]
Now we use the condition $z\not\in\bbV^\integ$, which implies that $\langle z\rangle^\vee\subseteq u\langle z\rangle\cap\bbV^\integ$. Combining with \eqref{eq:analytic14}, we obtain
\begin{align*}
\Omega_z=L_0\times\langle z\rangle^\vee\times(L_2^\vee)^\integ,\qquad
\Omega_z^\circ=L_0\times\langle z\rangle^\vee\times(uL_2^\vee)^\integ.
\end{align*}
Thus, \eqref{eq:analytic13} follows from Lemma \ref{le:analytic6} below. Part (3) is proved.

Now we show ($*$). There are two cases.

First, we assume $z\neq z'$, that is, $z\not\in V_{L^\flat}$. Let $L_2$ be the unique $O_E$-lattice of $V_2$ satisfying
\begin{align}\label{eq:analytic15}
L_2^\vee=\{x\in V_2\res x'\in L_1^\vee\}.
\end{align}
Then \eqref{eq:analytic14} clearly holds. Thus, it remains to show that $L_2$ is integral of type $t(L^{\flat\prime})$. Put $w\coloneqq z-z'\in V_{L^\flat}^\perp$ which is nonzero hence anisotropic. Then
\[
\bar{z}\coloneqq z'-\frac{(z',z')_{\bbV}}{(w,w)_{\bbV}}w
\]
is the unique element in $V_2$ such that $\bar{z}'=z'$. To compute $L_2$, we write
\[
L_1^\vee=M+\langle y+\alpha z'\rangle
\]
for some $y\in V_{L^\flat}\cap V_2$ and $\alpha\in E\setminus uO_E$, where $M\coloneqq L_1^\vee\cap V_2$. Then
\[
M^\dag\coloneqq L_1\cap V_2=\{x\in M^\vee\res (x,y)_{\bbV}\in u^{-1}O_E\}.
\]
Since $M^\vee/M^\dag$ is isomorphic to an $O_E$-submodule of $E/u^{-1}O_E$, we may take an element $y^\dag\in M^\vee$ that generates $M^\vee/M^\dag$. Then we have
\[
L_1=M^\dag+\langle y^\dag+\alpha^\dag z'\rangle
\]
for some $\alpha^\dag\in E^\times$ such that $(y^\dag,y)_{\bbV}+\alpha^\dag\alpha^\tc(z',z')_{\bbV}\in u^{-1}O_E$. Now by \eqref{eq:analytic15}, we have
\[
L_2^\vee=M+\langle y+\alpha\bar{z}\rangle.
\]
By the same argument, we have
\[
L_2=M^\dag+\langle y^\dag+\alpha^\dag\rho\bar{z}\rangle,
\]
where
\[
\rho\coloneqq\frac{(z',z')_{\bbV}}{(\bar{z},\bar{z})_{\bbV}}.
\]
By Lemma \ref{le:analytic7}(2), we have $t(L_2)=t(L_1)=t(L^{\flat\prime})$ as long as $L_2$ is integral. Thus, it suffices to show that $y^\dag+\alpha^\dag\rho\bar{z}\in\bbV^\integ$. We compute
\begin{align*}
&\quad(y^\dag+\alpha^\dag\rho\bar{z},y^\dag+\alpha^\dag\rho\bar{z})_{\bbV}-(y^\dag+\alpha^\dag z',y^\dag+\alpha^\dag z')_{\bbV}
=(\alpha^\dag\rho\bar{z},\alpha^\dag\rho\bar{z})_{\bbV}-(\alpha^\dag z',\alpha^\dag z')_{\bbV} \\
&=\Nm_{E/F}(\alpha^\dag)\(\frac{(z',z')_{\bbV}^2}{(\bar{z},\bar{z})_{\bbV}}-(z',z')_{\bbV}\)
=\Nm_{E/F}(\alpha^\dag)(z',z')_{\bbV}\(\frac{(z',z')_{\bbV}}{(z',z')_{\bbV}+\frac{(z',z')_{\bbV}^2}{(w,w)_{\bbV}}}-1\) \\
&=\Nm_{E/F}(\alpha^\dag)(z',z')_{\bbV}\(\frac{(w,w)_{\bbV}}{(z',z')_{\bbV}+(w,w)_{\bbV}}-1\)
=\frac{-(\alpha^\dag)^\tc}{\alpha^\dag}\frac{(\alpha^\dag z',z')_{\bbV}^2}{(z,z)_{\bbV}}.
\end{align*}
As $z'\in L_1^\vee$, we have $(\alpha^\dag z',z')_{\bbV}\in u^{-1}O_E$. As $z\not\in\bbV^\integ$, we have $(z,z)_{\bbV}\not\in u^{-1}O_E$. Together, we have $\frac{(\alpha^\dag z',z')_{\bbV}^2}{(z,z)_{\bbV}}\in O_F$. Thus, $y^\dag+\alpha^\dag\rho\bar{z}\in\bbV^\integ$ as $y^\dag+\alpha^\dag z'\in\bbV^\integ$, hence $L_2$ meets the requirement in ($*$).

Second, we assume $z=z'$, that is, $z\in V_{L^\flat}$. Take $L_2=(L_1^\vee\cap V_2)^\vee\oplus u^\delta(V_{L^\flat}^\perp)^\integ$ for some integer $\delta\geq 0$ determined later. We show that $(L_1^\vee\cap V_2)^\vee$ is an integral hermitian $O_E$-module of type $t(L^{\flat\prime})-1$. As in the previous case, we write
\[
L_1^\vee=M+\langle y+\alpha z'\rangle
\]
for some $y\in V_{L^\flat}\cap V_2$ and $\alpha\in E\setminus uO_E$, where $M\coloneqq L_1^\vee\cap V_2$. Then
\[
L_1=M^\dag+\langle y^\dag+\alpha^\dag z'\rangle
\]
so that $M^\vee$ is generated by $M^\dag$ and $y^\dag$. As $L_1$ is of type $t(L^{\flat\prime})$ which is its rank, we have $L_1\subseteq uL_1^\vee$, that is,
\[
M^\dag+\langle y^\dag+\alpha^\dag z'\rangle\subseteq uM+u\langle y+\alpha z'\rangle
\]
hence $M^\dag\subseteq uM$. As $z'\in L_1^\vee$, we have $(\alpha z',z')_{\bbV}\in u^{-1}O_E$. As $z'=z\not\in\bbV^\integ$, we have $(z',z')_{\bbV}\not\in u^{-1}O_E$ hence $\alpha^\dag\in uO_E$. Again as $z'\in L_1^\vee$, we have $\alpha^\dag z'\in uL_1^\vee$ hence $y^\dag\in uL_1^\vee\cap V_2= uM$. Together, we obtain $M^\vee\subseteq uM$, that is, $(L_1^\vee\cap V_2)^\vee$ is an integral hermitian $O_E$-module of type $t(L^{\flat\prime})-1$.

Consequently, $L_2$ is an integral $O_E$-lattice of $V_2$ of type $t(L^{\flat\prime})$. Since $L_2^\vee=(L_1^\vee\cap V_2)\oplus u^{-\delta-1}(V_{L^\flat}^\perp)^\integ$, it is clear that for $\delta$ sufficiently large, \eqref{eq:analytic14} holds for $i=0,1$. Thus, ($*$) is proved.

The lemma is all proved.
\end{proof}

\begin{lem}\label{le:analytic6}
Let $L$ be an integral hermitian $O_E$-module of rank $2m+1$ for some integer $m\geq 0$ with $t(L)=2m+1$. Then we have
\begin{align}\label{eq:analytic7}
\left|(L^\vee)^\integ/L\right|=q^{2m}\cdot\left|(uL^\vee)^\integ/L\right|.
\end{align}
Note that both $(L^\vee)^\integ$ and $(uL^\vee)^\integ$ are stable under the translation by $L$ as $t(L)=2m+1$.
\end{lem}

\begin{proof}
Put $V\coloneqq L\otimes_{O_F}F$. We prove by induction on $\val(L)$ for integral $O_E$-lattices $L$ of $V$ with $t(L)=2m+1$ that \eqref{eq:analytic7} holds.

The initial case is such that $\val(L)=2m+1$, that is, $L^\vee=u^{-1}L$. The pairing $u^2(\;,\;)_V$ induces a nondegenerate quadratic form on $L^\vee/L$. It is clear that $(L^\vee)^\integ/L$ is exactly the set of isotropic vectors in $L^\vee/L$ under the previous form. In particular, we have
\[
\left|(L^\vee)^\integ/L\right|=q^{2m}=q^{2m}\cdot\left|(uL^\vee)^\integ/L\right|.
\]

Now we consider $L$ with $\val(L)>2m+1$, and suppose that \eqref{eq:analytic7} holds for such $L'$ with $\val(L')<\val(L)$. Choose an orthogonal decomposition $L=L_0\oplus L_1$ in which $L_0$ is an integral hermitian $O_E$-module with fundamental invariants $(1,\dots,1)$ and such that all fundamental invariants of $L_1$ are at least $2$. In particular, $L_1$ has positive rank. It is easy to see that we may choose a hermitian $O_E$-module $L'_1$ contained in $u^{-1}L_1$ satisfying $L_1\varsubsetneq L'_1$ and $t(L'_1)=t(L_1)$. Put $L'\coloneqq L_0\oplus L'_1$. By the induction hypothesis, we have
\[
\left|(L'^\vee)^\integ/L'\right|=q^{2m}\cdot\left|(uL'^\vee)^\integ/L'\right|.
\]
It remains to show that
\begin{align}\label{eq:analytic8}
\left|((L^\vee)^\integ\setminus (L'^\vee)^\integ)/L\right|=q^{2m}\cdot\left|((uL^\vee)^\integ\setminus(uL'^\vee)^\integ)/L\right|.
\end{align}
We claim that the map
\[
((L^\vee)^\integ\setminus (L'^\vee)^\integ)/L\to((uL^\vee)^\integ\setminus(uL'^\vee)^\integ)/L
\]
given by the multiplication by $u$ is $q^{2m}$-to-$1$. Take an element $x\in(uL^\vee)^\integ\setminus(uL'^\vee)^\integ$. Its preimage is bijective to the set of elements $(y_0,y_1)\in L_0/uL_0\oplus L_1/uL_1$ such that $u^{-1}(x+(y_0,y_1))\in V^\integ$, which amounts to the equation
\[
(x,x)_V+\Tr_{E/F}(x,y_0)_V+\Tr_{E/F}(x,y_1)_V+(y_0,y_0)_V \in u^2 O_F.
\]
Since $x\in(uL_0^\vee)\times((uL_1^\vee)^\integ\setminus(u^2L_1^\vee)^\integ)$, there exists $y_1\in L_1$ such that $(x,y_1)_V\in O_E^\times$. In other words, for each $y_0$, the above relation defines a nontrivial linear equation on $L_1/uL_1$. Thus, the preimage of $x$ has cardinality $q^{2m}$. We obtain \eqref{eq:analytic8} hence complete the induction process.
\end{proof}

\begin{proof}[Proof of Proposition \ref{pr:analytic}]
We fix an element $L^\flat\in\flat(\bbV)$. If $L^\flat$ is not integral, then $\partial\Den_{L^\flat}^\rv\equiv0$ hence the proposition is trivial. Thus, we now assume $L^\flat$ integral and will freely adopt notation from Lemma \ref{le:analytic8}.

To show that $\partial\Den_{L^\flat}^\rv$ extends to a compactly supported locally constant function on $\bbV$, it suffices to show that for every $y\in V_{L^\flat}/L^\flat$, there exists an integer $\delta(y)>0$ such that $\partial\Den_{L^\flat}^\rv(y+x)$ is constant for $x\in u^{\delta(y)}(V_{L^\flat}^\perp)^\integ\setminus\{0\}$. If $L^\flat+\langle y\rangle$ is not integral, then there exists $\delta(y)>0$ such that $L^\flat+\langle y+x\rangle$ is not integral for $x\in u^{\delta(y)}(V_{L^\flat}^\perp)^\integ\setminus\{0\}$, which implies $\partial\Den_{L^\flat}^\rv(y+x)=0$.

Now we fix an element $y\in V_{L^\flat}/L^\flat$ such that $L^\flat+\langle y\rangle$ is integral. We claim that we may take $\delta(y)=a_{n+1}$, which is the maximal element in the fundamental invariants of $L^\flat$. It amounts to showing that for every fixed pair $(f_1,f_2)$ of generators of the $O_E$-module $(V_{L^\flat}^\perp)^\integ$, we have
\begin{align}\label{eq:analytic5}
\partial\Den_{L^\flat}^\rv(y+u^\delta f_1)-\partial\Den_{L^\flat}^\rv(y+u^{\delta-1} f_2)=0
\end{align}
for $\delta>a_{n-1}$. For every $\delta'\in\dZ$, we define two sets
\begin{align*}
\fL_1^{\delta'}&\coloneqq\{L\in\fL\res L\subseteq L^\vee,\delta_L=\delta',y+u^\delta f_1\in L\},\\
\fL_2^{\delta'}&\coloneqq\{L\in\fL\res L\subseteq L^\vee,\delta_L=\delta',y+u^{\delta-1}f_2\in L\}.
\end{align*}
By Remark \ref{re:analytic}(4), we have
\begin{align*}
\partial\Den_{L^\flat}^\rv(y+u^\delta f_1)&=2\sum_{\delta'\leq\delta}\sum_{\substack{L\in\fL_1^{\delta'}\\ t(L\cap V_{L^\flat})>1}}\mu(t(L))
=2\sum_{\substack{L^\flat\subseteq L^{\flat\prime}\subseteq (L^{\flat\prime})^\vee\\t(L^{\flat\prime})>1}}\sum_{\delta'\leq\delta}
\sum_{\substack{L\in\fL_1^{\delta'}\\ L\cap V_{L^\flat}=L^{\flat\prime}}}\mu(t(L)),\\
\partial\Den_{L^\flat}^\rv(y+u^{\delta-1}f_2)&=2\sum_{\delta'\leq\delta-1}\sum_{\substack{L\in\fL_2^{\delta'}\\ t(L\cap V_{L^\flat})>1}}\mu(t(L))
=2\sum_{\substack{L^\flat\subseteq L^{\flat\prime}\subseteq (L^{\flat\prime})^\vee\\t(L^{\flat\prime})>1}}\sum_{\delta'\leq\delta-1}
\sum_{\substack{L\in\fL_2^{\delta'}\\ L\cap V_{L^\flat}=L^{\flat\prime}}}\mu(t(L)).
\end{align*}
Now we claim that
\begin{align}\label{eq:analytic10}
\sum_{\delta'\leq\delta}
\sum_{\substack{L\in\fL_1^{\delta'}\\ L\cap V_{L^\flat}=L^{\flat\prime}}}\mu(t(L))-
\sum_{\delta'\leq\delta-1}
\sum_{\substack{L\in\fL_2^{\delta'}\\ L\cap V_{L^\flat}=L^{\flat\prime}}}\mu(t(L))=0
\end{align}
for every $L^{\flat\prime}$ in the summation. Since $\delta>a_{n+1}$, it follows that for $\delta'<0$, we have
\[
\fL_1^{\delta'}=\fL_2^{\delta'}=\{L\in\fL\res L\subseteq L^\vee,\delta_L=\delta',y\in L\}.
\]
Thus, the left-hand side of \eqref{eq:analytic10} equals
\begin{align}\label{eq:analytic11}
\sum_{\delta'=0}^\delta
\sum_{\substack{L\in\fL_1^{\delta'}\\ L\cap V_{L^\flat}=L^{\flat\prime}}}\mu(t(L))-
\sum_{\delta'=0}^{\delta-1}
\sum_{\substack{L\in\fL_2^{\delta'}\\ L\cap V_{L^\flat}=L^{\flat\prime}}}\mu(t(L)).
\end{align}
However, we also have $\fL_1^0=\{L\in\fL\res L\subseteq L^\vee,\delta_L=\delta',y\in L\}$, which implies
\[
\sum_{\substack{L\in\fL_1^0\\ L\cap V_{L^\flat}=L^{\flat\prime}}}\mu(t(L))
=\CF_{L^{\flat\prime}}(y)\sum_{\substack{L\subseteq L^\vee \\ L\cap V_{L^\flat}=L^{\flat\prime}\\ \delta_L=0}}\mu(t(L)),
\]
which vanishes by Lemma \ref{le:analytic8}(4). Thus, we obtain
\begin{align}\label{eq:analytic12}
\eqref{eq:analytic11}=\sum_{\delta'=1}^\delta
\sum_{\substack{L\in\fL_1^{\delta'}\\ L\cap V_{L^\flat}=L^{\flat\prime}}}\mu(t(L))-
\sum_{\delta'=0}^{\delta-1}
\sum_{\substack{L\in\fL_2^{\delta'}\\ L\cap V_{L^\flat}=L^{\flat\prime}}}\mu(t(L)).
\end{align}
Finally, the automorphism of $\fE$ sending $(L^{\flat\prime},\delta',\varepsilon)$ to $(L^{\flat\prime},\delta'-1,\varepsilon\circ(u\alpha\cdot))$, where $\alpha\in O_E^\times$ is the element satisfying $f_1=\alpha f_2$, induces a bijection from $\fL_1^{\delta'}$ to $\fL_2^{\delta'-1}$ preserving both $L\cap V_{L^\flat}$ and $t(L)$. Thus, \eqref{eq:analytic12} vanishes hence \eqref{eq:analytic10} and \eqref{eq:analytic5} hold.

Now we show that the support of $\widehat{\partial\Den_{L^\flat}^\rv}$ is contained in $\bbV^\integ$. Take an element $z\in\bbV\setminus\bbV^\integ$. Using Remark \ref{re:analytic}(4), we have
\begin{align*}
\widehat{\partial\Den_{L^\flat}^\rv}(z)&=\int_{\bbV}\widehat{\partial\Den_{L^\flat}^\rv}(x)\psi(\Tr_{E/F}(x,z)_{\bbV})\rd z \\
&=2\sum_{\substack{L^\flat\subseteq L\subseteq L^\vee \\ t(L\cap V_{L^\flat})>1}}\mu(t(L))\vol(L)\CF_{L^\vee}(z)  \\
&=2\sum_{\substack{L^\flat\subseteq L^{\flat\prime}\subseteq (L^{\flat\prime})^\vee\\t(L^{\flat\prime})>1}}
\sum_{\substack{L\subseteq L^\vee \\ L\cap V_{L^\flat}=L^{\flat\prime}\\ z\in L^\vee}}\mu(t(L))\vol(L) \\
&=2\sum_{\substack{L^\flat\subseteq L^{\flat\prime}\subseteq (L^{\flat\prime})^\vee\\t(L^{\flat\prime})>1}}
\vol(L^{\flat\prime})\vol((V_{L^\flat}^\perp)^\integ)\sum_{\substack{L\subseteq L^\vee \\ L\cap V_{L^\flat}=L^{\flat\prime}\\ z\in L^\vee}}q^{-\delta_L}\mu(t(L)),
\end{align*}
which is valid and vanishes by Lemma \ref{le:analytic8}(3).

Proposition \ref{pr:analytic} is proved.
\end{proof}

\if false

\begin{proposition}
Let $L^\flat$ be an element of $\flat(\bbV)$ (Definition \ref{de:analytic}). Then $\partial\Den_{L^\flat}^\rv$ extends (uniquely) to a (compactly supported) locally constant function on $\bbV$, which we still denote by $\partial\Den_{L^\flat}^\rv$. Moreover, $\widehat{\partial\Den_{L^\flat}^\rv}$ vanishes on $V_{L^\flat}^\perp\setminus\bbV^\integ$ (Definition \ref{de:analytic1}).
\end{proposition}

\begin{proof}
To shorten formulae, in this proof, we put
\[
\mu(t)\coloneqq\prod_{i=1}^{\frac{t}{2}-1}(1-q^{2i})
\]
for every positive even integer $t$.

We fix an element $L^\flat\in\flat(\bbV)$. If $L^\flat$ is not integral, then $\partial\Den_{L^\flat}^\rv\equiv0$ hence the proposition is trivial. Thus, we now assume $L^\flat$ integral.

We fix a basis $\{f\}$ of $V_{L^\flat}^\perp$ such that $(f,f)_{\bbV}\in O_F^\times$. For every compact subset $X$ of $\bbV$ not contained in $V_{L^\flat}$, we denote by $\delta_X$ the maximal integer such that the image of $X$ under the projection map $\bbV\to V_{L^\flat}^\perp$ induced by the orthogonal decomposition $\bbV=V_{L^\flat}\oplus V_{L^\flat}^\perp$ is contained in $\langle u^{\delta_X}f\rangle$. We denote by $\fL$ the set of $O_E$-lattices of $\bbV$ containing $L^\flat$, and by $\fE$ the set of triples $(L^{\flat\prime},\delta,\varepsilon)$ in which $L^{\flat\prime}$ is an $O_E$-lattice of $V_{L^\flat}$ containing $L^\flat$, $\delta\in\dZ$, and $\varepsilon\in L^{\flat\prime}\otimes_{O_F}F/O_F$. We will show the following four statements.
\begin{enumerate}
  \item The map $\fL\to\fE$ sending $L$ to the triple $(L\cap V_{L^\flat},\delta_L,\varepsilon_L)$ is a bijection, where $\varepsilon_L$ is the image of $u^{\delta_L}f$ under the extension map $\langle u^{\delta_L}f\rangle\to(L\cap V_{L^\flat})\otimes_{O_F}F/O_F$ induced by $L$. Moreover,
      \begin{itemize}
        \item if $L$ is integral, then $L\cap V_{L^\flat}$ is integral and $\varepsilon_L\in(L\cap V_{L^\flat})^\vee/(L\cap V_{L^\flat})$;

        \item when $\delta_L\geq 0$, $L$ is integral if and only if $L\cap V_{L^\flat}$ is integral and $\varepsilon_L\in((L\cap V_{L^\flat})^\vee)^\integ/(L\cap V_{L^\flat})$.
      \end{itemize}

  \item For every fixed $\delta\geq 0$ and an integral $O_E$-lattice $L^{\flat\prime}$ of $V_{L^\flat}$ containing $L^\flat$ with $t(L^{\flat\prime})>1$, we have
      \[
      \sum_{\substack{L\subseteq L^\vee \\ L\cap V_{L^\flat}=L^{\flat\prime}\\ \delta_L=\delta}}\mu(t(L))=0.
      \]

  \item The function $\partial\Den_{L^\flat}^\rv$ extends to a compactly supported locally constant function on $\bbV$.

  \item The Fourier transform $\widehat{\partial\Den_{L^\flat}^\rv}$ vanishes on $V_{L^\flat}^\perp\setminus\bbV^\integ$.
\end{enumerate}

Part (1) is straightforward. In what follows, we will denote by $L^{\flat\prime}_{\delta,\epsilon}\in\fL$ the preimage of $(L^{\flat\prime},\delta,\varepsilon)$ under this bijection.

For (2), we have
\[
\sum_{\substack{L\subseteq L^\vee \\ L\cap V_{L^\flat}=L^{\flat\prime}\\ \delta_L=\delta}}\mu(t(L))=
\sum_{\varepsilon\in((L^{\flat\prime})^\vee)^\integ/L^{\flat\prime}}\mu(t(L^{\flat\prime}_{\delta,\varepsilon}))
\]
by (1). We fix an orthogonal decomposition $L^{\flat\prime}=L_0\oplus L_1$ in which $L_0$ is a self-dual hermitian $O_E$-lattice of rank $n-1-t(L^\flat)$. Then we have
\[
\sum_{\varepsilon\in((L^{\flat\prime})^\vee)^\integ/L^{\flat\prime}}\mu(t(L^{\flat\prime}_{\delta,\varepsilon}))
=\sum_{\varepsilon\in(L_1^\vee)^\integ/L_1}\mu(t(L^{\flat\prime}_{\delta,\varepsilon}))
=\sum_{\varepsilon\in(L_1^\vee)^\integ/L_1}\(\prod_{i=1}^{\frac{t(L^{\flat\prime}_{\delta,\varepsilon})}{2}-1}(1-q^{2i})\).
\]
By Remark \ref{re:analytic3}, we know that
\begin{align}\label{eq:analytic9}
t(L^{\flat\prime}_{\delta,\varepsilon})=
\begin{dcases}
t(L^\flat)+1, &\text{if $\varepsilon\in(uL_1^\vee)^\integ/L_1$,}\\
t(L^\flat)-1, &\text{if $\varepsilon\in((L_1^\vee)^\integ\setminus (uL_1^\vee)^\integ)/L_1$}.
\end{dcases}
\end{align}
Thus, (2) follows from Lemma \ref{le:analytic6} below and the assumption that $t(L^{\flat\prime})>1$.

For (3), it suffices to show that for every $y\in V_{L^\flat}/L^\flat$, there exists an integer $\delta(y)>0$ such that $\partial\Den_{L^\flat}^\rv(y+x)$ is constant for $x\in\langle u^{\delta(y)}f\rangle\setminus\{0\}$. If $L^\flat+\langle y\rangle$ is not integral, then there exists $\delta(y)>0$ such that $L^\flat+\langle y+x\rangle$ is not integral for $x\in\langle u^{\delta(y)}f\rangle\setminus\{0\}$, which implies $\partial\Den_{L^\flat}^\rv(y+x)=0$.

Now we fix an element $y\in V_{L^\flat}/L^\flat$ such that $L^\flat+\langle y\rangle$ is integral. For a pair of integers $\delta,\delta'$, we define
\[
\fL_{\delta,\delta'}\coloneqq\{L\in\fL\res L\subseteq L^\vee,\delta_L=\delta',y+u^\delta f\in L\}.
\]
By Remark \ref{re:analytic}(4), we have for $\delta\geq 0$,
\begin{align*}
\partial\Den_{L^\flat}^\rv(y+u^\delta f)=\sum_{\delta'\leq\delta}\sum_{\substack{L\in\fL_{\delta,\delta'}\\ t(L\cap V_{L^\flat})>1}}\mu(t(L))
=\sum_{\delta'\leq\delta}\sum_{\substack{L^\flat\subseteq L^{\flat\prime}\subseteq (L^{\flat\prime})^\vee\\t(L^{\flat\prime})>1}}
\sum_{\substack{L\in\fL_{\delta,\delta'}\\ L\cap V_{L^\flat}=L^{\flat\prime}}}\mu(t(L)).
\end{align*}
Let $a_{n-1}$ be the maximal element in the fundamental invariants of $L^\flat$. Then when $\delta'<\delta-a_{n-1}$, we have
\begin{align}\label{eq:analytic6}
\fL_{\delta,\delta'}=\fL_{\delta-1,\delta'}=\{L\in\fL\res L\subseteq L^\vee,\delta_L=\delta',y\in L\}.
\end{align}
Now we take $\delta(y)=a_{n-1}$. Then for $\delta>\delta(y)$, we have
\begin{align*}
&\quad\partial\Den_{L^\flat}^\rv(y+u^\delta f)-\partial\Den_{L^\flat}^\rv(y+u^{\delta-1} f) \\
&=\sum_{\delta'=0}^\delta\sum_{\substack{L^\flat\subseteq L^{\flat\prime}\subseteq (L^{\flat\prime})^\vee\\t(L^{\flat\prime})>1}}
\sum_{\substack{L\in\fL_{\delta,\delta'}\\ L\cap V_{L^\flat}=L^{\flat\prime}}}\mu(t(L))
-\sum_{\delta'=0}^{\delta-1}\sum_{\substack{L^\flat\subseteq L^{\flat\prime}\subseteq (L^{\flat\prime})^\vee\\t(L^{\flat\prime})>1}}
\sum_{\substack{L\in\fL_{\delta-1,\delta'}\\ L\cap V_{L^\flat}=L^{\flat\prime}}}\mu(t(L)).
\end{align*}
Again by \eqref{eq:analytic6}, we have
\[
\sum_{\substack{L^\flat\subseteq L^{\flat\prime}\subseteq (L^{\flat\prime})^\vee\\t(L^{\flat\prime})>1}}
\sum_{\substack{L\in\fL_{\delta,0}\\ L\cap V_{L^\flat}=L^{\flat\prime}}}\mu(t(L))
=
\sum_{\substack{L^\flat+\langle y\rangle\subseteq L^{\flat\prime}\subseteq (L^{\flat\prime})^\vee\\t(L^{\flat\prime})>1}}
\sum_{\substack{L\subseteq L^\vee \\ L\cap V_{L^\flat}=L^{\flat\prime}\\ \delta_L=0}}\mu(t(L))
\]
which vanishes by (2). Thus, we have
\begin{align*}
&\qquad\partial\Den_{L^\flat}^\rv(y+u^\delta f)-\partial\Den_{L^\flat}^\rv(y+u^{\delta-1} f) \\
&=\sum_{\delta'=1}^\delta\sum_{\substack{L^\flat\subseteq L^{\flat\prime}\subseteq (L^{\flat\prime})^\vee\\t(L^{\flat\prime})>1}}
\(\sum_{\substack{L\in\fL_{\delta,\delta'}\\ L\cap V_{L^\flat}=L^{\flat\prime}}}\mu(t(L))
-\sum_{\substack{L\in\fL_{\delta-1,\delta'-1}\\ L\cap V_{L^\flat}=L^{\flat\prime}}}\mu(t(L))\).
\end{align*}
By (1) and \eqref{eq:analytic9}, when $\delta'\geq 1$, the automorphism of $\fE$ sending $(L^{\flat\prime},\delta',\varepsilon)$ to $(L^{\flat\prime},\delta'-1,\varepsilon)$ induces a bijection from $\fL_{\delta,\delta'}$ to $\fL_{\delta-1,\delta'-1}$ preserving both $L\cap V_{L^\flat}$ and $t(L)$. Therefore, we obtain
\[
\partial\Den_{L^\flat}^\rv(y+u^\delta f)=\partial\Den_{L^\flat}^\rv(y+u^{\delta-1} f)
\]
for $\delta>\delta(y)$. Finally, by Remark \ref{re:analytic}(5), we know that $\partial\Den_{L^\flat}^\rv(y+x)$ is constant for $x\in\langle u^{\delta(y)}f\rangle\setminus\{0\}$. Part (3) is proved.

For (4), we define a function $\Delta_{L^\flat}$ on $V_{L^\flat}^\perp$ by the formula
\[
\Delta_{L^\flat}(x)\coloneqq\int_{V_{L^\flat}}\partial\Den_{L^\flat}^\rv(y+x)\rd y,
\]
which is a compactly supported locally constant function. It suffices to show that $\Delta_{L^\flat}$ is invariant under the translation by the $O_E$-lattice $u^{-1}(V_{L^\flat}^\perp)^\integ$. For $x\in V_{L^\flat}^\perp\setminus u^{-1}(V_{L^\flat}^\perp)^\integ$, $x+u^{-1}(V_{L^\flat}^\perp)^\integ$ is contained in $\{[\alpha]_{L^\flat}(x)\res \alpha\in O_E^\times\}$. Thus, the invariance around $x$ follows from Remark \ref{re:analytic}(5). Again by Remark \ref{re:analytic}(5), it remains to show that for every $\delta\geq 0$, we have $\Delta_{L^\flat}(u^\delta f)=\Delta_{L^\flat}(u^{\delta-1}f)$, which is equivalent to
\begin{align}\label{eq:analytic4}
\sum_{y\in V_{L^\flat}/L^\flat}\sum_{\substack{L^\flat+\langle y+u^\delta f\rangle\subseteq L\subseteq L^\vee \\ t(L\cap V_{L^\flat})>1}}
\mu(t(L))=
\sum_{y\in V_{L^\flat}/L^\flat}\sum_{\substack{L^\flat+\langle y+u^{\delta-1}f\rangle\subseteq L\subseteq L^\vee \\ t(L\cap V_{L^\flat})>1}}
\mu(t(L))
\end{align}
by Remark \ref{re:analytic}(4). Note that the left-hand side is a partial sum of the (finite) sum on the right-side side.

By (1), we know that the lattices $L$ that appear in the right-hand side but not the left-hand side of \eqref{eq:analytic4} are exactly those with $\delta_L=\delta$. Moreover, for such $L$, the number of $y\in V_{L^\flat}/L^\flat$ such that $y+u^\delta f\in L$ equals $|(L\cap V_{L^\flat})/L^\flat|$. Thus, \eqref{eq:analytic4} is equivalent to
\begin{align*}
\sum_{\substack{L^\flat\subseteq L^{\flat\prime}\subseteq (L^{\flat\prime})^\vee\\t(L^{\flat\prime})>1}}
|L^{\flat\prime}/L^\flat|
\sum_{\substack{L\subseteq L^\vee \\ L\cap V_{L^\flat}=L^{\flat\prime}\\ \delta_L=\delta}}\mu(t(L))=0,
\end{align*}
which holds by (2). Thus, (4) follows.

Proposition \ref{pr:analytic} follows from (3) and (4).
\end{proof}

In the remaining part of this subsection, we will compute $\partial\Den(\bbL)$ explicitly when $r=1$. However, we will make the discussion for general rank to our best.

Let $L^\flat$ be an element of $\flat(\bbV)$ (Definition \ref{de:analytic}). Define the \emph{(normalized) local Siegel series} of $L^\flat$ to be the polynomial $\Den(X,L^\flat)\in\dQ[X]$ such that for every integer $s\geq 1$,
\[
\Den(q^{-s},L^\flat)=\frac{\Den(H_{r-1+s},L^\flat)}{\prod_{i=s}^{r-1+s}(1-q^{-2i})},
\]
where $\Den$ is defined in Definition \ref{de:density}. The following lemma is parallel to Lemma \ref{le:analytic2}.

\begin{lem}\label{le:analytic3}
We have
\[
\Den(X,L^\flat)=\sum_{L^\flat\subseteq L\subseteq L^\vee}|L/L^\flat|\cdot X^{2\length_{O_E}(L/L^\flat)}\prod_{i=1}^{\frac{t(L)-1}{2}}(1-q^{2i}X^2),
\]
where the sum is taken over integral $O_E$-lattices of $V_{L^\flat}$ containing $L^\flat$.
\end{lem}

\begin{proof}
This is a direct consequence of Lemma \ref{le:analytic1} and the definition above.
\end{proof}

The following lemma on the functional equations for local Siegel series is due to Ikeda.

\begin{lem}\label{le:analytic5}
We have
\begin{align*}
\Den(X,\bbL)&=-X^{\val(L)}\cdotp\Den(X^{-1},\bbL),\\
\Den(X,L^\flat)&=(q^{1/2}X)^{\val(L^\flat)-1}\cdot\Den((qX)^{-1},L^\flat),
\end{align*}
when $\bbL$ and $L^\flat$ are integral, respectively.
\end{lem}

\begin{proof}
The first identity follows from \cite{Ike08}*{Corollary~3.2(i)} and Lemma \ref{le:analytic7}(1). The second identity follows from \cite{Ike08}*{Corollary~3.2(ii)} and Lemma \ref{le:analytic7}(1).
\end{proof}

The following lemma provides some induction formulae for local Siegel series and their derivatives.

\begin{lem}\label{le:analytic4}
Suppose that $L^\flat$ is integral and has fundamental invariants $(a_1,\dots,a_{n-1})$ (with $a_{n-1}>0$). Let $x$ be an element of $(V_{L^\flat}^\perp)^\integ$.
\begin{enumerate}
  \item If $\val(x)>a_{n-1}$, then for every integer $s\geq 2r$, we have
     \[
     \Den(H_s,L^\flat+\langle x\rangle)=q^{2r-2s}\cdot\Den(H_s,L^\flat+\langle u^{-1}x\rangle)+(1-q^{-2s})\cdot\Den(H_{s-1},L^\flat).
     \]

  \item If $\val(x)>a_{n-1}$, then we have
     \[
     \Den(X,L^\flat+\langle x\rangle)=X^2\cdot\Den(X,L^\flat+\langle u^{-1}x\rangle)+
     (1-X^2)\cdot\Den(X,L^\flat).
     \]

  \item If $\val(x)\geq a_{n-1}-1$, then we have
     \[
     \Den(X,L^\flat+\langle x\rangle)=\Den(X,L^\flat)-q^{\frac{\val(L^\flat)-1}{2}}X^{\val(x)+1}\cdot\Den(q^{-1}X,L^\flat).
     \]
\end{enumerate}
\end{lem}

\begin{proof}
For (1), by Lemma \ref{le:analytic1}, we have
\begin{align*}
\Den(H_s,L^\flat+\langle x\rangle)=\Den(H_s,L^\flat+\langle x\rangle)_1+\Den(H_s,L^\flat+\langle x\rangle)_2,
\end{align*}
where
\begin{align*}
\Den(H_s,L^\flat+\langle x\rangle)_1&\coloneqq\sum_{L^\flat+\langle u^{-1}x\rangle\subseteq L'\subseteq L'^\vee}
|L'/(L^\flat+\langle x\rangle)|^{2r-2s}\prod_{s-r-\frac{t(L')}{2}<i\leq s}(1-q^{-2i}),\\
\Den(H_s,L^\flat+\langle x\rangle)_2&\coloneqq\sum_{\substack{L^\flat+\langle x\rangle\subseteq L'\subseteq L'^\vee \\ u^{-1}x\not\in L'}}|L'/(L^\flat+\langle x\rangle)|^{2r-2s}\prod_{s-r-\frac{t(L')}{2}<i\leq s}(1-q^{-2i}).
\end{align*}
For the first term, we have
\begin{align*}
\Den(H_s,L^\flat+\langle x\rangle)_1&=\sum_{L^\flat+\langle u^{-1}x\rangle\subseteq L'\subseteq L'^\vee}
(q\cdot|L'/(L^\flat+\langle u^{-1}x\rangle)|)^{2r-2s}\prod_{s-r-\frac{t(L')}{2}<i\leq s}(1-q^{-2i}) \\
&=q^{2r-2s}\cdot\Den(H_s,L^\flat+\langle u^{-1}x\rangle).
\end{align*}
Now we evaluate the second term. For an $O_E$-lattice $L'$ of $\bbV$ containing $L^\flat\oplus\langle x\rangle$ but not $x$, we may write $L'=M^\flat+\langle x\rangle$ for some $M^\flat\in\flat(\bbV)$. We make the following claim.
\begin{itemize}
  \item[($*$)] The hermitian $O_E$-module $L'$ is integral if and only if the hermitian $O_E$-module $M^\flat$ is integral; and when they are integral, $t(M^\flat)=t(L')-1$.
\end{itemize}
Assuming ($*$), by the same argument (toward the counting of the second sum in the statement) of \cite{CY20}*{Proposition~4.8} and \cite{CY20}*{Theorem~4.9} (with $d=1$ and $e_n=x$ in our case),\footnote{Here, our claim ($*$) replaces \cite{CY20}*{Conjecture~4.4}.} we obtain
\begin{align*}
\Den(H_s,L^\flat+\langle x\rangle)_2&=\sum_{L^\flat\subseteq L^{\flat\prime}\subseteq (L^{\flat\prime})^\vee}
|L^{\flat\prime}/L^\flat|\cdot|L^{\flat\prime}/L^\flat|^{2r-2s}\prod_{s-r-\frac{t(L^{\prime\flat})+1}{2}<i\leq s}(1-q^{-2i}) \\
&=\sum_{L^\flat\subseteq L^{\flat\prime}\subseteq (L^{\flat\prime})^\vee}
|L^{\flat\prime}/L^\flat|^{(2r-1)-2(s-1)}\prod_{(s-1)-\frac{(2r-1)+t(L^{\prime\flat})}{2}<i\leq s}(1-q^{-2i}) \\
&=(1-q^{-2s})\cdot\Den(H_{s-1},L^\flat).
\end{align*}
Thus, (1) follows as long as we confirm ($*$).

For ($*$), put $M^{\flat\prime}\coloneqq M^\flat\cap(L^\flat+\langle x\rangle)$. Then $L^\flat+\langle x\rangle=M^{\flat\prime}+\langle x\rangle$ and $M^{\flat\prime}$ is an integral hermitian $O_E$-module with the same fundamental invariants as $L^\flat$. Now if $L'$ is integral, then $M^\flat$ is clearly integral. Conversely, if $M^\flat$ is integral, then since $M^{\flat\prime}\subseteq M^\flat\subseteq(M^\flat)^\vee\subseteq(M^{\flat\prime})^\vee$, we have for every element $y\in M^\flat$, $u^{a_{n-1}}y\in M^{\flat\prime}\subseteq L^\flat+\langle x\rangle$, which implies $(y,x)_{\bbV}\in O_E$, hence $L'$ is integral. The second assertion amounts to that $M^\flat$ and $L'$ have the same number of zeros in their fundamental invariants, which follows from Lemma \ref{le:analytic7}(1) and the fact that $(y,x)_{\bbV}\in O_E$ for every $y\in L'$.

Part (2) is a direct consequence of (1) and the definitions of local Siegel series.

For (3), plugging the functional equations in Lemma \ref{le:analytic5} for $\bbL=L^\flat+\langle ux\rangle$, $\bbL=L^\flat+\langle x\rangle$, and $L^\flat$ to (2), we obtain
\begin{align*}
&-X^{\val(L^\flat)+\val(x)+2}\cdot\Den(X^{-1},L^\flat+\langle ux\rangle)=
-X^{\val(L^\flat)+\val(x)+2}\cdot\Den(X^{-1},L^\flat+\langle x\rangle) \\
&\qquad\qquad+(1-X^2)(q^{1/2}X)^{\val(L^\flat)-1}\cdot\Den((qX)^{-1},L^\flat).
\end{align*}
Replacing $X$ by $X^{-1}$, we obtain
\[
\Den(X,L^\flat+\langle ux\rangle)=\Den(X,L^\flat+\langle x\rangle)+(1-X^2)q^{\frac{\val(L^\flat)-1}{2}}X^{\val(x)+1}\cdot\Den(q^{-1}X,L^\flat).
\]
Together with (2), we obtain (3).
\end{proof}

\begin{lem}
Suppose that $L^\flat$ is integral and has fundamental invariants $(a_1,\dots,a_{n-1})$ (with $a_{n-1}>0$). Let $x$ be an element of $V_{L^\flat}^\perp$ such that $\val(x)>a_{n-1}$. Then we have
\begin{align*}
\partial\Den_{L^\flat}(x)-\partial\Den_{L^\flat}(u^{-1}x)&=2\Den(1,L^\flat), \\
\partial\Den_{L^\flat}^\rh(x)-\partial\Den_{L^\flat}^\rh(u^{-1}x)&=2\Den(1,L^\flat).
\end{align*}
\end{lem}

\begin{proof}
We have
\[
\partial\Den_{L^\flat}^\rh(x)-\partial\Den_{L^\flat}^\rh(u^{-1}x)=2
\sum_{\substack{L^\flat+\langle x\rangle\subseteq L\subseteq L^\vee \\ u^{-1}x\not\in L \\ t(L\cap V_{L^\flat})=1}}1
\]
by Definition \ref{de:density1}. We may rewrite the sum as
\[
\sum_{\substack{L^\flat+\langle x\rangle\subseteq L\subseteq L^\vee \\ u^{-1}x\not\in L \\ t(L\cap V_{L^\flat})=1}}1=
\sum_{\substack{L^\flat\subseteq L^{\flat\prime}\subseteq(L^{\flat\prime})^\vee \\ t(L^{\flat\prime})=1}}
\left|\left\{L\subseteq L^\vee\res L\cap V_{L^\flat}=L^{\flat\prime}, L\cap V_{L^\flat}^\perp=\langle x\rangle\right\}\right|.
\]
Now we compute the cardinality in the above summation for each $L^{\flat\prime}$. For an $O_E$-lattice $L$ of $\bbV$ satisfying $L\cap V_{L^\flat}^\perp=\langle x\rangle$, we let $L_\flat$ be the image of $L$ under the projection map $\bbV\to V_{L^\flat}$ induced by the orthogonal decomposition $\bbV=V_{L^\flat}\oplus V_{L^\flat}^\perp$. Then $L$ is determined by $L_\flat$ and the extension map $\varepsilon_L\colon L_\flat\to\langle x\rangle\otimes_{O_E}E/O_E$. It is clear that $L\cap V_{L^\flat}=L^{\flat\prime}$ if and only if $L^{\flat\prime}$ is contained in $L_\flat$ and is the kernel of $\varepsilon_L$. On the other hand, the facts $t(L^{\flat\prime})=1$ and $\val(x)>a_{n-1}$ imply that $L$ is further integral if and only if $L_\flat$ is further integral. It then follows easily that
\[
\left|\left\{L\subseteq L^\vee\res L\cap V_{L^\flat}=L^{\flat\prime}, L\cap V_{L^\flat}^\perp=\langle x\rangle\right\}\right|
=q^{\frac{\val(L^{\flat\prime})-1}{2}}.
\]
Together, we obtain
\begin{align*}
\partial\Den_{L^\flat}^\rh(x)-\partial\Den_{L^\flat}^\rh(u^{-1}x)&=2\sum_{\substack{L^\flat\subseteq L^{\flat\prime}\subseteq(L^{\flat\prime})^\vee \\ t(L^{\flat\prime})=1}}q^{\frac{\val(L^{\flat\prime})-1}{2}} \\
&=2q^{\frac{\val(L^\flat)-1}{2}}\cdot\Den(q^{-1},L^\flat)
=2\Den(1,L^\flat),
\end{align*}
in which the last equality follows from Lemma \ref{le:analytic5}.
\end{proof}

\fi

\subsection{Bruhat--Tits stratification}
\label{ss:bt}

Let the setup be as in Subsection \ref{ss:kr}. We assume Hypothesis \ref{hy:rz}.

We first generalize Definition \ref{de:rz_special} to a more general context. For every subset $X$ of $\bbV$ such that $\langle X\rangle$ is finitely generated, we put
\[
\cN(X)\coloneqq\bigcap_{x\in X}\cN(x),
\]
which is always a finite intersection, and depends only on $\langle X\rangle$. Clearly, we have $\cN(X')\subseteq\cN(X)$ if $\langle X\rangle \subseteq\langle X'\rangle$. When $X=\{x,\dots\}$ is explicitly presented, we simply write $\cN(x,\dots)$ instead of $\cN(\{x,\dots\})$.

\begin{remark}\label{re:special}
When $\langle X\rangle$ is an $O_E$-lattice of $\bbV$, the formal subscheme $\cN(X)$ is a proper closed subscheme of $\cN$. This can be proved by the same argument for \cite{LZ}*{Lemma~2.10.1}.
\end{remark}

\begin{definition}\label{de:bt}
Let $\Lambda$ be a vertex $O_E$-lattice of $\bbL$ (Definition \ref{de:analytic1}).
\begin{enumerate}
  \item We equip the $k$-vector space $\Lambda^\vee/\Lambda$ with a $k$-valued pairing $(\;,\;)_{\Lambda^\vee/\Lambda}$ by the formula
     \[
     (x,y)_{\Lambda^\vee/\Lambda}\coloneqq u^2\Tr_{E/F}(x^\sharp,y^\sharp)_{\bbV}\mod(u^2)
     \]
     where $x^\sharp$ and $y^\sharp$ are arbitrary lifts of $x$ and $y$, respectively. Then $\Lambda^\vee/\Lambda$ becomes a nonsplit (nondegenerate) quadratic space over $k$ of (even positive) dimension $t(\Lambda)$.

  \item Let $\cV_\Lambda$ be the reduced subscheme of $\cN(\Lambda)$, and put
      \[
      \cV_\Lambda^\circ\coloneqq\cV_\Lambda-\bigcup_{\Lambda\subsetneqq\Lambda'}\cV_{\Lambda'}.
      \]
\end{enumerate}
\end{definition}

\begin{proposition}[Bruhat--Tits stratification, \cite{Wu}]\label{pr:bt}
The reduced subscheme $\cN_\red$ is a disjoint union of $ \cV_\Lambda^\circ$ for all vertex $O_E$-lattices $\Lambda$ of $\bbV$ in the sense of stratification, such that $\cV_{\Lambda}\cap\cV_{\Lambda'}$ coincides with $\cV_{\Lambda+\Lambda'}$ (resp.\ is empty) if $\Lambda+\Lambda'$ is (resp.\ is not) a vertex $O_E$-lattice.

Moreover, for every vertex $O_E$-lattice $\Lambda$,
\begin{enumerate}
  \item $\cV_\Lambda$ is canonically isomorphic to the generalized Deligne--Lusztig variety of $\OG(\Lambda^\vee/\Lambda)$ over $\ol{k}$ classifying maximal isotropic subspaces $U$ of $(\Lambda^\vee/\Lambda)\otimes_k\ol{k}$ satisfying
      \[
      \dim(U\cap\delta(U))=\tfrac{t(\Lambda)}{2}-1,
      \]
      where $\delta\in\Gal(\ol{k}/k)$ denotes the Frobenius element;

  \item the intersection of $\cV_\Lambda$ with each connected component of $\cN_\red$ is connected, nonempty, and smooth projective over $\ol{k}$ of dimension $\tfrac{t(\Lambda)}{2}-1$.
\end{enumerate}

\end{proposition}

\begin{proof}
This follows from \cite{Wu}*{Proposition~5.13 \& Theorem~5.18}. Note that we use lattices in $\bbV$, which is different from the hermitian space $C$ used in \cite{Wu}, to parameterize strata. By the obvious analogue of \cite{KR11}*{Lemma~3.9}, we may naturally identify $\bbV$ with $C$, after which the stratum $\cS_\Lambda$ in \cite{Wu} coincides with our stratum $\cV_{u\Lambda^\vee}$.
\end{proof}

\begin{remark}
In the above proposition, when $t(\Lambda)=4$, $\cV_\Lambda$ is isomorphic to two copies of $\dP^1_{\ol{k}}$, though we do not need this explicit description in the following.
\end{remark}

\begin{corollary}\label{co:bt1}
For every nonzero element $x\in\bbV$, we have
\[
\cN(x)_\red=\bigcup_{x\in\Lambda}\cV^\circ_\Lambda
\]
where the union is taken over all vertex $O_E$-lattices of $\bbV$ containing $x$.
\end{corollary}

\begin{proof}
Since $\cN(x)_\red$ is a reduced closed subscheme of $\cN_\red$, it suffices to check that
\[
\cN(x)(\ol{k})=\bigcup_{x\in\Lambda}\cV^\circ_\Lambda(\ol{k}).
\]
By Definition \ref{de:bt}(2), we have
\[
\cN(x)(\ol{k})\supseteq\bigcup_{x\in\Lambda}\cV^\circ_\Lambda(\ol{k}).
\]
For the other direction, by Proposition \ref{pr:bt}, we have to show that if $\Lambda$ does not contain $x$, then $\cN(x)(\ol{k})\cap\cV^\circ_\Lambda(\ol{k})=\emptyset$. Suppose that we have $s\in\cN(x)(\ol{k})\cap\cV^\circ_\Lambda(\ol{k})$, then $s$ should belong to $\cV_{\Lambda'}(\ol{k})$ where $\Lambda'$ is the $O_E$-lattice generated by $\Lambda$ and $x$. In particular, $\Lambda'$ is vertex and strictly contains $\Lambda$. But this contradicts with the definition of $\cV^\circ_\Lambda$. The corollary follows.
\end{proof}

\begin{corollary}\label{co:bt2}
Suppose that $r\geq 2$. For every $x\in\bbV$, the intersection of $\cN(x)$ with each connected component of $\cN_\red$ is strictly a closed subscheme of the latter.
\end{corollary}

\begin{proof}
By Corollary \ref{co:bt1} and Proposition \ref{pr:bt}(2), it suffices to show that the intersection of all vertex $O_E$-lattices of $\bbV$ is $\{0\}$.

Take a nonsplit hermitian subspace $V_2$ of $\bbV$ of dimension $2$ and an $O_E$-lattice $L_2$ of $V_2$ of fundamental invariants $(1,1)$. Then the orthogonal complement $V_2^\perp$ of $V_2$ in $\bbV$ admits a self-dual $O_E$-lattice $L_1$. Choose a normal basis (Remark \ref{re:analytic3}) $\{e_1,\dots,e_{2r-2}\}$ of $L_1$ under which the moment matrix is given by $\(\begin{smallmatrix} 0 & u^{-1}\\ -u^{-1} & 0 \end{smallmatrix}\)^{\oplus r-1}$. For every tuple $a=(a_1,\dots,a_{2r-2})\in\dZ^{2r-2}$ satisfying $a_{2i-1}+a_{2i}=0$ for $1\leq i\leq r-1$, the $O_E$-lattice
\[
\Lambda_a\coloneqq L_2\oplus\langle u^{a_1}e_1,\dots,u^{a_{2r-2}}e_{2r-2}\rangle
\]
is integral with fundamental invariants $(0,\dots,0,1,1)$, hence vertex. It is clear that the intersection of all such $\Lambda_a$ is $L_2$. Since $r\geq 2$, the intersection of all $2$-dimensional nonsplit hermitian subspaces of $\bbV$ is $\{0\}$. Thus, the intersection of all vertex $O_E$-lattices of $\bbV$ is $\{0\}$.
\end{proof}

\begin{lem}\label{le:tate}
Let $\Lambda$ be a vertex $O_E$-lattice of $\bbV$. For each connected component $\cV_\Lambda^+$ of $\cV_\Lambda$ and integer $d\geq 0$, the group of $d$-cycles of $\cV_\Lambda^+$, up to $\ell$-adic homological equivalence for every rational prime $\ell\neq p$, is generated by $\cV_{\Lambda'}\cap\cV_\Lambda^+$ for all vertex $O_E$-lattices $\Lambda'$ containing $\Lambda$ with $t(\Lambda')=2d+2$.
\end{lem}

\begin{proof}
Let $k'$ be the quadratic extension of $k$ in $\ol{k}$. Note that $\cV_\Lambda^+$ has a canonical structure over $k'$, so that $\cV_\Lambda^{\circ+}\coloneqq\cV_\Lambda^\circ\cap\cV_\Lambda^+$ (over $k'$) is the classical Deligne--Lusztig variety of $\SO(\Lambda^\vee/\Lambda)$ of Coxeter type.

Recall that $\delta$ is the Frobenius element of $\Gal(\ol{k}/k)$. Fix a rational prime $\ell$ different from $p$. For every finite dimensional $\ol\dQ_\ell$-vector space $V$ with an action by $\delta^2$, we denote by $V^\dag$ the subspace consisting of elements on which $\delta^2$ acts by roots of unity. Then for the lemma, it suffices to show that for every $d\geq 0$, $\rH_{2d}(\cV_\Lambda^+,\ol\dQ_\ell(-d))^\dag$ is generated by (the cycle class of) $\cV_{\Lambda'}\cap\cV_\Lambda^+$ for all vertex $O_E$-lattices $\Lambda'$ containing $\Lambda$ with $t(\Lambda')=2d+2$. By the same argument for \cite{LZ}*{Theorem~5.3.2}, it reduces to the following claim:
\begin{itemize}
  \item[($*$)] The action of $\delta^2$ on $V\coloneqq\bigoplus_{j\geq 0}\rH^{2j}(\cV_\Lambda^{\circ+},\ol\dQ_\ell(j))$ is semisimple, and $V^\dag=\rH^0(\cV_\Lambda^{\circ+},\ol\dQ_\ell)$.
\end{itemize}
There are three cases.

When $t(\Lambda)=2$, $\cV_\Lambda^{\circ+}$ is isomorphic to $\Spec\ol{k}$ hence ($*$) is trivial.

When $t(\Lambda)=4$, $\cV_\Lambda^{\circ+}$ is an affine curve hence ($*$) is again trivial.

When $t(\Lambda)\geq 6$, by Case $\pres{2}{D}_n$ (with $n=\tfrac{t(\Lambda)}{2}\geq 3$) in \cite{Lus76}*{Section~7.3}, the action of $\delta^2$ on $\bigoplus_{j\geq 0}\rH^j_c(\cV_\Lambda^{\circ+},\ol\dQ_\ell)$ has eigenvalues $\{1,q^2,q^4,\dots,q^{t(\Lambda)-2}\}$ and that the eigenvalue $q^{2j}$ appears in $\rH^{j+\frac{t(\Lambda)}{2}-1}_c(\cV_\Lambda^{\circ+},\ol\dQ_\ell)$. Moreover by \cite{Lus76}*{Theorem~6.1}, the action of $\delta^2$ is semisimple. Thus, ($*$) follows from the Poincar\'{e} duality.

The lemma is proved.
\end{proof}

\subsection{Linear invariance of intersection numbers}
\label{ss:linear}

Let the setup be as in Subsection \ref{ss:kr}. We assume Hypothesis \ref{hy:rz}.

For every nonzero element $x\in\bbV$, we define a chain complex of locally free $\sO_\cN$-modules
\[
C(x)\coloneqq\(\cdots\to 0 \to \sI_{\cN(x)} \to \sO_\cN \to 0\)
\]
supported in degrees $1$ and $0$ with the map $\sI_{\cN(x)} \to \sO_\cN$ being the natural inclusion. We extend the definition to $x=0$ by setting
\begin{align}\label{eq:linear}
C(0)\coloneqq\(\cdots\to 0 \to \omega \xrightarrow{0} \sO_\cN \to 0\)
\end{align}
supported in degrees $1$ and $0$, where $\omega$ is the line bundle from Definition \ref{de:linear1}.

The following is our main result of this subsection.

\begin{proposition}\label{pr:linear}
Let $0\leq m\leq n$ be an integer. Suppose that $x_1,\dots,x_m\in\bbV$ and $y_1,\dots,y_m\in\bbV$ generate the same $O_E$-submodule. Then we have an isomorphism
\[
\rH_i(C(x_1)\otimes_{\sO_\cN}\cdots\otimes_{\sO_\cN} C(x_m))\simeq\rH_i(C(y_1)\otimes_{\sO_\cN}\cdots\otimes_{\sO_\cN} C(y_m))
\]
of $\sO_\cN$-modules for every $i$.
\end{proposition}

Proposition \ref{pr:linear} has the following two immediate corollaries.

\begin{corollary}\label{co:linear1}
Let $0\leq m\leq n$ be an integer. Suppose that $x_1,\dots,x_m\in\bbV$ and $y_1,\dots,y_m\in\bbV$ generate the same $O_E$-submodule. Then we have
\[
[C(x_1)\otimes_{\sO_\cN}\cdots\otimes_{\sO_\cN} C(x_m)]=[C(y_1)\otimes_{\sO_\cN}\cdots\otimes_{\sO_\cN} C(y_m)]
\]
in $\rK_0(\cN)$, where $\rK_0(\cN)$ denotes the K-group of $\cN$ \cite{LL}*{Section~B}.
\end{corollary}

\begin{corollary}\label{co:linear2}
Suppose that $x_1,\dots,x_n\in\bbV$ generate an $O_E$-lattice of $\bbV$. The Serre intersection multiplicity
\begin{multline*}
\chi\(\sO_{\cN(x_1)}\overset{\dL}\otimes_{\sO_\cN}\cdots\overset{\dL}\otimes_{\sO_\cN}\sO_{\cN(x_n)}\) \\
\coloneqq\sum_{i,j\geq 0}(-1)^{i+j}\length_{O_{\breve{E}}}
\rH^j\(\cN,\rH_i\(\sO_{\cN(x_1)}\overset{\dL}\otimes_{\sO_\cN}\cdots\overset{\dL}\otimes_{\sO_\cN}\sO_{\cN(x_n)}\)\)
\end{multline*}
depends only on the $O_E$-lattice of $\bbV$ generated by $x_1,\dots,x_n$. Note that by Remark \ref{re:special}, the above number is finite.
\end{corollary}

Now we start to prove Proposition \ref{pr:linear}, following \cite{How19}. Let $(X,\iota_X,\lambda_X)$ be the universal object over $\cN$. We have a short exact sequence
\begin{align*}
0 \to \Fil(X) \to \rD(X) \to \Lie(X) \to 0
\end{align*}
of locally free $\sO_\cN$-modules, where $\rD(X)$ denotes the covariant crystal of $X$ restricted to the Zariski site of $\cN$. Then $\iota_X$ induces actions of $O_E$ on all terms so that the short exact sequence is $O_E$-linear.

We define an $\sO_\cN$-submodule $F_X$ of $\Lie(X)$ as the kernel of $\iota_X(u)-u$ on $\Lie(X)$, which is stable under the $O_E$-action.

\begin{lem}\label{le:linear1}
The $\sO_\cN$-submodule $F_X$ is locally free of rank $n-1$ and is locally a direct summand of $\Lie(X)$.
\end{lem}

\begin{proof}
Let $s\in\cN(\ol{k})$ be a closed point. By the Wedge condition and the Spin condition in Definition \ref{de:exotic}, we know that the map
\[
\iota_X(u)-u\colon\Lie(X)\otimes_{\sO_\cN}\sO_{\cN,s}\to\Lie(X)\otimes_{\sO_\cN}\sO_{\cN,s}
\]
has rank $1$ on both generic and special fibers. Thus, $F_X\otimes_{\sO_\cN}\sO_{\cN,s}$ is a direct summand of $\Lie(X)\otimes_{\sO_\cN}\sO_{\cN,s}$ of rank $n-1$. The lemma follows.
\end{proof}

The symmetrization $\sigma_X$ of the polarization $\lambda_X$ (Remark \ref{re:polarization}) induces a perfect symmetric $\sO_\cN$-bilinear pairing
\[
(\;,\;)\colon\rD(X)\times\rD(X)\to \sO_\cN
\]
satisfying $(\iota_X(\alpha)x,y)=(x,\iota_X(\alpha^\tc)y)$ for every $\alpha\in O_E$ and $x,y\in\rD(X)$. As $\Fil(X)$ is a maximal isotropic $\sO_\cN$-submodule of $\rD(X)$ with respect to $(\;,\;)$, we have an induced perfect $\sO_\cN$-bilinear pairing
\[
(\;,\;)\colon\Fil(X)\times\Lie(X)\to\sO_\cN,
\]
under which we denote by $F_X^\perp\subseteq\Fil(X)$ the annihilator of $F_X$. Then the $\sO_\cN$-submodule $F_X^\perp$ is locally free of rank $1$ and is locally a direct summand of $\Fil(X)$.

Following \cite{How19}*{Section~3}, we put
\begin{align*}
\epsilon &\coloneqq u\otimes 1 + 1\otimes u \in O_E\otimes_{O_F}\sO_\cN,\\
\epsilon^\tc &\coloneqq u\otimes 1 - 1\otimes u \in O_E\otimes_{O_F}\sO_\cN.
\end{align*}

\begin{lem}\label{le:linear2}
There are inclusions of $\sO_\cU$-modules $F_X^\perp\subseteq\epsilon\rD(X)\subseteq\rD(X)$ which are locally direct summands. The map $\epsilon\colon\rD(X)\to\epsilon\rD(X)$ descends to a surjective map
\[
\Lie(X)\xrightarrow{\epsilon}\epsilon\rD(X)/F_X^\perp,
\]
whose kernel $L_X$ is locally a direct summand $\sO_\cU$-submodule of $\Lie(X)$ of rank $1$. Moreover, the $O_E$-action stabilizes $L_X$, and acts on $\Lie(X)/L_X$ and $L_X$ via $\varphi_0$ and $\varphi_0^\tc$, respectively.
\end{lem}

\begin{proof}
This follows from the same proof for \cite{How19}*{Proposition~3.3}.
\end{proof}

\begin{definition}\label{de:linear1}
We define the \emph{line bundle of modular forms} $\omega$ to be $L_X^{-1}$, where $L_X$ is the line bundle on $\cN$ from Lemma \ref{le:linear2}.
\end{definition}

For every closed formal subscheme $Z$ of $\cN$, we denote by $\widetilde{Z}$ the closed formal subscheme defined by the sheaf $\sI_Z^2$. Take a nonzero element $x\in\bbV$. By the definition of $\cN(x)$, we have a distinguished morphism
\[
X_0\res_{\cN(x)}\xrightarrow{x} X\res_{\cN(x)}
\]
of $O_F$-divisible groups, which induces an $O_E$-linear map
\[
\rD(X_0)\res_{\cN(x)}\xrightarrow{x}\rD(X)\res_{\cN(x)}
\]
of vector bundles. By the Grothendieck--Messing theory, the above map admits a canonical extension
\[
\rD(X_0)\res_{\widetilde{\cN(x)}}\xrightarrow{\tilde{x}}\rD(X)\res_{\widetilde{\cN(x)}},
\]
which further restricts to a map
\begin{align}\label{eq:linear1}
\Fil(X_0)\res_{\widetilde{\cN(x)}}\xrightarrow{\tilde{x}}\Lie(X)\res_{\widetilde{\cN(x)}}.
\end{align}

From now on, we fix a generator $\gamma$ of the rank $1$ free $O_{\breve{E}}$-module $\Fil(X_0)$.

\begin{lem}\label{le:linear3}
The image $\tilde{x}(\gamma)$ is a section of $L_X$ over $\widetilde{\cN(x)}$, whose vanishing locus coincides with $\cN(x)$, where $\tilde{x}$ is the map \eqref{eq:linear1}.
\end{lem}

\begin{proof}
This follows from the same proof for \cite{How19}*{Proposition~4.1}.
\end{proof}

The following lemma is parallel to \cite{KR11}*{Proposition~3.5}.

\begin{lem}\label{le:linear4}
For every nonzero element $x\in\bbV$, the closed formal subscheme $\cN(x)$ of $\cN$ is either empty or a relative Cartier divisor.
\end{lem}

\begin{proof}
The case $r=1$ has been proved in \cite{RSZ17}*{Proposition~6.6}. Thus, we now assume $r\geq 2$.

We may assume that $\cN(x)$ is nonempty. By the same argument in the proof of \cite{How19}*{Proposition~4.3}, $\cN(x)$ is locally defined by one equation. It remains to show that such equation is not divisible by $u$. Since $r\geq 2$, this follows from \cite{KR11}*{Lemma~3.6}, Lemma \ref{le:rz}, and Corollary \ref{co:bt2}.
\end{proof}

\begin{proof}[Proof of Proposition \ref{pr:linear}]
The proof of \cite{How19}*{Theorem~5.1} can be applied in the same way to Proposition \ref{pr:linear}, using Lemma \ref{le:linear3} and Lemma \ref{le:linear4}.
\end{proof}

To end this subsection, we prove some results that will be used later.

\begin{lem}\label{le:linear5}
The $\sO_\cN$-submodule $L_X$ from Lemma \ref{le:linear2} coincides with the image of the map $\iota_X(u)-u\colon\Lie(X)\to\Lie(X)$.
\end{lem}

\begin{proof}
Denote by $L'_X$ the image of the map $\iota_X(u)-u\colon\Lie(X)\to\Lie(X)$. As we have $L'_X\simeq\Lie(X)/F_X$, $L'_X$ is a locally free $\sO_\cN$-submodule of $\Lie(X)$ of rank $1$ by Lemma \ref{le:linear1}. By the Spin condition in Definition \ref{de:exotic}, for every closed point $s\in\cN(\ol{k})$, the induced map $L'_X\otimes_{\sO_\cN}\ol{k}\to\Lie(X)\otimes_{\sO_\cN}\ol{k}$ over the residue field at $s$ is injective. Thus, the quotient $\sO_\cN$-module $\Lie(X)/L'_X$ is locally free. It remains to show that $L'_X\subseteq L_X$.

By definition, every section of $L'_X$ can be locally written as the image of $(\iota_X(u)-u)x$ for some section $x$ of $\rD(X)$. We need to show that
\begin{enumerate}
  \item $\epsilon(\iota_X(u)-u)x$ is a section of $\Fil(X)$;

  \item $(\epsilon(\iota_X(u)-u)x,y)=0$ for every section $y$ of $F_X$.
\end{enumerate}
For (1), we have $\epsilon(\iota_X(u)-u)x=(\iota_X(u)+u)(\iota_X(u)-u)x=(\iota_X(u^2)-u^2)x$. Since $\iota_X(u^2)-u^2$ acts by zero on $\Lie(X)$, (1) follows. For (2), we have
\[
(\epsilon(\iota_X(u)-u)x,y)=
((\iota_X(u)-u)x,(-\iota_X(u)+u)y)=0
\]
as $y$ is a section of $\Ker(\iota_X(u)-u)$. Thus, (2) follows.

The lemma is proved.
\end{proof}

\begin{lem}\label{le:linear6}
Let $\Lambda$ be a vertex $O_E$-lattice of $\bbV$ with $t(\Lambda)=4$. Then $\omega$ has degree $q-1$ on each connected component of (the smooth projective curve) $\cV_\Lambda$ (Definition \ref{de:bt}).
\end{lem}

\begin{proof}
Let $\delta$ be the Frobenius element of $\Gal(\ol{k}/k)$.

Let $s\in\cN(\ol{k})$ be a closed point represented by the quadruple $(X,\iota_X,\lambda_X;\rho_X)$. Let $\sfM$ be the covariant $O_F$-Dieudonn\'{e} module of $X$ equipped with the $O_E$-action $\iota_X$, which becomes a free $O_{\breve{E}}$-module. We have $\Lie(X)=\sfM/\sfV\sfM$. By Definition \ref{de:linear1} and Lemma \ref{le:linear5}, the fiber $\omega^{-1}\res_s$ is canonically identified with $((u\otimes 1)\sfM+\sfV\sfM)/\sfV\sfM$, which is further canonically isomorphic to $((u\otimes 1)\sfV^{-1}\sfM+\sfM)/\sfM$. By the identification between $\cV_\Lambda$ and the generalized Deligne--Lusztig variety of $\OG(\Lambda^\vee/\Lambda)$ in Proposition \ref{pr:bt} given in \cite{Wu}*{Proposition~4.29}, we know that $\omega^{-1}\res_{\cV_\Lambda}$ coincides with $(\delta(U)+U)/U$ where $U$ is the tautological subbundle of $(\Lambda^\vee/\Lambda)\otimes_k\sO_{\cV_\Lambda}$.

To compute the degree of $(\delta(U)+U)/U$, let $\cV_\Lambda^+$ and $\cV_\Lambda^-$ be the two connected components of $\cV_\Lambda$. Let $\cL_\Lambda$ be the scheme over $\ol{k}$ classifying lines in $\Lambda^\vee/\Lambda$ with the tautological bundle $L$. We may identify $\cV_\Lambda^+$ and $\cV_\Lambda^-$ as two closed subschemes of $\cL_\Lambda$ via the assignment $U\mapsto\delta(U)\cap U$ (see \cite{HP14}*{Section~3.2} for more details). Then, $\cV_\Lambda^+$ and $\cV_\Lambda^-$ are the two irreducible components of the locus where $L$ and $\delta(L)$ generate an isotropic subspace, and the assignment $L\mapsto\delta(L)$ switches $\cV_\Lambda^+$ and $\cV_\Lambda^-$. Let $\cI_\Lambda$ be the locus where $L$ is isotropic and $L=\delta(L)$. Then $\cI_\Lambda$ is a disjoint union of $q^2+1$ copies of $\Spec\ol{k}$ since there are exactly $q^2+1$ isotropic lines in $\Lambda^\vee/\Lambda$, and is contained in $\cV_\Lambda^+\cap\cV_\Lambda^-$. Note that the map $\delta(U)/(\delta(U)\cap U)\to(\delta(U)+U)/U$ is an isomorphism, and there is a short exact sequence
\[
0 \to \delta(\delta(U)\cap U) \to \delta(U)/(\delta(U)\cap U) \to \sO_{\cI_\Lambda} \to 0
\]
of $\sO_{\cV_\Lambda^\pm}$-modules. Since $\delta(U)\cap U$ is the restriction of the tautological bundle $L$ on $\cL_\Lambda$, we have
\begin{align*}
\deg\(\omega^{-1}\res_{\cV_\Lambda^\pm}\)&=\deg\((\delta(U)+U)/U\res_{\cV_\Lambda^\pm}\)=\deg\(\delta(\delta(U)\cap U)\res_{\cV_\Lambda^\pm}\)+(q^2+1) \\
&=\deg\(L^{\otimes q}\res_{\cV_\Lambda^\pm}\)+(q^2+1)=-q\deg(\cV_\Lambda^\pm)+(q^2+1),
\end{align*}
where $\deg(\cV_\Lambda^\pm)$ denotes the degree of the curve $\cV_\Lambda^\pm$ in the projective space $\cL_\Lambda$. Thus, it remains to show that $\deg(\cV_\Lambda^\pm)=q+1$.

To compute the degree, take a $3$-dimensional quadratic subspace $H$ of $\Lambda^\vee/\Lambda$. Let $\cL_\Lambda^H$ be the hyperplane of $\cL_\Lambda$ that consists of lines contained in $H$. Then $\cL_\Lambda^H\cap\cV_\Lambda$ is the subscheme of lines $L\subseteq H$ that is isotropic and fixed by $\delta$, which is a disjoint union of $q+1$ copies of $\Spec\ol{k}$ since there are exactly $q+1$ isotropic lines in $H$. As $\cL_\Lambda^H\cap\cV_\Lambda$ is contained in $\cI_\Lambda$, it is contained in $\cV_\Lambda^+\cap\cV_\Lambda^-$. Therefore, we have $\deg(\cV_\Lambda^\pm)=q+1$.

\if false

To compute the degree, we choose a basis $\{e_1,e_2,f_1,f_2\}$ of $(\Lambda^\vee/\Lambda)\otimes_k\ol{k}$ satisfying $(e_i,e_j)_{\Lambda^\vee/\Lambda}=(f_i,f_j)_{\Lambda^\vee/\Lambda}=0$ and $(e_i,f_j)_{\Lambda^\vee/\Lambda}=\delta_{ij}$ for $1\leq i,j\leq 2$, and such that $\delta$ fixes $e_1$, $f_1$, but switches $e_2$ and $f_2$, under which we identify $(\Lambda^\vee/\Lambda)\otimes_k\ol{k}$ with $\ol{k}^4$ with coordinates $(\mu_1,\mu_2,\nu_1,\nu_2)$. Then $\cV_\Lambda^+$ and $\cV_\Lambda^-$ are the two irreducible components of the subscheme of $\dP^3_{\ol{k}}$ defined by the two equations
\begin{align*}
\mu_1\nu_1+\mu_2\mu_2&=0, \\
\mu_1\nu_1^q+\mu_1^q\nu_1+\mu_2^{q+1}+\nu_2^{q+1}&=0.
\end{align*}
Note that the image of the Segre embedding $\dP^1_{\ol{k}}\times\dP^1_{\ol{k}}\to\dP^3_{\ol{k}}$ sending $([\alpha_1:\alpha_2],[\beta_1:\beta_2])$ to $[\alpha_1\beta_1:\alpha_2\beta_2:\alpha_2\beta_2:-\alpha_1\beta_1]$ gives the locus of the first equation. Then the locus in $\dP^1_{\ol{k}}\times\dP^1_{\ol{k}}$ cut out by the second equation is the union of
\[
([\alpha_1:\alpha_2],[\alpha_1^q:-\alpha_2^q]),\qquad
([\beta_1^q:-\beta_2^q],[\beta_1:\beta_2]),
\]
both being curves of degree $q+1$ in $\dP^3_{\ol{k}}$.

\fi

The lemma is proved.
\end{proof}

\subsection{Proof of Theorem \ref{th:kr} when $r=1$}
\label{ss:kr_proof1}

Let the setup be as in Subsection \ref{ss:kr}. We assume Hypothesis \ref{hy:rz}. In this subsection, we assume $r=1$. Note that since $\bbV$ is nonsplit, the fundamental invariants of an integral $O_E$-lattice of $\bbV$ must consist of two positive odd integers.

\begin{lem}\label{le:two}
Let $\bbL$ be an integral $O_E$-lattice of $\bbV$ with fundamental invariants $(2b_1+1,2b_2+1)$. Then
\[
\partial\Den(\bbL)=2\sum_{j=0}^{b_1}\(1+q+\cdots+q^j+(b_2-j)q^j\).
\]
\end{lem}

\begin{proof}
We denote by $\fL$ the set of integral $O_E$-lattices of $\bbV$ containing $\bbL$. We now count $\fL$.

Fix an orthogonal basis $\{e_1,e_2\}$ of $\bbV$ with $(e_1,e_1)_{\bbV}\in O_F^\times$ and $(e_2,e_2)_{\bbV}\in O_F^\times$ and such that $\bbL=\langle u^{b_1}e_1\rangle+\langle u^{b_2}e_2\rangle$. For every $L\in\fL$, we let $j(L)$ be the unique integer such that $L\cap \langle e_1\rangle\otimes_{O_F}F=\langle u^{j(L)}e_1\rangle$ and let $k(L)$ be the unique integer such that image of $L$ under the natural projection map $\bbV\to\langle e_2\rangle\otimes_{O_F}F$ is $\langle u^{k(L)}e_2\rangle$. Then by Lemma \ref{le:analytic7}(1), $L$ is uniquely determined by $j(L)$, $k(L)$, and the extension map $\varepsilon_L\colon\langle u^{k(L)}e_2\rangle\to \langle u^{j(L)}e_1\rangle\otimes_{O_F}F/O_F$. The condition that $L$ contains $\bbL$ is equivalent to that $j(L)\leq b_1$, $k(L)\leq b_2$, and that $\varepsilon_L$ vanishes on $\langle u^{b_2}e_2\rangle$. Since $\bbL$ is nonsplit, the condition that $L$ is integral is equivalent to that $j(L)\geq 0$, $k(L)\geq 0$, and that the image of $\varepsilon_L$ is contained $\langle e_1\rangle/\langle u^{j(L)}e_1\rangle$. Thus, the number of $L\in\fL$ with $j(L)=j$ for some fixed $0\leq j\leq b_1$ equals $1+q+\cdots+q^j+(b_2-j)q^j$. Summing over all $0\leq j\leq b_1$, we obtain
\[
|\fL|=\sum_{j=0}^{b_1}\(1+q+\cdots+q^j+(b_2-j)q^j\).
\]
The lemma then follows from \eqref{eq:analytic3} as $t(\bbL)=2$.
\end{proof}

\begin{proposition}\label{pr:two}
Theorem \ref{th:kr} holds when $r=1$. More explicitly, for an integral $O_E$-lattice $\bbL$ of $\bbV$ with fundamental invariants $(2b_1+1,2b_2+1)$, we have
\[
\Int(\bbL)=\partial\Den(\bbL)=2\sum_{j=0}^{b_1}\(1+q+\cdots+q^j+(b_2-j)q^j\).
\]
\end{proposition}

\begin{proof}
If $\bbL$ is not integral, then $\Int(\bbL)=\partial\Den(\bbL)=0$. If $\bbL$ is integral with fundamental invariants $(2b_1+1,2b_2+1)$. We may take an orthogonal basis $\{x_1,x_2\}$ of $\bbL$ such that $\val(x_1)=2b_1+1$ and $\val(x_2)=2b_2+1$.

Put $\bbD\coloneqq\End_{O_F}(X_0)\otimes\dQ$, which is a division quaternion algebra over $F$ with the $F$-linear embedding $\iota_{X_0}\colon E\to\bbD$. By the Serre construction, we may naturally identify $\bbD$ with $\bbV$, and we have an identity
\begin{align}\label{eq:two2}
\cN(x_1)=\sum_{j=0}^{b_1}\cW_{\pres{\overline{x_1}}E,j}
\end{align}
of divisors, decomposing the special divisor as a sum of quasi-canonical lifting divisors (see \cite{RSZ17}*{Section~6 \& Proposition~7.1}).

We claim that for every $0\leq j\leq b_1$, the identity
\begin{align}\label{eq:two3}
\length_{O_{\breve{E}}}\cW_{\pres{\overline{x_1}}E,j}\cap\cN(x_2)=2\(1+q+\cdots+q^j+(b_2-l)q^j\)
\end{align}
holds. In fact, this can be proved in the same way as for \cite{KR11}*{Proposition~8.4} using Keating's formula \cite{Vol07}*{Theorem~2.1}. Notice that in \cite{KR11}*{Proposition~8.4} we replace $e_s$ by $2q^j$ since $E/F$ is ramified, and that the factor $2$ comes from the fact that $\cZ_l$ has two connected components. By \eqref{eq:two2} and \eqref{eq:two3}, we have
\[
\Int(\bbL)=\length_{O_{\breve{E}}}\cN(x_1)\cap\cN(x_2)=\sum_{j=0}^{b_1}2\(1+q+\cdots+q^j+(b_2-l)q^j\).
\]
The proposition follows by Lemma \ref{le:two}.
\end{proof}

\begin{definition}\label{de:two}
For $L^\flat\in\flat(\bbV)$, we put
\[
\cN(L^\flat)^\circ\coloneqq\cN(L^\flat)-\cN(u^{-1}L^\flat)
\]
as an effective divisor by (the $r=1$ case of) Lemma \ref{le:linear4}.
\end{definition}

\begin{corollary}\label{co:two}
Take an element $L^\flat\in\flat(\bbV)$. For every $x\in\bbV\setminus V_{L^\flat}$, we have
\[
\length_{O_{\breve{E}}}\cN(L^\flat)^\circ\cap\cN(x)=2\sum_{\substack{L\subseteq L^\vee\\ L\cap V_{L^\flat}=L^\flat}}\CF_L(x).
\]
\end{corollary}

\begin{proof}
By Proposition \ref{pr:two}, we have
\[
\length_{O_{\breve{E}}}\cN(L^\flat)\cap\cN(x)=\Int(L^\flat+\langle x\rangle)=\partial\Den(L^\flat+\langle x\rangle)
=2\sum_{\substack{L\subseteq L^\vee\\ L^\flat\subseteq L\cap V_{L^\flat}}}\CF_L(x),
\]
in which the last identity is due to \eqref{eq:analytic3}. Similarly, we have
\[
\length_{O_{\breve{E}}}\cN(u^{-1}L^\flat)\cap\cN(x)
=2\sum_{\substack{L\subseteq L^\vee\\ u^{-1}L^\flat\subseteq L\cap V_{L^\flat}}}\CF_L(x).
\]
Taking the difference, we obtain the corollary.
\end{proof}

\subsection{Fourier transform of geometric side}
\label{ss:geometric}

Let the setup be as in Subsection \ref{ss:kr}. We assume Hypothesis \ref{hy:rz}. We will freely use notation concerning K-groups of formal schemes from \cite{LL}*{Section~B} and \cite{Zha}*{Appendix~B}, based on the work \cite{GS87}.

\begin{definition}
Let $\cX$ be a formal scheme over $\Spf O_{\breve{E}}$.
\begin{enumerate}
  \item We denote by $\cX^\rh$ the closed formal subscheme of $\cX$ defined by the ideal sheaf $\sO_\cX[p^\infty]$.

  \item We say that an element in $\rK_0(\cX)$ has \emph{proper support} if it belongs to the subgroup $\rK_0^Z(\cX)$ for a proper closed subscheme $Z$ of $\cX$.
\end{enumerate}
\end{definition}

\begin{definition}
Let $X$ be a subset of $\bbV$ such that $\langle X\rangle$ is finitely generated of rank $m$.
\begin{enumerate}
  \item We denote by $\pres{\rK}\cN(X)\in\rK_0(\cN)$ the element $[C(x_1)\otimes_{\sO_\cN}\cdots\otimes_{\sO_\cN} C(x_m)]$ from Subsection \ref{ss:linear} for a basis $\{x_1,\dots,x_m\}$ of the $O_E$-module generated by $X$, which is independent of the choice of the basis by Corollary \ref{co:linear1}.

  \item We denote by $\pres{\rK}\cN(X)^\rh\in\rK_0(\cN)$ the class of $\cN(X)^\rh$.

  \item We put $\pres{\rK}\cN(X)^\rv\coloneqq\pres{\rK}\cN(X)-\pres{\rK}\cN(X)^\rh\in\rK_0(\cN)$.
\end{enumerate}
\end{definition}

\begin{lem}\label{le:geometric4}
Let $L^\flat$ be an element of $\flat(\bbV)$ (Definition \ref{de:analytic}). We have
\begin{enumerate}
  \item $\cN(L^\flat)^\rh$ is either empty or finite flat over $\Spf O_{\breve{E}}$;

  \item all of $\pres{\rK}\cN(L^\flat)$, $\pres{\rK}\cN(L^\flat)^\rh$, and $\pres{\rK}\cN(L^\flat)^\rv$ belong to $\rF^{n-1}\rK_0(\cN)$;

  \item $\pres{\rK}\cN(L^\flat)^\rv$ has proper support.
\end{enumerate}
\end{lem}

\begin{proof}
Part (1) follows from Lemma \ref{le:geometric2} and Lemma \ref{le:geometric1}.

Take a basis $\{x_1,\dots,x_{n-1}\}$ of the $O_E$-module $L^\flat$.

For (2), it suffices to show $\pres{\rK}\cN(L^\flat)\in\rF^{n-1}\rK_0(\cN)$ by (1). By definition, $\pres{\rK}\cN(L^\flat)$ is the cup product of the classes in $\rK_0(\cN)$ of $\cN(x_1),\dots,\cN(x_{n-1})$, each being a divisor by Lemma \ref{le:linear4}. Thus, $\pres{\rK}\cN(L^\flat)$ belongs to $\rF^{n-1}\rK_0(\cN)$ by (the analogue for formal schemes of) \cite{GS87}*{Proposition~5.5}.

For (3), by the same argument for \cite{LZ}*{Lemma~5.1.1}, we know that there exists a proper closed subscheme $Z$ of $\cN$ such that $\cN(L^\flat)$ is contained in $\cN(L^\flat)^\rh\bigcup Z$. By (1), the difference
\[
\pres{\rK}\cN(x_1)^\rh.\cdots.\pres{\rK}\cN(x_{n-1})^\rh-\pres{\rK}\cN(L^\flat)^\rh
\]
belongs to $\rK^Z_0(\cN)$. By definition,
\begin{align*}
\pres{\rK}\cN(L^\flat)^\rv&=\pres{\rK}\cN(x_1).\cdots.\pres{\rK}\cN(x_{n-1})-\pres{\rK}\cN(L^\flat)^\rh \\
&=(\pres{\rK}\cN(x_1)^\rv+\pres{\rK}\cN(x_1)^\rh).\cdots.(\pres{\rK}\cN(x_{n-1})^\rv+\pres{\rK}\cN(x_{n-1})^\rh)-\pres{\rK}\cN(L^\flat)^\rh \\
&=\cC+\(\pres{\rK}\cN(x_1)^\rh.\cdots.\pres{\rK}\cN(x_{n-1})^\rh-\pres{\rK}\cN(L^\flat)^\rh\),
\end{align*}
where
\[
\cC\coloneqq\sum_{i=1}^{n-1}\pres{\rK}\cN(x_i)^\rv.\pres{\rK}\cN(x_1).\cdots.\widehat{\pres{\rK}\cN(x_i)}.\cdots.\pres{\rK}\cN(x_{n-1})
\]
belongs to $\rK^Z_0(\cN)$. Thus, $\pres{\rK}\cN(L^\flat)^\rv$ belongs to $\rK^Z_0(\cN)$ hence has proper support.
\end{proof}

\begin{definition}\label{de:geometric}
Let $L^\flat$ be an element of $\flat(\bbV)$ (Definition \ref{de:analytic}). For $x\in\bbV\setminus V_{L^\flat}$, we put
\begin{align*}
\Int_{L^\flat}(x)&\coloneqq\pres{\rK}\cN(L^\flat).\pres{\rK}\cN(x), \\
\Int_{L^\flat}^\rh(x)&\coloneqq\pres{\rK}\cN(L^\flat)^\rh.\pres{\rK}\cN(x), \\
\Int_{L^\flat}^\rv(x)&\coloneqq\pres{\rK}\cN(L^\flat)^\rv.\pres{\rK}\cN(x).
\end{align*}
Here, the intersection numbers are well-defined since $\cN(L^\flat)\cap\cN(x)$ is a proper closed subscheme of $\cN$ by Remark \ref{re:special}. Note that $\Int_{L^\flat}(x)=\Int(L^\flat+\langle x\rangle)$ (Definition \ref{de:int}).
\end{definition}

The following is our main result of this subsection.

\begin{proposition}\label{pr:geometric}
Let $L^\flat$ be an element of $\flat(\bbV)$ (Definition \ref{de:analytic}).
\begin{enumerate}
  \item We have $\Int_{L^\flat}^\rh(x)=\partial\Den_{L^\flat}^\rh(x)$ for $x\in\bbV\setminus V_{L^\flat}$, where $\partial\Den_{L^\flat}^\rh$ is from Definition \ref{de:density1}.

  \item The function $\Int_{L^\flat}^\rv$ extends (uniquely) to a (compactly supported) locally constant function on $\bbV$, which we still denote by $\Int_{L^\flat}^\rv$. Moreover, we have
      \[
      \widehat{\Int_{L^\flat}^\rv}=-\Int_{L^\flat}^\rv.
      \]
      In particular, the support of $\widehat{\Int_{L^\flat}^\rv}$ is contained in $\bbV^\integ$ (Definition \ref{de:analytic1}).
\end{enumerate}
\end{proposition}

The rest of this subsection is devoted to the proof of this proposition.

\begin{remark}[Cancellation law for special cycles]\label{re:cancellation}
Let $\bbV'$ be a hermitian subspace of $\bbV$ that is nonsplit and of positive even dimension $n'$. Let $L$ be an integral hermitian $O_E$-module contained in $\bbV$ such that $L\cap\bbV^{\prime\perp}$ is a self-dual $O_E$-lattice of $\bbV^{\prime\perp}$. We may choose
\begin{itemize}
  \item an object $(\bbX',\iota_{\bbX'},\lambda_{\bbX'})\in\Exo_{(n'-1,1)}^\rb(\ol{k})$ (Definition \ref{de:exotic}),

  \item an object $(Y,\iota_Y,\lambda_Y)\in\Exo_{(n-n',0)}(O_{\breve{E}})$ (Remark \ref{re:exotic}),\footnote{When $n'=n$, we simply ignore $(Y,\iota_Y,\lambda_Y)$.}

  \item a quasi-morphism $\varrho$ from $(Y,\iota_Y,\lambda_Y)\otimes_{O_{\breve{E}}}\ol{k}\oplus(\bbX',\iota_{\bbX'},\lambda_{\bbX'})$ to $(\bbX,\iota_{\bbX},\lambda_{\bbX})$ in the category $\Exo_{(n-1,1)}^\rb(S\otimes_{O_{\breve{E}}}\ol{k})$ satisfying
      \begin{itemize}
        \item $\varrho$ identifies $\Hom_{O_E}(X_0\otimes_{O_{\breve{E}}}\ol{k},\bbX')\otimes\dQ$ with $\bbV'$ as hermitian spaces;

        \item $\varrho$ identifies $\Hom_{O_E}(X_0\otimes_{O_{\breve{E}}}\ol{k},Y\otimes_{O_{\breve{E}}}\ol{k})$ with $L\cap\bbV^{\prime\perp}$ as hermitian $O_E$-modules.
      \end{itemize}
\end{itemize}
Let $\cN'\coloneqq\cN_{(\bbX',\iota_{\bbX'},\lambda_{\bbX'})}$ be the relative Rapoport--Zink space for the triple $(\bbX',\iota_{\bbX'},\lambda_{\bbX'})$ (Definition \ref{de:rz}). We have a morphism $\cN'\to\cN$ such that for every object $S$ of $\Sch_{/O_{\breve{E}}}^\rv$, $\cN(S)$ it sends an object $(X',\iota_{X'},\lambda_{X'};\rho_{X'})\in\cN'(S)$ to the object
\[
(Y\otimes_{O_{\breve{E}}}S\oplus X',\iota_Y\otimes_{O_{\breve{E}}}S\oplus\iota_{X'},
\lambda_Y\otimes_{O_{\breve{E}}}S\oplus\lambda_{X'};\varrho\circ(\id_Y\otimes_{O_{\breve{E}}}S\oplus\rho_{X'}))\in\cN(S).
\]
We have
\begin{enumerate}
  \item The morphism $\cN'\to\cN$ above identifies $\cN'$ with the closed formal subscheme $\cN(L\cap\bbV^{\prime\perp})$ of $\cN$.

  \item Suppose that $L\cap\bbV'\neq\{0\}$, then $\cN(L)$ coincides with the image of $\cN'(L\cap\bbV')$ under the morphism $\cN'\to\cN$ above.

  \item For a nonzero element $x\in\bbV$ written as $x=y+x'$ with respect to the orthogonal decomposition $\bbV=\bbV^{\prime\perp}\oplus\bbV'$, we have
      \[
      \cN'\times_{\cN}\cN(x)=
      \begin{dcases}
      \emptyset,&\text{if $y\not\in L\cap\bbV^{\prime\perp}$,} \\
      \cN',&\text{if $y\in L\cap\bbV^{\prime\perp}$ and $x'=0$,} \\
      \cN'(x'),&\text{if $y\in L\cap\bbV^{\prime\perp}$ and $x'\neq 0$.}
      \end{dcases}
      \]

  \item If $L$ is an $O_E$-lattice of $\bbV$, then we have $\Int(L)=\Int(L\cap\bbV')$.
\end{enumerate}
\end{remark}

\begin{lem}\label{le:geometric1}
Let $L^{\flat\prime}\in\flat(\bbV)$ be an element that is integral and satisfies $t(L^{\flat\prime})=1$.
\begin{enumerate}
  \item The formal subscheme $\cN(L^{\flat\prime})$ is finite flat over $\Spf O_{\breve{E}}$.

  \item If we put $\cN(L^{\flat\prime})^\circ\coloneqq\cN(L^{\flat\prime})-\cN(L^{\flat\prime}_-)$ as an element in $\rF^{n-1}\rK_0(\cN)$, then for every $x\in\bbV\setminus V_{L^\flat}$,
      \[
      \cN(L^{\flat\prime})^\circ.\pres{\rK}\cN(x)=
      2\sum_{\substack{L\subseteq L^\vee\\ L\cap V_{L^\flat}=L^{\flat\prime}}}\CF_L(x).
      \]
      Here, $L^{\flat\prime}_-$ is the unique element in $\flat(\bbV)$ satisfying $L^{\flat\prime}\subseteq L^{\flat\prime}_-\subseteq(L^{\flat\prime})^\vee$ with $|L^{\flat\prime}_-/L^{\flat\prime}|=q$ (so that $L^{\flat\prime}_-$ is either not integral, or is integral with $t(L^{\flat\prime}_-)=1$).
\end{enumerate}
\end{lem}

\begin{proof}
Since $t(L^{\flat\prime})=1$, we may choose a $2$-dimensional (nonsplit) hermitian subspace $\bbV'$ of $\bbV$ such that $L^{\flat\prime}\cap\bbV^{\prime\perp}$ is a self-dual $O_E$-lattice of $\bbV^{\prime\perp}$. We adopt the construction in Remark \ref{re:cancellation}.

For (1), we have $\cN(L^{\flat\prime})=\cN'(L^{\flat\prime}\cap\bbV')$, which is finite flat over $\Spf O_{\breve{E}}$ by (the $r=1$ case of) Lemma \ref{le:linear4}.

For (2), we write $x=y+x'$ with respect to the orthogonal decomposition $\bbV=\bbV^{\prime\perp}\oplus\bbV'$. Since $x\not\in V_{L^\flat}$, we have $x'\neq 0$. By Remark \ref{re:cancellation}(2), $\cN(L^{\flat\prime})^\circ$ coincides with (the class of) $\cN'(L^{\flat\prime}\cap\bbV')^\circ$ in $\rF^1\rK_0(\cN')$ under the map $\rF^1\rK_0(\cN')\to\rF^{n-1}\rK_0(\cN)$. There are two cases.

If $y\not\in L^{\flat\prime}\cap\bbV^{\prime\perp}$, then $\cN(L^{\flat\prime})^\circ.\pres{\rK}\cN(x)=0$ by Remark \ref{re:cancellation}(3), and there is no integral $O_E$-lattice of $\bbV$ containing $L^{\flat\prime}+\langle x\rangle$. Thus, (2) follows.

If $y\in L^{\flat\prime}\cap\bbV^{\prime\perp}$, then by Remark \ref{re:cancellation}(3), we have
\[
\cN(L^{\flat\prime})^\circ.\pres{\rK}\cN(x)=\cN'(L^{\flat\prime}\cap\bbV')^\circ.\pres{\rK}\cN'(x')
=\length_{O_{\breve{E}}}\cN'(L^{\flat\prime}\cap\bbV')^\circ\cap\cN'(x').
\]
By Corollary \ref{co:two}, we have
\[
\length_{O_{\breve{E}}}\cN'(L^{\flat\prime}\cap\bbV')^\circ\cap\cN'(x')=
2\sum_{\substack{L'\subseteq L'^\vee(\subseteq\bbV')\\ L'\cap(V_{L^\flat}\cap\bbV')=L^{\flat\prime}\cap\bbV'}}\CF_{L'}(x')
=2\sum_{\substack{L\subseteq L^\vee\\ L\cap V_{L^\flat}=L^{\flat\prime}}}\CF_L(x).
\]
Thus, (2) follows.
\end{proof}

\begin{lem}\label{le:geometric2}
Let $L^\flat$ be an element of $\flat(\bbV)$ (Definition \ref{de:analytic}). We have
\[
\cN(L^\flat)^\rh=\bigcup_{\substack{L^\flat\subseteq L^{\flat\prime}\subseteq(L^{\flat\prime})^\vee \\ t(L^{\flat\prime})=1}}
\cN(L^{\flat\prime})^\circ
\]
as closed formal subschemes of $\cN$, and the identity
\[
\pres{\rK}\cN(L^\flat)^\rh=\sum_{\substack{L^\flat\subseteq L^{\flat\prime}\subseteq(L^{\flat\prime})^\vee \\ t(L^{\flat\prime})=1}}
\cN(L^{\flat\prime})^\circ
\]
in $\rF^{n-1}\rK_0(\cN)/\rF^n\rK_0(\cN)$, where $\cN(L^{\flat\prime})^\circ$ is introduced in Lemma \ref{le:geometric1}(2).
\end{lem}

\begin{proof}
This lemma can be proved by the same way as for \cite{LZ}*{Theorem~4.2.1}, as long as we establish the following claim replacing \cite{LZ}*{Lemma~4.5.1} in the case where $E/F$ is ramified.
\begin{itemize}
  \item Let $L$ be a self-dual hermitian $O_E$-module of rank $n$ and $L^\flat$ a hermitian $O_E$-module contained in $L$. If $L/L^\flat$ is free, then $L^\flat$ is integral with $t(L^\flat)=1$.
\end{itemize}
However, this is just a special case of Lemma \ref{le:analytic7}(2).
\end{proof}

\begin{lem}\label{le:geometric3}
Let $\Lambda$ be a vertex $O_E$-lattice of $\bbV$ with $t(\Lambda)=4$. Take an arbitrary connected component $\cV_\Lambda^+$ of the smooth projective curve $\cV_\Lambda$ from Proposition \ref{pr:bt}, regarded as an element in $\rF^{n-1}\rK_0(\cN)$. For every $x\in\bbV\setminus\{0\}$, put $\Int_{\cV_\Lambda^+}(x)\coloneqq\cV_\Lambda^+.\pres{\rK}\cN(x)$. Then $\Int_{\cV_\Lambda^+}$ extends (uniquely) to a compactly supported locally constant function on $\bbV$, which we still denote by $\Int_{\cV_\Lambda^+}$. Moreover, we have
\[
\widehat{\Int_{\cV_\Lambda^+}}=-\Int_{\cV_\Lambda^+}.
\]
\end{lem}

\begin{proof}
Since $t(\Lambda)=4$, we may choose a $4$-dimensional (nonsplit) hermitian subspace $\bbV'$ of $\bbV$ such that $\Lambda\cap\bbV^{\prime\perp}$ is a self-dual $O_E$-lattice of $\bbV^{\prime\perp}$. We adopt the construction in Remark \ref{re:cancellation}. Write $x=y+x'$ with respect to the orthogonal decomposition $\bbV=\bbV^{\prime\perp}\oplus\bbV'$. Put $\Lambda'\coloneqq\Lambda\cap\bbV'$. By Remark \ref{re:cancellation}(2) and Definition \ref{de:bt}(2), $\cV_\Lambda$ coincides with $\cV_{\Lambda'}$ under the natural morphism $\cN'\to\cN$. Denote by $\cV_{\Lambda'}^+$ the connected component of $\cV_{\Lambda'}$ that corresponds to $\cV_\Lambda^+$. By Remark \ref{re:cancellation}(3), we have
\[
\cV_\Lambda^+.\pres{\rK}\cN(x)=
\begin{dcases}
0,&\text{if $y\not\in\Lambda\cap\bbV^{\prime\perp}$,} \\
\cV_{\Lambda'}^+.\pres{\rK}\cN'(x'),&\text{if $y\in\Lambda\cap\bbV^{\prime\perp}$.}
\end{dcases}
\]
In other words, we have $\Int_{\cV_\Lambda^+}=\CF_{\Lambda\cap\bbV^{\prime\perp}}\otimes\Int_{\cV_{\Lambda'}^+}$. Therefore, it suffices to consider the case where $n=4$.

We now give an explicit formula for $\Int_{\cV_\Lambda^+}(x)$ when $n=4$. Let $\cN^+$ be the connected component of $\cN$ that contains $\cV_\Lambda^+$, and put $Z^+\coloneqq Z\cap\cN^+$ for every formal subscheme $Z$ of $\cN$. Put $\Lambda(x)\coloneqq\Lambda+\langle x\rangle$. There are three cases.

\begin{enumerate}
  \item Suppose that $\Lambda(x)$ is not integral. By Corollary \ref{co:bt1}, $\cV_\Lambda$ has empty intersection with $\cN(x)$. Thus, we have $\Int_{\cV_\Lambda^+}(x)=0$.

  \item Suppose that $\Lambda(x)$ is integral but $x\not\in\Lambda$. Then $\Lambda(x)$ has fundamental invariants $(0,0,1,1)$. By Corollary \ref{co:bt1}, $\cV_\Lambda^+\cap\cN(x)_\red=\cV_{\Lambda(x)}^+$ which is a $\ol{k}$-point. Thus, we have $\Int_{\cV_\Lambda^+}(x)\geq 1$. Choose a normal basis (Remark \ref{re:analytic3}) $\{x_1,x_2,x_3,x_4\}$ of $\Lambda$ and write $x=\lambda_1x_1+\lambda_2x_2+\lambda_3x_3+\lambda_4x_4$ with $\lambda_i\in E$. Without lost of generality, we may assume $\lambda_4\not\in O_E$. Since $ux\in\Lambda$, we have $\Lambda(x)=\langle x_1,x_2,x_3,x\rangle$. By Corollary \ref{co:bt1}, $\cN(x_1)\cap\cN(x_2)\cap\cN(x_3)$ contains $\cV_\Lambda$ as a closed subscheme. By Remark \ref{re:cancellation} and Proposition \ref{pr:two} applied to $\bbV'$ spanned by $x_3$ and $x_4$, $\cN(\Lambda(x))$ is a $0$-dimensional scheme and $\Int(\Lambda(x))=2$. It follows that
      \[
      \Int_{\cV_\Lambda^+}(x)\leq\length_{O_{\breve{E}}}\(\cN(x_1)\cap\cN(x_2)\cap\cN(x_3)\)\cap\cN(x)^+
      =\Int^+(\Lambda(x))=1
      \]
      by Lemma \ref{le:geometric5} below. Thus, we obtain $\Int^+(\Lambda(x))=1$ hence $\Int_{\cV_\Lambda^+}(x)=1$.

  \item Suppose that $x\in\Lambda$. Then $\cV_\Lambda^+$ is a closed subscheme of $\cN(x)$, which implies
      \[
      \sO_{\cV_{\Lambda^+}}\overset{\dL}\otimes_{\sO_\cN}\sO_{\cN(x)}=
      \(\sO_{\cV_{\Lambda^+}}\overset{\dL}\otimes_{\sO_{\cN(x)}}\sO_{\cN(x)}\)\overset{\dL}\otimes_{\sO_\cN}\sO_{\cN(x)}=
      \sO_{\cV_{\Lambda^+}}\overset{\dL}\otimes_{\sO_{\cN(x)}}\(\sO_{\cN(x)}\overset{\dL}\otimes_{\sO_\cN}\sO_{\cN(x)}\).
      \]
      However, by Corollary \ref{co:linear1}, we have $\sO_{\cN(x)}\overset{\dL}\otimes_{\sO_\cN}\sO_{\cN(x)}=\sO_{\cN(x)}\otimes_{\sO_\cN}C(0)$ in $\rK_0(\cN)$, where $C(0)$ is the complex \eqref{eq:linear}. Thus, we obtain
      \[
      \Int_{\cV_\Lambda^+}(x)=\chi\(C(0)\res_{\cV_\Lambda^+}\)=\deg\(\sO_{\cV_\Lambda^+}\)-\deg\(\omega\res_{\cV_\Lambda^+}\)
      =-\deg\(\omega\res_{\cV_\Lambda^+}\)=1-q
      \]
      by Lemma \ref{le:linear6}.
\end{enumerate}

Since there are exactly $q^2+1$ vertex $O_E$-lattices of $\bbV$ properly containing $\Lambda$, combining (1--3), we obtain
\[
\Int_{\cV_\Lambda^+}=-q(1+q)\CF_\Lambda+\sum_{\Lambda\varsubsetneq\Lambda'\subseteq\Lambda'^\vee}\CF_{\Lambda'}.
\]
It follows that
\begin{align}\label{eq:geometric}
\widehat{\Int_{\cV_\Lambda^+}}=-\frac{1+q}{q}\CF_{\Lambda^\vee}+
\frac{1}{q}\sum_{\Lambda\varsubsetneq\Lambda'\subseteq\Lambda'^\vee}\CF_{\Lambda'^\vee}.
\end{align}
\begin{itemize}
  \item If $x\in\Lambda$, then $\widehat{\Int_{\cV_\Lambda^+}}(x)=-\frac{1+q}{q}+\frac{q^2+1}{q}=q-1$.

  \item If $\Lambda(x)$ is integral but $x\not\in\Lambda$, then the number of $\Lambda'$ in the summation of \eqref{eq:geometric} such that $x\in\Lambda'^\vee$ is exactly $1$ (namely, $\Lambda(x)$ itself). Thus, we have $\widehat{\Int_{\cV_\Lambda^+}}(x)=-\frac{1+q}{q}+\frac{1}{q}=-1$.

  \item If $\Lambda(x)$ is not integral but $x\in\Lambda^\vee$, then the set of $\Lambda'$ in the summation of \eqref{eq:geometric} satisfying $x\in\Lambda'^\vee$ is bijective to the set of isotropic lines in $\Lambda^\vee/\Lambda$ perpendicular to $x$. Now since $\Lambda(x)$ is not integral, $x$ is anisotropic in $\Lambda^\vee/\Lambda$, which implies that the previous set has cardinality $q+1$. Thus, we have $\widehat{\Int_{\cV_\Lambda^+}}(x)=-\frac{1+q}{q}+\frac{q+1}{q}=0$.

  \item If $x\not\in\Lambda^\vee$, then $\widehat{\Int_{\cV_\Lambda^+}}(x)=0$.
\end{itemize}
Therefore, we have $\widehat{\Int_{\cV_\Lambda^+}}=-\Int_{\cV_\Lambda^+}$. The lemma is proved.
\end{proof}

\begin{lem}\label{le:geometric5}
Denote the two connected components of $\cN$ by $\cN^+$ and $\cN^-$, and $\Int^\pm(\bbL)$ the intersection multiplicity in Definition \ref{de:int} on $\cN^\pm$. Then
\[
\Int^+(\bbL)=\Int^-(\bbL)=\tfrac{1}{2}\Int(\bbL).
\]
\end{lem}

\begin{proof}
Choose a normal basis (Remark \ref{re:analytic3}) $\{x_1,\dots,x_n\}$ of $\bbL$. Since $\bbV$ is nonsplit, there exists an anisotropic element in the basis, say $x_n$. Let $\theta$ the unique element in $\rU(\bbV)(F)$ satisfying $\theta(x_i)=1$ for $1\leq i\leq n-1$ and $\theta(x_n)=-x_n$. Then $\theta$ induces an automorphism of $\cN$, preserving $\cN(x_i)$ for $1\leq i\leq n$, but switching $\cN^+$ and $\cN^-$ as $\det\theta=-1$. Thus, we have $\Int^+(\bbL)=\Int^-(\bbL)$. Since $\Int(\bbL)=\Int^+(\bbL)+\Int^-(\bbL)$, the lemma follows.
\end{proof}

\begin{proof}[Proof of Proposition \ref{pr:geometric}]
We first consider (1). By Lemma \ref{le:geometric2}, we have for $x\in\bbV\setminus V_{L^\flat}$,
\[
\Int_{L^\flat}^\rh(x)=\sum_{\substack{L^\flat\subseteq L^{\flat\prime}\subseteq(L^{\flat\prime})^\vee \\ t(L^{\flat\prime})=1}}
\cN(L^{\flat\prime})^\circ.\pres{\rK}\cN(x),
\]
which, by Lemma \ref{le:geometric1}, equals
\[
2\sum_{\substack{L^\flat\subseteq L^{\flat\prime}\subseteq(L^{\flat\prime})^\vee \\ t(L^{\flat\prime})=1}}
\sum_{\substack{L\subseteq L^\vee\\ L\cap V_{L^\flat}=L^{\flat\prime}}}\CF_L(x)
=2\sum_{\substack{L^\flat\subseteq L\subseteq L^\vee \\ t(L\cap V_{L^\flat})=1}}\CF_L(x).
\]
Thus, Proposition \ref{pr:geometric}(1) follows from Definition \ref{de:density1}.

We first consider (2). We may assume $r\geq 2$ since otherwise $\Int_{L^\flat}^\rv\equiv 0$ hence (2) is trivial. We write $\cN=\cN^+\cup\cN^-$ for the two connected components. For every vertex $O_E$-lattice $\Lambda$ of $\bbV$, we put $\cV_\Lambda^\pm\coloneqq\cV_\Lambda\cap\cN^\pm$. By Lemma \ref{le:geometric4} and Proposition \ref{pr:bt}, there exists finitely finitely many vertex $O_E$-lattices $\Lambda_1,\dots,\Lambda_m$ of $\bbV$ of type $n$ such that
\[
\pres{\rK}\cN(L^\flat)^\rv \in \sum_{i=1}^m\rF^{n-1}\rK_0^{\cV_{\Lambda_i}}(\cN)\subseteq\rF^{n-1}\rK_0(\cN).
\]
Since the natural map $\rF^{\frac{t(\Lambda_i)}{2}-2}\rK_0(\cV_{\Lambda_i})\to\rF^{n-1}\rK_0^{\cV_{\Lambda_i}}(\cN)$ is an isomorphism for $1\leq i\leq m$, by Lemma \ref{le:tate}, there exist rational numbers $c_\Lambda^\pm$ for vertex $O_E$-lattices $\Lambda$ of $\bbV$ with $t(\Lambda)=4$, of which all but finitely many are zero, such that
\[
\pres{\rK}\cN(L^\flat)^\rv - \(\sum_{\Lambda}c_\Lambda^+\cdot\cV_\Lambda^++c_\Lambda^-\cdot\cV_\Lambda^-\)
\]
has zero intersection with $\rF^1\rK_0(\cN)$. Thus, Proposition \ref{pr:geometric}(2) follows from Lemma \ref{le:geometric3}.
\end{proof}

\subsection{Proof of Theorem \ref{th:kr}}
\label{ss:kr_proof}

Let the setup be as in Subsection \ref{ss:kr}. In this subsection, for an element $L^\flat\in\flat(\bbV)$ (Definition \ref{de:analytic}), we set $\val(L^\flat)=-1$ if $L^\flat$ is not integral.

\begin{lem}\label{le:kr1}
Suppose that $r\geq 2$ and take an integral element $L^\flat\in\flat(\bbV)$ whose fundamental invariants $(a_1,\dots,a_{n-2},a_{n-1})$ satisfy $a_{n-2}<a_{n-1}$ (in particular, $a_{n-1}$ is odd). Then the number of integral $O_E$-lattices of $\bbV$ containing $L^\flat$ with fundamental invariants $(a_1,\dots,a_{n-2},a_{n-1}-1,a_{n-1}-1)$ is either $0$ or $2$. When the number is $2$ and those lattices are denoted by $L^{\flat+}$ and $L^{\flat-}$, we have
\begin{enumerate}
  \item $L^{\flat\pm}\cap V_{L^\flat}=L^\flat$;

  \item $a_{n-1}\geq 3$;

  \item there are orthogonal decompositions $L^\flat=L^\flat_\leftarrow\oplus L^\flat_\rightarrow$ and $L^{\flat\pm}=L^\flat_\leftarrow\oplus L^{\flat\pm}_\rightarrow$, in which $L^\flat_\leftarrow$, $L^\flat_\rightarrow$, and $L^{\flat\pm}_\rightarrow$ are integral hermitian $O_E$-modules with fundamental invariants $(a_1,\dots,a_{n-2})$, $(a_{n-1})$, and $(a_{n-1}-1,a_{n-1}-1)$, respectively.
\end{enumerate}
\end{lem}

\begin{proof}
Let $L$ be an integral $O_E$-lattice $L$ of $\bbV$ containing $L^\flat$ with fundamental invariants $(a_1,\dots,a_{n-2},a_{n-1}-1,a_{n-1}-1)$.

We first claim that (1) must hold. We have $\val(L\cap V_{L^\flat})\geq a_1+\cdots+a_{n-2}+a_{n-1}-1$ by Lemma \ref{le:analytic7}(1). Since $L\cap V_{L^\flat}$ contains $L^\flat$ and $\val(L\cap V_{L^\flat})$ is odd, we must have $L\cap V_{L^\flat}=L^\flat$.

Choose a normal basis $(e_1,\dots,e_{n-1})$ of $L^\flat$ (Remark \ref{re:analytic3}), and rearrange them such that for every $1\leq i\leq n-1$, exactly one of the following three happens:
\begin{enumerate}[label=(\alph*)]
  \item $(e_i,e_i)_{\bbV}=\beta_i u^{a_i-1}$ for some $\beta_i\in O_F^\times$;

  \item $(e_i,e_{i+1})_{\bbV}=u^{a_i-1}$;

  \item $(e_i,e_{i-1})_{\bbV}=-u^{a_i-1}$.
\end{enumerate}
By the claim on (1), we may write $L=L^\flat+\langle x\rangle$ in which
\[
x=\lambda_1 e_1+\cdots+\lambda_{n-1}e_{n-1}+x_n
\]
for some $\lambda_i\in(E\setminus O_E)\cup\{0\}$ and $0\neq x_n\in V_{L^\flat}^\perp$. Let $T$ be the moment matrix with respect to the basis $\{e_1,\dots,e_{n-1},x\}$ of $L$.

We show by induction that for $1\leq i\leq n-2$, $\lambda_i=0$. Suppose we know $\lambda_1=\cdots\lambda_{i-1}=0$. For $\lambda_i$ (with $1\leq i\leq n-2$), there are three cases.
\begin{itemize}
  \item If $e_i$ is in the situation (a) above, then applying Lemma \ref{le:analytic7}(1) to the $i$-by-$i$ minor of $T$ consisting of rows $\{1,\dots,i\}$ and columns $\{1,\dots,i-1,n\}$, we obtain $\val_E(\lambda_i\beta_i u^{a_i-1})\geq a_i-1$, which implies $\lambda_i=0$.

  \item If $e_i$ is in the situation (b) above, then applying Lemma \ref{le:analytic7}(1) to the $i$-by-$i$ minor of $T$ consisting of rows $\{1,\dots,i-1,i+1\}$ and columns $\{1,\dots,i-1,n\}$, we obtain $\val_E(-\lambda_i u^{a_i-1})\geq a_i-1$, which implies $\lambda_i=0$.

  \item If $e_i$ is in the situation (c) above, then applying Lemma \ref{le:analytic7}(1) to the $i$-by-$i$ minor of $T$ consisting of rows $\{1,\dots,i\}$ and columns $\{1,\dots,i-1,n\}$, we obtain $\val_E(\lambda_i u^{a_i-1})\geq a_i-1$, which implies $\lambda_i=0$.
\end{itemize}

Note that $e_{n-1}$ is in the situation (a). Applying Lemma \ref{le:analytic7}(1) to the $(n-1)$-by-$(n-1)$ minor of $T$ consisting of rows $\{1,\dots,n-1\}$ and columns $\{1,\dots,n-2,n\}$, we obtain $\val_E(\lambda_{n-1}\beta_{n-1}u^{a_{n-1}-1})\geq a_{n-1}-2$, which implies $\lambda_{n-1}\in u^{-1}O_E$. On the other hand, $\lambda_{n-1}\neq 0$ since otherwise $a_{n-1}$ will appear in the fundamental invariants of $L$, which is a contradiction. Thus, we have $\lambda_{n-1}\in u^{-1}O_E\setminus O_E$. After rescaling by an element in $O_E^\times$, we may assume $\lambda_{n-1}=u^{-1}$. Applying Lemma \ref{le:analytic7}(1) to the $(n-1)$-by-$(n-1)$ minor of $T$ consisting of rows $\{1,\dots,n-2,n\}$ and columns $\{1,\dots,n-2,n\}$, we obtain
\begin{align}\label{eq:kr1}
\val_E\((x_n,x_n)_{\bbV}-u^{-2}\beta_{n-1}u^{a_{n-1}-1}\)\geq a_{n-1}-2.
\end{align}
We note the following facts.
\begin{itemize}
  \item The set of $x_n\in V_{L^\flat}^\perp$ satisfying \eqref{eq:kr1} is stable under the multiplication by $1+uO_E$.

  \item The set of orbits of such $x_n$ under the multiplication by $1+uO_E$ is bijective to the set of $L$.

  \item The number of orbits is either $0$ or $2$.

  \item If the number is $2$, then $a_{n-1}\geq 3$, since $\bbV$ is nonsplit.
\end{itemize}
Thus, the main part of the lemma is proved, with the properties (1) and (2) included. For (3), we simply take $L^\flat_\leftarrow=\langle e_1,\dots,e_{n-2}\rangle$ with $L^\flat_\rightarrow$ and $L^{\flat\pm}_\rightarrow$ uniquely determined.

The lemma is proved.
\end{proof}

In the rest of subsection, we say that $L^\flat$ is \emph{special} if $L^\flat$ is like in Lemma \ref{le:kr1} for which the number is $2$.
We now define an open compact subset $S_{L^\flat}$ of $\bbV$ for an integral element $L^\flat\in\flat(\bbV)$ in the following way:
\[
S_{L^\flat}\coloneqq
\begin{dcases}
L^{\flat+}\cup L^{\flat-}, &\text{if $L^\flat$ is special,}\\
L^\flat+(V_{L^\flat}^\perp)^\integ, &\text{if $L^\flat$ is not special.}
\end{dcases}
\]

\begin{lem}\label{le:kr2}
Take an integral element $L^\flat\in\flat(\bbV)$. Then for every $x\in\bbV\setminus(V_{L^\flat}\cup S_{L^\flat})$, we may write
\[
L^\flat+\langle x\rangle = L^{\flat\prime}+\langle x'\rangle
\]
for some $L^{\flat\prime}\in\flat(\bbV)$ satisfying $\val(L^{\flat\prime})<\val(L^\flat)$.
\end{lem}

\begin{proof}
Take an element $x\in\bbV\setminus(V_{L^\flat}\cup S_{L^\flat})$. Put $L\coloneqq L^\flat+\langle x\rangle$. If $L$ is not integral, then by Remark \ref{re:analytic3}, we may write $L=L^{\flat\prime}+\langle x'\rangle$ with $L^{\flat\prime}\in\flat(\bbV)$ that is not integral; hence the lemma follows.

In what follows, we assume $L$ integral and write its fundamental invariants as $(a'_1,\dots,a'_n)$. By Remark \ref{re:analytic3}, it suffices to show that $a'_1+\cdots+a'_{n-1}\leq a_1+\cdots+a_{n-1}-2$.

Choose a normal basis $(e_1,\dots,e_{n-1})$ of $L^\flat$ (Remark \ref{re:analytic3}), and rearrange them such that for every $1\leq i\leq n-1$, exactly one of the following three happens:
\begin{enumerate}[label=(\alph*)]
  \item $(e_i,e_i)_{\bbV}=\beta_i u^{a_i-1}$ for some $\beta_i\in O_F^\times$;

  \item $(e_i,e_{i+1})_{\bbV}=u^{a_i-1}$;

  \item $(e_i,e_{i-1})_{\bbV}=-u^{a_i-1}$.
\end{enumerate}
Write $x=\lambda_1 e_1+\cdots+\lambda_{n-1}e_{n-1}+x_n$ for some $\lambda_i\in(E\setminus O_E)\cup\{0\}$ and $0\neq x_n\in V_{L^\flat}^\perp$. Let $T$ be the moment matrix with respect to the basis $\{e_1,\dots,e_{n-1},x\}$ of $L$.

If $\lambda_1=\cdots=\lambda_{n-1}=0$, then since $x\not\in S_{L^\flat}$, we have either $\langle x\rangle$ is not integral, or $\val(x)\leq a_{n-1}-2$ (only possible when $L^\flat$ is special) which implies $a'_1+\cdots+a'_{n-1}\leq a_1+\cdots+a_{n-1}-2$.

If $\lambda_i\neq 0$ for some $1\leq i\leq n-1$ such that $e_i$ is in the situation (b) or (c), then applying Lemma \ref{le:analytic7}(1) to the $(n-1)$-by-$(n-1)$ minor of $T$ deleting the $i$-th row and the $i$-th column, we obtain $a'_1+\cdots+a'_{n-1}\leq a_1+\cdots+a_{n-1}-2$.

If $\lambda_i\not\in u^{-1}O_E$ for some $1\leq i\leq n-1$ such that $e_i$ is in the situation (a), then applying Lemma \ref{le:analytic7}(1) to the $(n-1)$-by-$(n-1)$ minor of $T$ deleting the $i$-th row and the $n$-th column, we obtain $a'_1+\cdots+a'_{n-1}\leq a_1+\cdots+a_{n-1}-2$.

If $\lambda_i\neq 0$ and $\lambda_j\neq 0$ for $1\leq i<j\leq n-1$ such that both $e_i$ and $e_j$ are in the situation (a), then applying Lemma \ref{le:analytic7}(1) to the $(n-1)$-by-$(n-1)$ minor of $T$ deleting the $i$-th row and the $j$-th column, we obtain $a'_1+\cdots+a'_{n-1}\leq a_1+\cdots+a_{n-1}-2$.

The remaining case is that $\lambda_i\in u^{-1}O_E\setminus O_E$ for a unique element $1\leq i\leq n-1$ such that $e_i$ is in the situation (a). Then $L^\flat+\langle x\rangle$ is the orthogonal sum of $\langle e_1,\dots,\widehat{e_i},\dots,e_{n-1}\rangle$ and $\langle e_i,x\rangle$. In particular, if we write the fundamental invariants of $\langle e_i,x\rangle$ as $(b_1,b_2)$, then the fundamental invariants of $L^\flat+\langle x\rangle$ is the nondecreasing rearrangement of $(a_1,\dots,\widehat{a_i},\dots,a_{n-1},b_1,b_2)$. We have two cases:
\begin{itemize}
  \item If $(x,x)_{\bbV}\in u^{e_i-1}O_F$, then $(b_1,b_2)=(a_i-1,a_i-1)$. Thus, we have either $a'_1+\cdots+a'_{n-1}\leq a_1+\cdots+a_{n-1}-2$, or $i=n-1$, $a_{n-2}<a_{n-1}$, and $L^\flat+\langle x\rangle$ has fundamental invariants $(a_1,\dots,a_{n-2},a_{n-1}-1,a_{n-1}-1)$ (hence $L^\flat$ is special). The latter case is not possible as $x\not\in S_{L^\flat}$.

  \item If $(x,x)_{\bbV}\not\in u^{e_i-1}O_F$, then $b_1\leq a_i-2$. Thus we have $a'_1+\cdots+a'_{n-1}\leq a_1+\cdots+a_{n-1}-2$ or $i=n-1$.
\end{itemize}
The lemma is proved.
\end{proof}

\begin{proof}[Proof of Theorem \ref{th:kr}]
For every element $L^\flat\in\flat(\bbV)$, we define a function
\[
\Phi_{L^\flat}\coloneqq\partial\Den_{L^\flat}^\rv-\Int_{L^\flat}^\rv,
\]
which is a compactly supported locally constant function on $\bbV$ by Proposition \ref{pr:analytic} and Proposition \ref{pr:geometric}(2).
It enjoys the following properties:
\begin{enumerate}
  \item For $x\in\bbV\setminus V_{L^\flat}$, we have $\Phi_{L^\flat}(x)=\partial\Den_{L^\flat}(x)-\Int_{L^\flat}(x)$ by Proposition \ref{pr:geometric}(1).

  \item $\Phi_{L^\flat}$ is invariant under the translation by $L^\flat$, which follows from (1) and the similar properties for $\partial\Den_{L^\flat}$ and $\Int_{L^\flat}$.

  \item The support of $\widehat{\Phi_{L^\flat}}$ is contained in $\bbV^\integ$, by Proposition \ref{pr:analytic} and Proposition \ref{pr:geometric}(2).
\end{enumerate}

We prove by induction on $\val(L^\flat)$ that $\Phi_{L^\flat}\equiv0$.

The initial case is that $\val(L^\flat)=-1$, that is, $L^\flat$ is not integral. Then we have $\partial\Den_{L^\flat}=\Int_{L^\flat}=0$ hence $\Phi_{L^\flat}\equiv0$ by (1).

Now consider $L^\flat$ that is integral, and assume $\Phi_{L^{\flat\prime}}\equiv0$ for every $L^{\flat\prime}\in\flat(\bbV)$ satisfying $\val(L^{\flat\prime})<\val(L^\flat)$. For every $x\in\bbV\setminus(V_{L^\flat}\cup S_{L^\flat})$, by Lemma \ref{le:kr2}, we may write $L^\flat+\langle x\rangle = L^{\flat\prime}+\langle x'\rangle$ with some $L^{\flat\prime}\in\flat(\bbV)$ satisfying $\val(L^{\flat\prime})<\val(L^\flat)$; and we have
\begin{align*}
\Phi_{L^\flat}(x)&=\partial\Den_{L^\flat}(x)-\Int_{L^\flat}(x)  \\
&=\partial\Den(L^\flat+\langle x\rangle)-\Int(L^\flat+\langle x\rangle) \\
&=\partial\Den(L^{\flat\prime}+\langle x'\rangle)-\Int(L^{\flat\prime}+\langle x'\rangle) \\
&=\Phi_{L^{\flat\prime}}(x')=0
\end{align*}
by the induction hypothesis. Thus, the support of $\Phi_{L^\flat}$ is contained in $S_{L^\flat}$. There are two cases.

Suppose that $L^\flat$ is not special. By (2), we may write $\Phi_{L^\flat}=\CF_{L^\flat}\otimes\phi$ for a locally constant function $\phi$ on $V_{L^\flat}^\perp$ supported on $(V_{L^\flat}^\perp)^\integ$. Then $\widehat{\Phi_{L^\flat}}=C\cdot\CF_{(L^\flat)^\vee}\otimes\widehat\phi$ for some $C\in\dQ^\times$. Now since $\widehat\phi$ is invariant under the translation by $u^{-1}(V_{L^\flat}^\perp)^\integ$, we must have $\widehat\phi=0$ by (3), that is, $\Phi_{L^\flat}\equiv 0$.

Suppose that $L^\flat$ is special. We fix the orthogonal decompositions $L^\flat=L^\flat_\leftarrow\oplus L^\flat_\rightarrow$ and $L^{\flat\pm}=L^\flat_\leftarrow\oplus L^{\flat\pm}_\rightarrow$ from Lemma \ref{le:kr1}. Put $V_\leftarrow\coloneqq L^\flat_\leftarrow\otimes_{O_F}F$ and denote by $V_\rightarrow$ the orthogonal complement of $V_\leftarrow$ in $\bbV$. Then both $L^{\flat+}_\rightarrow$ and $L^{\flat-}_\rightarrow$ are integral $O_E$-lattices of $V_\rightarrow$ with fundamental invariants $(a_{n-1}-1,a_{n-1}-1)$. Moreover, we have $S_{L^\flat}=L^\flat_\leftarrow\times(L^{\flat+}_\rightarrow\cup L^{\flat-}_\rightarrow)$. Thus, by (2), we may write $\Phi_{L^\flat}=\CF_{L^\flat_\leftarrow}\otimes\phi$ for a locally constant function $\phi$ on $V_\rightarrow$ supported on $L^{\flat+}_\rightarrow\cup L^{\flat-}_\rightarrow$. Since $a_{n-1}\geq 3$ by Lemma \ref{le:kr1}, we have $L^{\flat+}_\rightarrow\cup L^{\flat-}_\rightarrow\subseteq uV_\rightarrow^\integ$, which implies that the support of $\phi$ is contained in $uV_\rightarrow^\integ$. On the other hand, by (3), the support of $\widehat\phi$ is contained in $V_\rightarrow^\integ$. Together, we must have $\phi=0$ by the Uncertainty Principle \cite{LZ}*{Proposition~8.1.6}, that is, $\Phi_{L^\flat}\equiv 0$.

By (1), we have $\partial\Den_{L^\flat}(x)=\Int_{L^\flat}(x)$ for every $x\in\bbV\setminus V_{L^\flat}$. In particular, Theorem \ref{th:kr} follows as every $O_E$-lattice $\bbL$ of $\bbV$ is of the form $L^\flat+\langle x\rangle$ for some $L^\flat\in\flat(\bbV)$.
\end{proof}

\subsection{Comparison with absolute Rapoport--Zink spaces}
\label{ss:comparison}

Let the setup be as in Subsection \ref{ss:kr}. In this subsection, we compare $\cN$ to certain (absolute) Rapoport--Zink space under the assumption that $F$ is \emph{unramified} over $\dQ_p$. Put $f\coloneqq[F:\dQ_p]$ hence $q=p^f$. This subsection is redundant if $f=1$.

To begin with, we fix a subset $\Phi$ of $\Hom(E,\dC_p)=\Hom(E,\breve{E})$ containing $\varphi_0$ and satisfying $\Hom(E,\breve{E})=\Phi\coprod\Phi^\tc$. Recall that we have regarded $E$ as a subfield of $\breve{E}$ via $\varphi_0$. We introduce more notation.
\begin{itemize}
  \item For every ring $R$, we denote by $\sfW(R)$ the $p$-typical Witt ring of $R$, with $\sfF$, $\sfV$, $[\;]$, and $\sfI(R)$ its ($p$-typical) Frobenius, the Verschiebung, the Teichm\"{u}ller lift, and the augmentation ideal, respectively. For an $\sfF^i$-linear map $\sff\colon\sfP\to\sfQ$ between $\sfW(R)$-modules with $i\geq 1$, we denote by $\sff^\natural\colon\sfW(R)\otimes_{\pres{\sfF^i}{,\sfW(R)}}\sfP\to\sfQ$ its induced $\sfW(R)$-linear map.

  \item For $i\in\dZ/f\dZ$, put $\psi_i\coloneqq\sfF^i\colon O_F\to O_F$, define $\hat\psi_i\colon O_F\to\sfW(O_F)$ to be the composition of $\psi_i$ with the Cartier homomorphism $O_F\to\sfW(O_F)$, and denote by $\varphi_i$ the unique element in $\Phi$ above $\psi_i$.

  \item For $i\in\dZ/f\dZ$, let $\epsilon_i$ be the unique unit in $\sfW(O_F)$ satisfying $\pres{\sfV}\epsilon_i=[\psi_i(u^2)]-\hat\psi_i(u^2)$, which exists by \cite{ACZ}*{Lemma~2.24}. We then fix a unit $\mu_u$ in $\sfW(O_{\breve{F}})$, where $\breve{F}$ denotes the complete maximal unramified extension of $F$ in $\breve{E}$, such that
      \begin{align}\label{eq:arz0}
      \frac{\pres{\sfF^f}\mu_u}{\mu_u}=\prod_{i=1}^{f-1}\pres{\sfF^{f-1-i}}{\epsilon_i},
      \end{align}
      which is possible since the right-hand side is a unit in $\sfW(O_F)$.

  \item For a $p$-divisible group $X$ over an object $S$ of $\Sch_{/O_{\breve{E}}}^\rv$ with an action by $O_F$, we have a decomposition
      \[
      \Lie(X)=\bigoplus_{i=0}^{f-1}\Lie_{\psi_i}(X)
      \]
      of $\sO_S$-modules according to the action of $O_F$ on $\Lie(X)$.
\end{itemize}

\begin{definition}\label{de:aexotic}
Let $S$ be an object of $\Sch_{/O_{\breve{E}}}$. We define a category $\Exo_{(n-1,1)}^\Phi(S)$ whose objects are triples $(X,\iota_X,\lambda_X)$ in which
\begin{itemize}
  \item $X$ is a $p$-divisible group over $S$ of dimension $nf$ and height $2nf$;

  \item $\iota_X\colon O_E\to\End(X)$ is an action of $O_E$ on $X$ satisfying:
    \begin{itemize}
      \item (Kottwitz condition): the characteristic polynomial of $\iota_X(u)$ on the $\sO_S$-module $\Lie_{\psi_0}(X)$ is $(T-u)^{n-1}(T+u)\in\sO_S[T]$,

      \item (Wedge condition): we have
         \begin{align*}
         \bigwedge^2\(\iota_X(u)-u\res\Lie_{\psi_0}(X)\)&=0,
         \end{align*}

      \item (Spin condition): for every geometric point $s$ of $S$, the action of $\iota_X(u)$ on $\Lie_{\psi_0}(X_s)$ is nonzero;

      \item (Banal condition): for $1\leq i\leq f-1$, $O_E$ acts on $\Lie_{\psi_i}(X)$ via $\varphi_i$;
    \end{itemize}

  \item $\lambda_X\colon X\to X^\vee$ is a $\iota_X$-compatible polarization such that $\Ker(\lambda_X)=X[\iota_X(u)]$.
\end{itemize}
A morphism (resp.\ quasi-morphism) from $(X,\iota_X,\lambda_X)$ to $(Y,\iota_Y,\lambda_Y)$ is an $O_E$-linear isomorphism (resp.\ quasi-isogeny) $\rho\colon X\to Y$ of height zero such that $\rho^*\lambda_Y=\lambda_X$.

When $S$ belongs to $\Sch_{/O_{\breve{E}}}^\rv$, we denote by $\Exo_{(n-1,1)}^{\Phi,\rb}(S)$ the subcategory of $\Exo_{(n-1,1)}^\Phi(S)$ consisting of $(X,\iota_X,\lambda_X)$ in which $X$ is supersingular.
\end{definition}

Note that both $\Exo_{(n-1,1)}^\rb$ and $\Exo_{(n-1,1)}^{\Phi,\rb}$ are prestacks (that is, presheaves valued in groupoids) on $\Sch_{/O_{\breve{E}}}^\rv$. Now we construct a morphism
\begin{align}\label{eq:arz}
-^\rel\colon\Exo_{(n-1,1)}^{\Phi,\rb}\to\Exo_{(n-1,1)}^\rb
\end{align}
of prestacks on $\Sch_{/O_{\breve{E}}}^\rv$. We will use the theory of displays \cites{Zin02,Lau08} and $O_F$-displays \cite{ACZ}.

Let $S=\Spec R$ be an affine scheme in $\Sch_{/O_{\breve{E}}}^\rv$. Take an object $(X,\iota_X,\lambda_X)$ of $\Exo_{(n-1,1)}^{\Phi,\rb}(S)$. Write $(\sfP,\sfQ,\sfF,\dot\sfF)$ for the display of $X$ (as a formal $p$-divisible group). The action of $O_F$ on $\sfP$ induces decompositions
\[
\sfP=\bigoplus_{i=0}^{f-1}\sfP_i,\quad
\sfQ=\bigoplus_{i=0}^{f-1}\sfQ_i,\quad
\sfF=\sum_{i=0}^{f-1}\sfF_i,\quad
\dot\sfF=\sum_{i=0}^{f-1}\dot\sfF_i,
\]
where $\sfP_i$ is the $\sfW(R)$-submodule on which $O_F$ acts via $\hat\psi_i$, and $\sfQ_i=\sfQ\cap\sfP_i$. It is clear that the above decomposition is $O_E$-linear, and $\sfP_i$ is a projective $O_E\otimes_{O_F,\hat\psi_i}\sfW(R)$-module of rank $n$.

\begin{lem}\label{le:arz2}
For $1\leq i\leq f-1$, we have
\begin{align*}
\sfQ_i=(u\otimes 1 - 1\otimes[\varphi_i(u)])\sfP_i+\sfI(R)\sfP_i,
\end{align*}
and that the map
\[
\sfF'_i\coloneqq\dot\sfF_i\circ (u\otimes 1 - 1\otimes[\varphi_i(u)])\cdot\colon\sfP_i\to\sfP_{i+1}
\]
is a Frobenius linear epimorphism hence isomorphism.
\end{lem}

\begin{proof}
The Banal condition in Definition \ref{de:aexotic} implies that for $1\leq i\leq f-1$,
\begin{align*}
(u\otimes 1 - 1\otimes[\varphi_i(u)])\sfP_i+\sfI(R)\sfP_i\subseteq\sfQ_i.
\end{align*}
To show the reverse inclusion, it suffices to show that the image of $(u\otimes 1 - 1\otimes[\varphi_i(u)])\sfP_i$ in $\sfP_i/\sfI(R)\sfP_i=\sfP_i\otimes_{\sfW(R)}R$ is projective of rank $n$. But the image is $(u\otimes 1 - 1\otimes\varphi_i(u))\sfP_i\otimes_{\sfW(R)}R$, which has rank $n$ since $\sfP_i$ is projective over $O_E\otimes_{O_F,\hat\psi_i}\sfW(R)$ of rank $n$.

Now we show that $(\sfF'_i)^\natural$ is surjective. It suffices to show that $\coker(\sfF'_i)^\natural\otimes_{\sfW(R)}\kappa$ vanishes for every homomorphism $\sfW(R)\to\kappa$ with $\kappa$ a perfect field of characteristic $p$. Since $\sfW(R)\to\kappa$ necessarily vanishes on $\sfI(R)$, it lifts to a homomorphism $\sfW(R)\to\sfW(\kappa)$. Thus, we may just assume that $R$ is a perfect field of characteristic $p$. Since
\[
(u\otimes 1 - 1\otimes[\varphi_i(u)])(-u\otimes 1-1\otimes[\varphi_i(u)])=[\psi_i(u^2)]-\hat\psi_i(u^2)=\pres{\sfV}\epsilon_i
\]
in which $\epsilon_i$ is a unit in $\sfW(O_F)$, the image of the map
\begin{align}\label{eq:arz2}
(u\otimes 1 - 1\otimes[\varphi_i(u)])\cdot\colon\sfP_i\to\sfP_i
\end{align}
contains $(u\otimes 1 - 1\otimes[\varphi_i(u)])\sfP_i+\sfW(R)\pres{\sfV}\epsilon_i\cdot\sfP_i$. As $R$ is a perfect field of characteristic $p$, we have $\sfW(R)\pres{\sfV}\epsilon_i=\sfI(R)$, hence \eqref{eq:arz2} is surjective. Thus, $\sfF'_i$ is a Frobenius linear epimorphism as $\sfF_i$ is.

The lemma is proved.
\end{proof}

Now we put
\[
\sfP^\rel\coloneqq\sfP_0,\quad
\sfQ^\rel\coloneqq\sfQ_0,\quad
\sfF^\rel\coloneqq\sfF'_{f-1}\circ\cdots\circ\sfF'_1\circ\sfF_0,\quad
\dot\sfF^\rel\coloneqq\sfF'_{f-1}\circ\cdots\circ\sfF'_1\circ\dot\sfF_0.
\]
Then $(\sfP^\rel,\sfQ^\rel,\sfF^\rel,\dot\sfF^\rel)$ defines an $f$(-$\dZ_p$)-display in the sense of \cite{ACZ}*{Definition~2.1} with an $O_E$-action, for which the Kottwitz condition, the Wedge condition, and the Spin condition are obviously inherited. It remains to construct the polarization $\lambda_{X^\rel}$. By Remark \ref{re:arz} below, we have the collection of perfect symmetric $\sfW(R)$-bilinear pairings $\{(\;,\;)_i\res i\in\dZ/f\dZ\}$ coming from $\lambda_X$. For $x,y\in\sfP_0$, put $x_i\coloneqq(\sfF'_{i-1}\circ\cdots\circ\sfF'_1\circ\dot\sfF_0)(x)$ and $y_i\coloneqq(\sfF'_{i-1}\circ\cdots\circ\sfF'_1\circ\dot\sfF_0)(y)$ for $1\leq i\leq f$, and we have
\begin{align*}
(\dot\sfF^\rel x,\dot\sfF^\rel y)_0
&=(\sfF'_{f-1}x_{f-1},\sfF'_{f-1}y_{f-1})_0 \\
&=(\dot\sfF_{f-1}((u\otimes 1 - 1\otimes[\varphi_{f-1}(u)])x_{f-1}),\dot\sfF_{f-1}((u\otimes 1 - 1\otimes[\varphi_{f-1}(u)])y_{f-1}))_0 \\
&=\pres{\sfV^{-1}}((u\otimes 1 - 1\otimes[\varphi_{f-1}(u)])x_{f-1},(u\otimes 1 - 1\otimes[\varphi_{f-1}(u)])y_{f-1})_{f-1} \\
&=\pres{\sfV^{-1}}{\(([\psi_{f-1}(u^2)]-\hat\psi_{f-1}(u^2))\cdot(x_{f-1},y_{f-1})_{f-1}\)} \\
&=\pres{\sfV^{-1}}{\(\pres{\sfV}\epsilon_{f-1}\cdot(x_{f-1},y_{f-1})_{f-1}\)} \\
&=\epsilon_{f-1}\cdot\pres{\sfF}(x_{f-1},y_{f-1})_{f-1} \\
&=\cdots=\(\prod_{i=1}^{f-1}\pres{\sfF^{f-1-i}}{\epsilon_i}\)\cdot\pres{\sfF^{f-1}}(x_1,y_1)_1 \\
&=\(\prod_{i=1}^{f-1}\pres{\sfF^{f-1-i}}{\epsilon_i}\)\cdot\pres{\sfF^{f-1}\sfV^{-1}}(x,y)_0.
\end{align*}
Put $(\;,\;)^\rel\coloneqq\mu_u(\;,\;)_0$, which satisfies $(\dot\sfF^\rel x,\dot\sfF^\rel y)^\rel=\pres{\sfF^{f-1}\sfV^{-1}}(x,y)^\rel$ by \eqref{eq:arz0}. Then the $f$(-$\dZ_p$)-display $(\sfP^\rel,\sfQ^\rel,\sfF^\rel,\dot\sfF^\rel)$ with $O_E$-action together with the pairing $(\;,\;)^\rel$ define an object $(X,\iota_X,\lambda_X)^\rel$ of $\Exo_{(n-1,1)}^\rb(S)$, as explained in the proof of \cite{Mih}*{Proposition~3.4} and Remark \ref{re:arz} below. It is clear that the construction is functorial in $S$.

\begin{remark}\label{re:arz}
For an object $(X,\iota_X,\lambda_X)$ of $\Exo_{(n-1,1)}^{\Phi,\rb}(S)$ with $(\sfP,\sfQ,\sfF,\dot\sfF)$ the display of $X$, we have a similar claim as in Remark \ref{re:polarization} concerning the polarization $\lambda_X$. In particular, as discussed in \cite{Mih}*{Section~11.1}, the polarization $\lambda_X$, or rather its symmetrization, is equivalent to a collection of perfect symmetric $\sfW(R)$-bilinear pairings
\[
\{(\;,\;)_i\colon\sfP_i\times\sfP_i\to\sfW(R)\res i\in\dZ/f\dZ\},
\]
satisfying $(\iota_X(\alpha)x,y)_i=(x,\iota_X(\alpha^\tc)y)_i$ for every $\alpha\in O_E$ and $(\dot\sfF_ix,\dot\sfF_iy)_{i+1}=\pres{\sfV^{-1}}{(x,y)_i}$ for $i\in\dZ/f\dZ$.

Similarly, for an object $(X',\iota_{X'},\lambda_{X'})$ of $\Exo_{(n-1,1)}^\rb(S)$ with $(\sfP',\sfQ',\sfF',\dot\sfF')$ the $f$(-$\dZ_p$)-display of $X'$, the polarization $\lambda_{X'}$ is equivalent to a perfect symmetric $\sfW(R)$-bilinear pairing
\[
(\;,\;)'\colon\sfP'\times\sfP'\to\sfW(R),
\]
satisfying $(\iota_{X'}(\alpha)x,y)'=(x,\iota_{X'}(\alpha^\tc)y)'$ for every $\alpha\in O_E$ and $(\dot\sfF'x,\dot\sfF'y)'=\pres{\sfF^{f-1}\sfV^{-1}}{(x,y)'}$.
\end{remark}

\begin{proposition}\label{pr:arz}
The morphism \eqref{eq:arz} is an isomorphism.
\end{proposition}

\begin{proof}
It suffices to show that for every affine scheme $S=\Spec R$ in $\Sch_{/O_{\breve{E}}}^\rv$, the functor $-^\rel(S)$ is fully faithful and essentially surjective.

We first show that $-^\rel(S)$ is fully faithful. Take an object $(X,\iota_X,\lambda_X)$ of $\Exo_{(n-1,1)}^{\Phi,\rb}(S)$. It suffices to show that the natural map $\Aut((X,\iota_X,\lambda_X))\to \Aut((X,\iota_X,\lambda_X)^\rel)$ is an isomorphism, which follows from a stronger statement that the natural map $\End_{O_E}(X)\to\End_{O_E}(X^\rel)$ is an isomorphism, where $X^\rel$ denotes the first entry of $(X,\iota_X,\lambda_X)^\rel$ which is an $O_F$-divisible group. For the latter, it amounts to showing that the natural map
\begin{align}\label{eq:arz1}
\End_{O_E}((\sfP,\sfQ,\sfF,\dot\sfF))\to\End_{O_E}((\sfP^\rel,\sfQ^\rel,\sfF^\rel,\dot\sfF^\rel))
\end{align}
is an isomorphism. For the injectivity, let $\sff$ be an element in the source, which decomposes as $\sff=\sum_{i=0}^{f-1}\sff_i$ for endomorphisms $\sff_i\colon\sfP_i\to\sfP_i$ which preserve $\sfQ_i$ and commute with $\sfF$ and $\dot\sfF$. Since for every $i\in\dZ/f\dZ$, $\dot\sfF_i$ is a Frobenius linear surjective map from $\sfQ_i$ to $\sfP_{i+1}$, the map $\sff$ is determined by $\sff_0$. Thus, \eqref{eq:arz1} is injective. For the surjectivity, let $\sff^\rel$ be an element in the target. Put $\sff_0\coloneqq\sff^\rel\colon\sfP_0\to\sfP_0$. By Lemma \ref{le:arz1}(2) below, there is a unique endomorphism $\sff_1$ of $\sfP_1$ rendering the following diagram
\[
\xymatrix{
\sfW(R)\otimes_{\pres{\sfF}{,\sfW(R)}}\sfQ_0 \ar[r]^-{\dot\sfF_0^\natural}\ar[d]_-{1\otimes(\sff_0\res_{\sfQ_0})} & \sfP_1 \ar[d]^-{\sff_1} \\
\sfW(R)\otimes_{\pres{\sfF}{,\sfW(R)}}\sfQ_0 \ar[r]^-{\dot\sfF_0^\natural} & \sfP_1
}
\]
commute. For $2\leq i\leq f-1$, we define $\sff_i$ to be the unique endomorphism of $\sfP_i$ satisfying that
\[
\sff_i\circ(\sfF'_{i-1}\circ\cdots\circ\sfF'_1)^\natural=(\sfF'_{i-1}\circ\cdots\circ\sfF'_1)^\natural\circ(1\otimes\sff_i).
\]
Then $\sff\coloneqq\sum_{i=0}^{f-1}\sff_i$ is an $O_E$-linear endomorphism of $\sfP$ which commutes with $\dot\sfF$ hence $\sfF$. It remains to check that $\sff(\sfQ)\subseteq\sfQ$, which follows from Lemma \ref{le:arz2}.

We then show that $-^\rel(S)$ is essentially surjective. Take an object $(X',\iota_{X'},\lambda_{X'})$ of $\Exo_{(n-1,1)}^\rb(S)$ in which $X'$ is given by an $f$(-$\dZ_p$)-display $(\sfP',\sfQ',\sfF',\dot\sfF')$. For $0\leq i\leq f-1$, put $\sfP_i\coloneqq\sfW(R)\otimes_{\pres{\sfF^i}{,\sfW(R)}}\sfP'$. Denote by $\sfu_0\colon\sfP_0\to\sfP_0$ the endomorphism given by the action of $u\in O_E$ on $\sfP'$. Put $\sfQ_0=\sfQ'$ and for $1\leq i\leq f-1$, put
\[
\sfQ_i\coloneqq((1\otimes\sfu_0)\otimes 1 - (1\otimes1)\otimes [\varphi_i(u)])\sfP_i+\sfI(R)\sfP_i.
\]
Fix a normal decomposition $\sfP'=\sfL'\oplus\sfT'$ for $\sfQ'$ and let $\ddot\sfF'\coloneqq\dot\sfF'\res_{\sfL'}+\sfF'\res_{\sfT'}\colon\sfP'\to\sfP'$ be the corresponding $\sfF^f$-linear isomorphism. For $0\leq i<f-1$, let $\ddot\sfF_i\colon\sfP_i\to\sfP_{i+1}$ be the Frobenius linear isomorphism induced by the identity map on $\sfP'$; and finally let $\ddot\sfF_{f-1}\colon\sfP_{f-1}\to\sfP_0$ be the Frobenius linear isomorphism induced by $\ddot\sfF'$. Let $\dot\sfF_0\colon\sfQ_0\to\sfP_1$ be the map defined by the formula $\dot\sfF_0(l+\pres{\sfV}w\cdot t)=\ddot\sfF_0(l)+w\ddot\sfF_0(t)$ for $l\in\sfL'$, $t\in\sfT'$, and $w\in\sfW(R)$, which is a Frobenius linear epimorphism. By Lemma \ref{le:arz1}(2) below, there is a unique endomorphism $\sfu_1$ of $\sfP_1$ rendering the following diagram
\[
\xymatrix{
\sfW(R)\otimes_{\pres{\sfF}{,\sfW(R)}}\sfQ_0 \ar[r]^-{\dot\sfF_0^\natural}\ar[d]_-{1\otimes(\sfu_0\res_{\sfQ_0})} & \sfP_1 \ar[d]^-{\sfu_1} \\
\sfW(R)\otimes_{\pres{\sfF}{,\sfW(R)}}\sfQ_0 \ar[r]^-{\dot\sfF_0^\natural} & \sfP_1
}
\]
commute.\footnote{We warn the readers that the endomorphism $\sfu_1$ might be different from $1\otimes\sfu_0$ as $u$ does not necessarily preserve the normal decomposition. However, the image of $\sfu_1-1\otimes\sfu_0$ is contained in $\sfI(R)\sfP_1$.} For $2\leq i\leq f-1$, we define $\sfu_i$ to be the unique endomorphism of $\sfP_i$ satisfying that
\[
\sfu_i\circ(\ddot\sfF_{i-1}\circ\cdots\circ\ddot\sfF_1)^\natural=(\ddot\sfF_{i-1}\circ\cdots\circ\ddot\sfF_1)^\natural\circ(1\otimes\sfu_1),
\]
and define a map $\dot\sfF_i\colon\sfQ_i\to\sfP_{i+1}$ by the following (compatible) formulae
\[
\begin{dcases}
\dot\sfF_i((\sfu_i\otimes 1-1\otimes[\varphi_i(u)])x)=\ddot\sfF_i(x), \\
\dot\sfF_i(\pres{\sfV}w\cdot x)=\frac{w}{\epsilon_i}\cdot(\sfu_{i+1}\otimes 1 + 1\otimes\pres{\sfF}[\varphi_i(u)])\ddot\sfF_i(x),
\end{dcases}
\]
for $x\in\sfP_i$ and $w\in\sfW(R)$, which is a Frobenius linear epimorphism. Put
\[
\sfP\coloneqq\bigoplus_{i=0}^{f-1}\sfP_i,\quad
\sfQ\coloneqq\bigoplus_{i=0}^{f-1}\sfQ_i,\quad
\dot\sfF\coloneqq\sum_{i=0}^{f-1}\dot\sfF_i,\quad
\sfu\coloneqq\sum_{i=0}^{f-1}\sfu_i.
\]
Then it is straightforward to check that $(\sfP,\sfQ,\sfF,\dot\sfF)$ is a display with an action by $O_E$ for which $u$ acts by $\sfu$, where $\sfF$ is determined by $\dot\sfF$ in the usual way. Now we construct a collection of perfect symmetric $\sfW(R)$-bilinear pairings $\{(\;,\;)_i\res i\in\dZ/f\dZ\}$ as in Remark \ref{re:arz}. Put $(\;,\;)_0\coloneqq\mu_u^{-1}(\;,\;)'$, where $(\;,\;)'$ is the pairing induced by $\lambda_{X'}$. Define inductively for $1\leq i\leq f-1$ the unique (perfect symmetric $\sfW(R)$-bilinear) pairing $(\;,\;)_i$ satisfying $(\dot\sfF_{i-1}x,\dot\sfF_{i-1}y)_i=\pres{\sfV^{-1}}{(x,y)_{i-1}}$. It is clear that we also have $(\dot\sfF_{f-1}x,\dot\sfF_{f-1}y)_0=\pres{\sfV^{-1}}{(x,y)_{f-1}}$. Then the display $(\sfP,\sfQ,\sfF,\dot\sfF)$ with the $O_E$-action together with the collection of pairings $\{(\;,\;)_i\res i\in\dZ/f\dZ\}$ define an object $(X,\iota_X,\lambda_X)\in\Exo_{(n-1,1)}^{\Phi,\rb}(S)$, which satisfies $(X,\iota_X,\lambda_X)^\rel\simeq(X',\iota_{X'},\lambda_{X'})$ by construction.

The proposition is proved.
\end{proof}

\begin{lem}\label{le:arz1}
Let $R$ be a ring on which $p$ is nilpotent. For a pair $(\sfP,\sfQ)$ in which $\sfP$ is a projective $\sfW(R)$-module of finite rank and $\sfQ$ is a submodule of $\sfP$ containing $\sfI(R)\sfP$ such that $\sfP/\sfQ$ is a projective $R$-module, we define $\sfQ^\star$ to be the image of $J(R)\sfP$ under the map $\sfW(R)\otimes_{\pres{\sfF}{,\sfW(R)}}\sfI(R)\sfP\to\sfW(R)\otimes_{\pres{\sfF}{,\sfW(R)}}\sfQ$ that is the base change of the inclusion map $\sfI(R)\sfP\to\sfQ$, where $J(R)$ denotes the kernel of $(\sfV^{-1})^\natural\colon\sfW(R)\otimes_{\pres{\sfF}{,\sfW(R)}}\sfI(R)\to\sfW(R)$. Then for every Frobenius linear epimorphism $\dot\sfF\colon\sfQ\to\sfP'$ with $\sfP'$ a projective $\sfW(R)$-module of the same rank as $\sfP$, we have
\begin{enumerate}
  \item the kernel of $\dot\sfF^\natural$ coincides with $\sfQ^\star$;

  \item for every endomorphism $\sff\colon\sfP\to\sfP$ that preserves $\sfQ$, there exists a unique endomorphism $\sff'\colon\sfP'\to\sfP'$ rendering the following diagram
      \[
      \xymatrix{
      \sfW(R)\otimes_{\pres{\sfF}{,\sfW(R)}}\sfQ \ar[r]^-{\dot\sfF^\natural}\ar[d]_-{1\otimes(\sff\res_\sfQ)} & \sfP' \ar[d]^-{\sff'} \\
      \sfW(R)\otimes_{\pres{\sfF}{,\sfW(R)}}\sfQ \ar[r]^-{\dot\sfF^\natural} & \sfP'
      }
      \]
      commute.
\end{enumerate}
\end{lem}

\begin{proof}
We first claim that $\sfJ(R)$ is contained in the kernel of the map
\begin{align}\label{eq:arz4}
\sfW(R)\otimes_{\pres{\sfF}{,\sfW(R)}}\sfI(R)\to\sfW(R)\otimes_{\pres{\sfF}{,\sfW(R)}}\sfW(R)=\sfW(R)
\end{align}
that is the base change of the inclusion map $\sfI(R)\to\sfW(R)$. Take an element $x=\sum a_i\otimes \pres{\sfV}b_i$ in $\sfW(R)\otimes_{\pres{\sfF}{,\sfW(R)}}\sfI(R)$. If $x\in J(R)$, then $\sum a_ib_i=0$. But the image of $x$ under \eqref{eq:arz4} is $\sum a_i\pres{\sfF\sfV}b_i$, which equals $p\sum a_ib_i$. Thus, $J(R)$ is contained in the kernel of \eqref{eq:arz4}.

For (1), choose a normal decomposition $\sfP=\sfL\oplus\sfT$ of $\sfW(R)$-modules such that $\sfQ=\sfL\oplus\sfI(R)\sfT$. By (the proof of) \cite{Lau10}*{Lemma~2.5}, there exists a Frobenius linear automorphism $\Psi$ of $P$ such that $\dot\sfF(l+at)=\Psi(l)+\pres{\sfV^{-1}}a\cdot\Psi(t)$ for $l\in\sfL$, $t\in\sfT$, and $a\in\sfI(R)$. Thus $\Ker\dot\sfF^\natural$ equals the submodule $J(R)\sfT$ of $\sfW(R)\otimes_{\pres{\sfF}{,\sfW(R)}}\sfQ$. However, by the claim above, the image of $J(R)\sfL$ under the map $\sfW(R)\otimes_{\pres{\sfF}{,\sfW(R)}}\sfI(R)\sfP\to\sfW(R)\otimes_{\pres{\sfF}{,\sfW(R)}}\sfQ$ vanishes. Thus, we have $J(R)\sfT=\sfQ^\star$.

For (2), the uniqueness follows since $\dot\sfF^\natural$ is surjective; and the existence follows since the map $1\otimes(\sff\res_\sfQ)$ preserves $\sfQ^\star$, which is a consequence of the definition of $\sfQ^\star$.
\end{proof}

To define our (absolute) Rapoport--Zink space, we fix an object $(\bbX,\iota_{\bbX},\lambda_{\bbX})\in\Exo_{(n-1,1)}^{\Phi,\rb}(\ol{k})$.

\begin{definition}\label{de:arz}
We define a functor $\cN^\Phi\coloneqq\cN^\Phi_{(\bbX,\iota_{\bbX},\lambda_{\bbX})}$ on $\Sch_{/O_{\breve{E}}}^\rv$ such that for every object $S$ of $\Sch_{/O_{\breve{E}}}^\rv$, $\cN(S)$ consists of quadruples $(X,\iota_X,\lambda_X;\rho_X)$ in which
\begin{itemize}
  \item $(X,\iota_X,\lambda_X)$ is an object of $\Exo_{(n-1,1)}^{\Phi,\rb}(S)$;

  \item $\rho_X$ is a quasi-morphism from $(X,\iota_X,\lambda_X)\times_S(S\otimes_{O_{\breve{E}}}\ol{k})$ to $(\bbX,\iota_{\bbX},\lambda_{\bbX})\otimes_{\ol{k}}(S\otimes_{O_{\breve{E}}}\ol{k})$ in the category $\Exo_{(n-1,1)}^{\Phi,\rb}(S\otimes_{O_{\breve{E}}}\ol{k})$.
\end{itemize}
\end{definition}

\begin{corollary}\label{co:arz}
The morphism
\[
\cN^\Phi=\cN^\Phi_{(\bbX,\iota_{\bbX},\lambda_{\bbX})}\to\cN\coloneqq\cN_{(\bbX,\iota_{\bbX},\lambda_{\bbX})^\rel}
\]
sending $(X,\iota_X,\lambda_X;\rho_X)$ to $((X,\iota_X,\lambda_X)^\rel;\rho_X^\rel)$ is an isomorphism.
\end{corollary}

\begin{proof}
This follows immediately from Proposition \ref{pr:arz}.
\end{proof}

Now we study special divisors on $\cN^\Phi$ and their relation with those on $\cN$. Fix a triple $(X_0,\iota_{X_0},\lambda_{X_0})$ where
\begin{itemize}
  \item $X_0$ is a supersingular $p$-divisible group over $\Spec O_{\breve{E}}$ of dimension $f$ and height $2f$;

  \item $\iota_{X_0}\colon O_E\to\End(X_0)$ is an $O_E$-action on $X_0$ such that for $0\leq i\leq f-1$, the summand $\Lie_{\psi_i}(X)$ has rank $1$ on which $O_E$ acts via $\varphi_i$;

  \item $\lambda_{X_0}\colon X_0\to X_0^\vee$ is a $\iota_{X_0}$-compatible principal polarization.
\end{itemize}
Note that $\iota_{X_0}$ induces an isomorphism $\iota_{X_0}\colon O_E\xrightarrow{\sim}\End_{O_E}(X_0)$. Put
\[
\bbV\coloneqq\Hom_{O_E}(X_0\otimes_{O_{\breve{E}}}\ol{k},\bbX)\otimes\dQ,
\]
which is a vector space over $E$ of dimension $n$, equipped with a natural hermitian form similar to \eqref{eq:rz_special}. By a construction similar to \eqref{eq:arz}, we obtain a triple $(X_0,\iota_{X_0},\lambda_{X_0})^\rel$ as in the definition of special divisors on $\cN$ (Definition \ref{de:rz_special}), and a canonical map
\[
\Hom_{O_E}(X_0\otimes_{O_{\breve{E}}}\ol{k},\bbX)\to\Hom_{O_E}(X_0^\rel\otimes_{O_{\breve{E}}}\ol{k},\bbX^\rel),
\]
which induces a map
\begin{align}\label{eq:arz3}
-^\rel\colon\bbV\to\bbV^\rel\coloneqq\Hom_{O_E}(X_0^\rel\otimes_{O_{\breve{E}}}\ol{k},\bbX^\rel)\otimes\dQ.
\end{align}
For every nonzero element $x\in\bbV$, we have similarly a closed formal subscheme $\cN^\Phi(x)$ of $\cN^\Phi$ defined similarly as in Definition \ref{de:rz_special}.

\begin{corollary}\label{co:arz1}
The map \eqref{eq:arz3} is an isomorphism of hermitian spaces. Moreover, under the isomorphism in Corollary \ref{co:arz}, we have $\cN^\Phi(x)=\cN(x^\rel)$.
\end{corollary}

\begin{proof}
By the definition of $-^\rel$, the map \eqref{eq:arz3} is clearly an isometry. Since both $\sfV$ and $\sfV^\rel$ have dimension $n$, \eqref{eq:arz3} is an isomorphism of hermitian spaces. The second assertion follows from Corollary \ref{co:arz} and construction of $-^\rel$, parallel to \cite{Mih}*{Remark~4.4}.
\end{proof}

\begin{remark}\label{re:aexotic}
Let $S$ be an object of $\Sch_{/O_{\breve{E}}}$. We have another category $\Exo_{(n,0)}^\Phi(S)$ whose objects are triples $(X,\iota_X,\lambda_X)$ in which
\begin{itemize}
  \item $X$ is a $p$-divisible group over $S$ of dimension $nf$ and height $2nf$;

  \item $\iota_X\colon O_E\to\End(X)$ is an action of $O_E$ on $X$ such that for $0\leq i\leq f-1$, $O_E$ acts on $\Lie_{\psi_i}(X)$ via $\varphi_i$;

  \item $\lambda_X\colon X\to X^\vee$ is a $\iota_X$-compatible polarization such that $\Ker(\lambda_X)=X[\iota_X(u)]$.
\end{itemize}
Morphisms are defined similarly as in Definition \ref{de:aexotic}. The category $\Exo_{(n,0)}^\Phi(S)$ is a connected groupoid. Moreover, one can show that there is a canonical isomorphism $\Exo_{(n,0)}^\Phi\to\Exo_{(n,0)}$ of prestacks after restriction to $\Sch_{/O_{\breve{E}}}^\rv$ similar to \eqref{eq:arz}.
\end{remark}

\begin{remark}
It is desirable to extend the results in this subsection to a general finite extension $F/\dQ_p$. We hope to address this problem in the future.
\end{remark}

\section{Local theta lifting at ramified places}

Throughout this section, we fix a \emph{ramified} quadratic extension $E/F$ of $p$-adic fields with $p$ odd, with $\tc\in\Gal(E/F)$ the Galois involution. We fix a uniformizer $u\in E$ satisfying $u^\tc=-u$, and denote by $q$ the cardinality of $O_E/(u)$. Let $n=2r$ be an even positive integer. We fix a nontrivial additive character $\psi_F\colon F\to\dC^\times$ of conductor $O_F$.

The goal of this section is to compute the doubling $L$-function, the doubling epsilon factor, the spherical doubling zeta integral, and the local theta lifting for a tempered admissible irreducible representation $\pi$ of $G_r(F)$ that is spherical with respect to the standard special maximal compact subgroup.

\subsection{Weil representation and spherical module}

We equip $W_r\coloneqq E^{2r}$ with the skew-hermitian form given by the matrix $\(\begin{smallmatrix}&1_r\\ -1_r &\end{smallmatrix}\)$. We denote by $\{e_1,\dots,e_{2r}\}$ the natural basis of $W_r$. Denote by $G_r$ the unitary group of $W_r$, which is a reductive group over $F$. We write elements of $W_r$ in the row form, on which $G_r$ acts from the right. Let $K_r\subseteq G_r(F)$ be the stabilizer of the lattice $O_E^{2r}\subseteq W_r$, which is a special maximal compact subgroup. We fix the Haar measure $\r{d}g$ on $G_r(F)$ that gives $K_r$ volume $1$. Let $P_r$ be the Borel subgroup of $G_r$ consisting of elements of the form
\[
\begin{pmatrix}
a & b \\
& \pres{\rt}{a}^{\tc,-1} \\
\end{pmatrix}
,
\]
in which $a$ is a lower-triangular matrix in $\Res_{E/F}\GL_r$. Let $P_r^0$ be the maximal parabolic subgroup of $G_r$ containing $P_r$ with the unipotent radical $N_r^0$, such that the standard diagonal Levi factor $M_r^0$ of $P_r^0$ is isomorphic to $\Res_{E/F}\GL_r$.

We fix a a \emph{split} hermitian space $(V,(\;,\;)_V)$ over $E$ of dimension $n=2r$, and a self-dual lattice $\Lambda_V$ of $V$, namely, $\Lambda_V=\Lambda_V^\vee\coloneqq\{x\in V\res\Tr_{E/F}(x,y)_V\in O_F\text{ for every $y\in \Lambda_V$}\}$. Put $H_V\coloneqq\rU(V)$, and let $L_V$ be the stabilizer of $\Lambda_V$ in $H_V(F)$. We fix the Haar measure $\r{d}h$ on $H_V(F)$ that gives $L_V$ volume $1$.

\begin{remark}
We have
\begin{enumerate}
  \item There exists an isomorphism $\kappa\colon W_r\to V$ of $E$-vector spaces satisfying $(\kappa(e_i),\kappa(e_j))_V=0$, $(\kappa(e_{r+i}),\kappa(e_{r+j}))_V=0$, and $(\kappa(e_i),\kappa(e_{r+j}))_V=u^{-1}\delta_{ij}$ for $1\leq i,j\leq r$, and such that $L_V$ is generated by $\{\kappa(e_i)\res 1\leq i\leq 2r\}$ as an $O_E$-submodule.

  \item The double coset $K_r\backslash G_r(F)/K_r$ has representatives
       \[
       \begin{pmatrix}
       u^{a_1} & &&&& \\
       & \ddots &&&& \\
       & & u^{a_r} &&& \\
       &&& (-u)^{-a_1} && \\
       &&&& \ddots & \\
       &&&&& (-u)^{-a_r}
       \end{pmatrix}
       \]
       where $0\leq a_1\leq \cdots\leq a_r$ are integers.
\end{enumerate}
\end{remark}

We introduce two Hecke algebras:
\[
\cH_{W_r}\coloneqq\dC[K_r\backslash G_r(F)/K_r],\qquad
\cH_V\coloneqq\dC[L_V\backslash H_V(F)/L_V].
\]
Then by the remark above, both $\cH_{W_r}$ and $\cH_V$ are commutative complex algebras, and are canonically isomorphic to $\cT_r\coloneqq\dC[T_1^{\pm 1},\dots,T_r^{\pm 1}]^{\{\pm1\}^r\rtimes\fS_r}$.

Let $(\omega_{W_r,V},\cV_{W_r,V})$ be the Weil representation of $G_r(F)\times H_V(F)$ (with respect to the additive character $\psi_F$ and the trivial splitting character). We recall the action under the Schr\"{o}dinger model $\cV_{W_r,V}\simeq C^\infty_c(V^r)$ as follows:
\begin{itemize}
  \item for $a\in\GL_r(E)$ and $\phi\in C^\infty_c(V^r)$, we have
     \[
     \omega_{W_r,V}\(\(\begin{smallmatrix}a & \\ & \pres{\rt}{a}^{\tc,-1} \end{smallmatrix}\)\)\phi(x)=|\dtm a|_E^r\cdot \phi(x a);
     \]

  \item for $b\in\Herm_r(F)$ and $\phi\in C^\infty_c(V^r)$, we have
     \[
     \omega_{W_r,V}\(\(\begin{smallmatrix} 1_r & b \\ & 1_r \end{smallmatrix}\)\)\phi(x)=\psi_F(\tr bT(x))\cdot \phi(x)
     \]
     where $T(x)\coloneqq\((x_i,x_j)_V\)_{1\leq i,j\leq r}$ is the moment matrix of $x=(x_1,\dots,x_r)$;

  \item for $\phi\in C^\infty_c(V^r)$, we have
     \[
     \omega_{W_r,V}\(\(\begin{smallmatrix} & 1_r \\ -1_r & \end{smallmatrix}\)\)\phi(x)=\widehat\phi(x);
     \]

  \item for $h\in H_V(F)$ and $\phi\in C^\infty_c(V^r)$, we have
     \[
     \omega_{W_r,V}(h)\phi(x)=\phi(h^{-1}x).
     \]
\end{itemize}
Here, we recall the Fourier transform $C^\infty_c(V^r)\to C^\infty_c(V^r)$ sending $\phi$ to $\widehat\phi$ defined by the formula
\[
\widehat\phi(x)\coloneqq\int_{V^r}\phi(y)\psi_F\(\sum_{i=1}^r\Tr_{E/F}(x_i,y_i)_V\)\rd y,
\]
where $\r{d}y$ is the self-dual Haar measure on $V^r$.

\begin{definition}
We define the \emph{spherical module} $\cS_{W_r,V}$ to be the subspace of $\cV_{W_r,V}$ consisting of elements that are fixed by $K_r\times L_V$, as a module over $\cH_{W_r}\otimes_\dC\cH_V$ via the representation $\omega_{W_r,V}$. We denote by $\Sph(V^r)$ the corresponding subspace of $C^\infty_c(V^r)$ under the Schr\"{o}dinger model.
\end{definition}

\begin{lem}\label{le:generator}
The function $\CF_{\Lambda_V^r}$ belongs to $\Sph(V^r)$.
\end{lem}

\begin{proof}
It suffices to check that
\[
\omega_{W_r,V}\(\(\begin{smallmatrix} & 1_r \\ -1_r & \end{smallmatrix}\)\)\CF_{\Lambda_V^r}=\CF_{\Lambda_V^r},
\]
which follows from the fact that $\Lambda_V^\vee=\Lambda_V$. The lemma follows.
\end{proof}

\begin{proposition}\label{pr:annihilator}
The annihilator of the $\cH_{W_r}\otimes_\dC\cH_V$-module $\cS_{W_r,V}$ is $\cI_{W_r,V}$, where $\cI_{W_r,V}$ denotes the diagonal ideal of $\cH_{W_r}\otimes_\dC\cH_V$.
\end{proposition}

\begin{proof}
The same proof of \cite{Liu20}*{Proposition~4.4} (with $\epsilon=+$ and $d=r$) works in this case as well, using Lemma \ref{le:generator}.
\end{proof}

In what follows, we review the construction of unramified principal series of $G_r(F)$ and $H_V(F)$.

We identify $M_r$, the standard diagonal Levi factor of $P_r$, with $(\Res_{E/F}\GL_1)^r$, under which we write an element of $M_r(F)$ as $a=(a_1,\dots,a_r)$ with $a_i\in E^\times$ its eigenvalue on $e_i$ for $1\leq i\leq r$. For every tuple $\sigma=(\sigma_1,\dots,\sigma_r)\in(\dC/\tfrac{2\pi i}{\log q}\dZ)^r$, we define a character $\chi^\sigma_r$ of $M_r(F)$ hence $P_r(F)$ by the formula
\[
\chi^\sigma_r(a)=\prod_{i=1}^r|a_i|_E^{\sigma_i+i-1/2}.
\]
We then have the normalized principal series
\[
\rI^\sigma_{W_r}\coloneqq\{\varphi\in C^\infty(G_r(F))\res\varphi(ag)=\chi^\sigma_r(a)\varphi(g)\text{ for $a\in P_r(F)$ and $g\in G_r(F)$}\},
\]
which is an admissible representation of $G_r(F)$ via the right translation. We denote by $\pi^\sigma_{W_r}$ the unique irreducible constituent of $\rI^\sigma_{W_r}$ that has nonzero $K_r$-invariants.

For $V$, we fix a basis $\{v_r,\dots,v_1,v_{-1},\dots,v_{-r}\}$ of the $O_E$-lattice $\Lambda_V$, satisfying $(v_i,v_j)_V=u^{-1}\delta_{i,-j}$ for every $1\leq i,j\leq r$. We have an increasing filtration
\begin{align}\label{eq:filtration}
\{0\}=Z_{r+1}\subseteq Z_r\subseteq \cdots \subseteq Z_1
\end{align}
of isotropic $E$-subspaces of $V$ where $Z_i$ be the $E$-subspaces of $V$ spanned by $\{v_r,\dots,v_i\}$. Let $Q_V$ be the (minimal) parabolic subgroup of $H_V$ that stabilizes \eqref{eq:filtration}. Let $M_V$ be the Levi factor of $Q_V$ stabilizing the lines spanned by $v_i$ for every $i$. Then we have the canonical isomorphism $M_V=(\Res_{E/F}\GL_1)^r$, under which we write an element of $M_V(F)$ as $b=(b_1,\dots,b_r)$ with $b_i\in E^\times$ its eigenvalue on $v_i$ for $1\leq i\leq r$. For every tuple $\sigma=(\sigma_1,\dots,\sigma_r)\in(\dC/\tfrac{2\pi i}{\log q}\dZ)^r$, we define a character $\chi^\sigma_V$ of $M_V(F)$ hence $Q_V(F)$ by the formula
\begin{align*}
\chi^\sigma_V(b)=\prod_{i=1}^r|b_i|_E^{\sigma_i+i-1/2}.
\end{align*}
We then have the normalized principal series
\[
\rI^\sigma_V\coloneqq\{\varphi\in C^\infty(H_V(F))\res\varphi(bh)=\chi^\sigma_V(b)\varphi(h)\text{ for $b\in Q_V(F)$ and $h\in H_V(F)$}\},
\]
which is an admissible representation of $H_V(F)$ via the right translation. We denote by $\pi^\sigma_V$ the unique irreducible constituent of $\rI^\sigma_V$ that has nonzero $L_V$-invariants.

\subsection{Doubling zeta integral and doubling L-factor}

In this section, we compute certain doubling zeta integrals and doubling $L$-factors for irreducible admissible representations $\pi$ of $G_r(F)$ satisfying $\pi^{K_r}\neq\{0\}$. We will freely use notation from \cite{Liu20}*{Section~5}.

We have the degenerate principal series $\rI^\Box_r(s)\coloneqq\Ind_{P_r^\Box}^{G_r^\Box}(|\;|_E^s\circ\Delta)$ of $G_r^\Box(F)$. Let $\ff^{(s)}_r$ be the unique section of $\rI^\Box_r(s)$ such that for every $g\in pK_r$ with $p\in P_r^\Box(F)$,
\[
\ff_r^{(s)}(g)=|\Delta(p)|_E^{s+r}.
\]
It is a holomorphic standard hence good section.

\begin{remark}\label{re:basis}
By definition, we have $\rI^\Box_r(s)\subseteq\rI^{\sigma^\Box_s}_{W_{2r}}$, where
\[
\sigma^\Box_s\coloneqq(s+r-\tfrac{1}{2},s+r-\tfrac{3}{2},\dots,s-r+\tfrac{3}{2},s-r+\tfrac{1}{2})\in(\dC/\tfrac{2\pi i}{\log q}\dZ)^{2r}.
\]
Moreover, if we denote by $\varphi^{\sigma^\Box_s}$ the unique section in $\rI^{\sigma^\Box_s}_{W_{2r}}$ that is fixed by $K_{2r}$ and such that $\varphi^{\sigma^\Box_s}(1_{4r})=1$, then $\ff_r^{(s)}=\varphi^{\sigma^\Box_s}$.
\end{remark}

Let $\pi$ be an irreducible admissible representation of $G_r(F)$. For every element $\xi\in\pi^\vee\boxtimes\pi$, we denote by $H_\xi\in C^\infty(G_r(F))$ its associated matrix coefficient. Then for every meromorphic section $f^{(s)}$ of $\rI^\Box_r(s)$, we have the (doubling) zeta integral:
\[
Z(\xi,f^{(s)})\coloneqq\int_{G_r(F)}H_\xi(g)f^{(s)}(\bw_r(g,1_{2r}))\rd g,
\]
which is absolutely convergent for $\RE s$ large enough and has a meromorphic continuation. We let $L(s,\pi)$ and $\varepsilon(s,\pi,\psi_F)$ be the doubling $L$-factor and the doubling epsilon factor of $\pi$, respectively, defined in \cite{Yam14}*{Theorem~5.2}.

Take an element $\sigma=(\sigma_1,\dots,\sigma_r)\in(\dC/\tfrac{2\pi i}{\log q}\dZ)^r$. We define an $L$-factor
\[
L^\sigma(s)\coloneqq\prod_{i=1}^r\frac{1}{(1-q^{\sigma_i-s})(1-q^{-\sigma_i-s})}.
\]
Since $\pi_{W_r}^\sigma$ is self-dual, the space $((\pi_{W_r}^\sigma)^\vee)^{K_r}\boxtimes(\pi_{W_r}^\sigma)^{K_r}$ is one dimensional. Let $\xi^\sigma$ be a generator of this one dimensional space; it satisfies $H_{\xi^\sigma}(1_{2r})\neq 0$. We normalize $\xi^\sigma$ so that $H_{\xi^\sigma}(1_{2r})=1$, which makes it unique.

\begin{proposition}\label{pr:zeta}
For $\sigma\in(\dC/\tfrac{2\pi i}{\log q}\dZ)^r$, we have
\[
Z(\xi^\sigma,\ff_r^{(s)})=\frac{L^\sigma(s+\tfrac{1}{2})}{b_{2r}(s)},
\]
where $b_{2r}(s)\coloneqq\prod_{i=1}^{r}\frac{1}{1-q^{-2s-2i}}$.
\end{proposition}

\begin{proof}
We have an isomorphism $m\colon\Res_{E/F}\GL_r\to M_r^0$ sending $a$ to $\(\begin{smallmatrix}a & \\ & \pres{\rt}{a}^{\tc,-1} \end{smallmatrix}\)$. Let $\tau$ be the unramified constituent of the normalized induction of $\boxtimes_{i=1}^r|\;|_E^{\sigma_i}$, as a representation of $\GL_r(E)$. We fix vectors $v_0\in\tau$ and $v_0^\vee\in\tau^\vee$ fixed by $M_r^0(F)\cap K_r=m(\GL_r(O_E))$ such that $\langle v_0^\vee,v_0\rangle_\tau=1$.

By a similar argument in \cite{GPSR}*{Section~6} or in the proof of \cite{Liu20}*{Proposition~5.6}, we have
\begin{align}\label{eq:zeta8}
Z(\xi^\sigma,\ff_r^{(s)})=C_{\bw'_r}(s)\int_{\GL_r(E)}\varphi^{\bw'_r\sigma^\Box_s}(\bw''_r(m(a),1_{2r}))
|\dtm a|_E^{-r/2}\langle\tau^\vee(a)v^\vee_0,v_0\rangle_\tau\rd a,
\end{align}
where
\[
C_{\bw'_r}(s)=\prod_{i=1}^r\frac{\zeta_E(2s+2i)}{\zeta_E(2s+r+i)}\prod_{i=1}^r\frac{\zeta_F(2s+2i-1)}{\zeta_F(2s+2i)}
=\prod_{i=1}^r\frac{\zeta_E(2s+2i-1)}{\zeta_E(2s+r+i)}.
\]
See the proof of \cite{Liu20}*{Proposition~5.6} for unexplained notation. By \cite{GPSR}*{Proposition~6.1}, we have
\[
\int_{\GL_r(E)}\varphi^{\bw'_r\sigma^\Box_s}(\bw''_r(m(a),1_{2r}))
|\dtm a|_E^{-r/2}\langle\tau^\vee(a)v^\vee_0,v_0\rangle_\tau\rd a=\frac{L(s+\tfrac{1}{2},\tau)L(s+\tfrac{1}{2},\tau^\vee)}{\prod_{i=1}^r\zeta_E(2s+i)}.
\]
Combining with \eqref{eq:zeta8}, we have
\begin{align*}
Z(\xi^\sigma,\ff_r^{(s)})&=\(\prod_{i=1}^r\frac{\zeta_E(2s+2i-1)}{\zeta_E(2s+r+i)}\)\cdot
\(\frac{L(s+\tfrac{1}{2},\tau)L(s+\tfrac{1}{2},\tau^\vee)}{\prod_{i=1}^r\zeta_E(2s+i)}\) \\
&=\frac{L(s+\tfrac{1}{2},\tau)L(s+\tfrac{1}{2},\tau^\vee)}{\prod_{i=1}^r\zeta_E(2s+2i)}
=\frac{L^\sigma(s+\tfrac{1}{2})}{b_{2r}(s)}.
\end{align*}
The proposition is proved.
\end{proof}

\begin{proposition}\label{pr:gcd}
For $\sigma\in(\dC/\tfrac{2\pi i}{\log q}\dZ)^r$, we have $L(s,\pi^\sigma_{W_r})=L^\sigma(s)$ and $\varepsilon(s,\pi^\sigma_{W_r},\psi_F)=1$.
\end{proposition}

\begin{proof}
It follows from the same argument for \cite{Yam14}*{Proposition~7.1}, using Proposition \ref{pr:zeta}.
\end{proof}

\begin{remark}\label{re:gcd}
It is clear that the base change $\BC(\pi^\sigma_{W_r})$ is well-defined, which is an unramified irreducible admissible representation of $\GL_n(E)$, and we have $L(s,\pi^\sigma_{W_r})=L(s,\BC(\pi^\sigma_{W_r}))$ by Proposition \ref{pr:gcd}.
\end{remark}

For an irreducible admissible representation $\pi$ of $G_r(F)$, let $\Theta(\pi,V)$ be the $\pi$-isotypic quotient of $\cV_{W_r,V}$, which is an admissible representation of $H_V(F)$, and $\theta(\pi,V)$ its maximal semisimple quotient. By \cite{Wal90}, $\theta(\pi,V)$ is either zero, or an irreducible admissible representation of $H_V(F)$, known as the \emph{theta lifting} of $\pi$ to $V$ (with respect to the additive character $\psi_F$ and the trivial splitting character).

\begin{proposition}\label{pr:theta}
For an irreducible admissible representation $\pi$ of $G_r(F)$ of the form $\pi_{W_r}^\sigma$ for an element $\sigma=(\sigma_1,\dots,\sigma_r)\in(i\dR/\tfrac{2\pi i}{\log q}\dZ)^r$, we have $\theta(\pi,V)\simeq\pi_V^\sigma$.
\end{proposition}

\begin{proof}
By the same argument in the proof of \cite{Liu20}*{Theorem~6.2}, we have $\Theta(\pi,V)^{L_V}\neq\{0\}$. By our assumption on $\sigma$, $\pi$ is tempered. By (the same argument for) \cite{GI16}*{Theorem~4.1(v)}, $\Theta(\pi,V)$ is a semisimple representation of $H_V(F)$ hence $\Theta(\pi,V)=\theta(\pi,V)$. In particular, we have $\theta(\pi,V)^{L_V}\neq\{0\}$. By Proposition \ref{pr:annihilator}, the diagonal ideal $\cI_{W_r,V}$ annihilates $(\pi_{W_r}^\sigma)^{K_r}\boxtimes\theta(\pi,V)^{L_V}$, which implies that $\theta(\pi,V)\simeq\pi_V^\sigma$.
\end{proof}

\section{Arithmetic inner product formula}

In this section, we collect all local ingredients and deduce our main theorems, following the same line as in \cite{LL}. In Subsection \ref{ss:recollection} and \ref{ss:atl}, we recall the doubling method and the arithmetic theta lifting from \cite{LL}, respectively. In Subsection \ref{ss:split}, we prove the vanishing of local indices at split places, by proving the second main ingredient of this article, namely, Theorem \ref{th:split_tempered}. In Subsection \ref{ss:inert}, we recall the formula for local indices at inert places. In Subsection \ref{ss:ramified}, we compute local indices at ramified places, based on the Kudla--Rapoport type formula Theorem \ref{th:kr}. In Subsection \ref{ss:archimedean}, we recall the formula for local indices at archimedean places. The deduction of the main results of the article is explained in Subsection \ref{ss:proof}, which is a straightforward modification of \cite{LL}*{Section~11}.

\subsection{Recollection on doubling method}
\label{ss:recollection}

For readers' convenience, we copy three groups of notation from \cite{LL}*{Section~2} to here. The only difference is item (H5), which reflects the fact that we study certain places in $\tV_F^\ram$ in the current article.

\begin{notation}\label{st:f}
Let $E/F$ be a CM extension of number fields, so that $\tc$ is a well-defined element in $\Gal(E/F)$. We continue to fix an embedding $\biota\colon E\hookrightarrow\dC$. We denote by $\bu$ the (archimedean) place of $E$ induced by $\biota$ and regard $E$ as a subfield of $\dC$ via $\biota$.
\begin{enumerate}[label=(F\arabic*)]
  \item We denote by
      \begin{itemize}
        \item $\tV_F$ and $\tV_F^\fin$ the set of all places and non-archimedean places of $F$, respectively;

        \item $\tV_F^\spl$, $\tV_F^\inert$, and $\tV_F^\ram$ the subsets of $\tV_F^\fin$ of those that are split, inert, and ramified in $E$, respectively;

        \item $\tV_F^{(\diamond)}$ the subset of $\tV_F$ of places above $\diamond$ for every place $\diamond$ of $\dQ$; and

        \item $\tV_E^?$ the places of $E$ above $\tV_F^?$.
      \end{itemize}
      Moreover,
      \begin{itemize}
        \item for every place $u\in\tV_E$ of $E$, we denote by $\ul{u}\in\tV_F$ the underlying place of $F$;

        \item for every $v\in\tV_F^\fin$, we denote by $\fp_v$ the maximal ideal of $O_{F_v}$, and put $q_v\coloneqq|O_{F_v}/\fp_v|$;

        \item for every $v\in\tV_F$, we put $E_v\coloneqq E\otimes_FF_v$ and denote by $|\;|_{E_v}\colon E_v^\times\to\dC^\times$ the normalized norm character.
      \end{itemize}

  \item Let $m\geq 0$ be an integer.
      \begin{itemize}
        \item We denote by $\Herm_m$ the subscheme of $\Res_{E/F}\Mat_{m,m}$ of $m$-by-$m$ matrices $b$ satisfying $\pres{\rt}{b}^\tc=b$. Put $\Herm_m^\circ\coloneqq\Herm_m\cap\Res_{E/F}\GL_m$.

        \item For every ordered partition $m=m_1+\cdots+m_s$ with $m_i$ a positive integer, we denote by $\partial_{m_1,\dots,m_s}\colon\Herm_m\to\Herm_{m_1}\times\cdots\times\Herm_{m_s}$ the morphism that extracts the diagonal blocks with corresponding ranks.

        \item We denote by $\Herm_m(F)^+$ (resp.\ $\Herm^\circ_m(F)^+$) the subset of $\Herm_m(F)$ of elements that are totally semi-positive definite (resp.\ totally positive definite).
      \end{itemize}

  \item For every $u\in\tV_E^{(\infty)}$, we fix an embedding $\iota_u\colon E\hookrightarrow\dC$ inducing $u$ (with $\iota_\bu=\biota$), and identify $E_u$ with $\dC$ via $\iota_u$.

  \item Let $\eta\coloneqq\eta_{E/F}\colon \dA_F^\times\to\dC^\times$ be the quadratic character associated to $E/F$. For every $v\in\tV_F$ and every positive integer $m$, put
      \[
      b_{m,v}(s)\coloneqq\prod_{i=1}^m L(2s+i,\eta_v^{m-i}).
      \]
      Put $b_m(s)\coloneqq\prod_{v\in\tV_F}b_{m,v}(s)$.

  \item For every element $T\in\Herm_m(\dA_F)$, we have the character $\psi_T\colon\Herm_m(\dA_F)\to\dC^\times$ given by the formula $\psi_T(b)\coloneqq\psi_F(\tr bT)$.

  \item Let $R$ be a commutative $F$-algebra. A (skew-)hermitian space over $R\otimes_FE$ is a free $R\otimes_FE$-module $V$ of finite rank, equipped with a (skew-)hermitian form $(\;,\;)_V$ with respect to the involution $\tc$ that is nondegenerate.
\end{enumerate}
\end{notation}

\begin{notation}\label{st:h}
We fix an even positive integer $n=2r$. Let $(V,(\;,\;)_V)$ be a hermitian space over $\dA_E$ of rank $n$ that is totally positive definite.
\begin{enumerate}[label=(H\arabic*)]
  \item For every commutative $\dA_F$-algebra $R$ and every integer $m\geq 0$, we denote by
      \[
      T(x)\coloneqq\((x_i,x_j)_V\)_{i,j}\in\Herm_m(R)
      \]
      the moment matrix of an element $x=(x_1,\dots,x_m)\in V^m\otimes_{\dA_F}R$.

  \item For every $v\in\tV_F$, we put $V_v\coloneqq V\otimes_{\dA_F}F_v$ which is a hermitian space over $E_v$, and define the local Hasse invariant of $V_v$ to be $\epsilon(V_v)\coloneqq\eta_v((-1)^r\dtm V_v)\in\{\pm 1\}$. In what follows, we will abbreviate $\epsilon(V_v)$ as $\epsilon_v$.

  \item Let $v$ be a place of $F$ and $m\geq 0$ an integer.
      \begin{itemize}
        \item For $T\in\Herm_m(F_v)$, we put $(V^m_v)_T\coloneqq\{x\in V^m_v\res T(x)=T\}$, and
            \[
            (V^m_v)_\reg\coloneqq\bigcup_{T\in\Herm_m^\circ(F_v)}(V^m_v)_T.
            \]

        \item We denote by $\sS(V_v^m)$ the space of (complex valued) Bruhat--Schwartz functions on $V_v^m$. When $v\in\tV_F^{(\infty)}$, we have the Gaussian function $\phi^0_v\in\sS(V_v^m)$ given by the formula $\phi^0_v(x)=\re^{-2\pi\tr T(x)}$.

        \item We have a Fourier transform map $\widehat{\phantom{a}}\colon \sS(V_v^m)\to \sS(V_v^m)$ sending $\phi$ to $\widehat\phi$ defined by the formula
            \[
            \widehat\phi(x)\coloneqq\int_{V_v^m}\phi(y)\psi_{E,v}\(\sum_{i=1}^m(x_i,y_i)_V\)\rd y,
            \]
            where $\r{d}y$ is the self-dual Haar measure on $V_v^m$ with respect to $\psi_{E,v}$.

        \item In what follows, we will always use this self-dual Haar measure on $V_v^m$.
      \end{itemize}

  \item Let $m\geq 0$ be an integer. For $T\in\Herm_m(F)$, we put
      \[
      \Diff(T,V)\coloneqq\{v\in\tV_F\res(V^m_v)_T=\emptyset\},
      \]
      which is a finite subset of $\tV_F\setminus\tV_F^\spl$.

  \item Take a nonempty finite subset $\tR\subseteq\tV_F^\fin$ that contains
      \[
      \{v\in\tV_F^\ram\res\text{either $\epsilon_v=-1$, or $2\mid v$, or $v$ is ramified over $\dQ$}\}.
      \]
      Let $\tS$ be the subset of $\tV_F^\fin\setminus\tR$ consisting of $v$ such that $\epsilon_v=-1$, which is contained in $\tV_F^\inert$.

  \item We fix a $\prod_{v\in\tV_F^\fin\setminus\tR}O_{E_v}$-lattice $\Lambda^\tR$ in $V\otimes_{\dA_F}\dA_F^{\infty,\tR}$ such that for every $v\in\tV_F^\fin\setminus\tR$, $\Lambda^\tR_v$ is a subgroup of $(\Lambda^\tR_v)^\vee$ of index $q_v^{1-\epsilon_v}$, where
      \[
      (\Lambda^\tR_v)^\vee\coloneqq\{x\in V_v\res\psi_{E,v}((x,y)_V)=1\text{ for every }y\in\Lambda^\tR_v\}
      \]
      is the $\psi_{E,v}$-dual lattice of $\Lambda^\tR_v$.

  \item Put $H\coloneqq\rU(V)$, which is a reductive group over $\dA_F$.

  \item Denote by $L^\tR\subseteq H(\dA_F^{\infty,\tR})$ the stabilizer of $\Lambda^\tR$, which is a special maximal subgroup. We have the (abstract) Hecke algebra away from $\tR$
      \[
      \dT^\tR\coloneqq\dZ[L^\tR\backslash H(\dA_F^{\infty,\tR})/L^\tR],
      \]
      which is a ring with the unit $\CF_{L^\tR}$, and denote by $\dS^\tR$ the subring
      \[
      \varinjlim_{\substack{\tT\subseteq\tV_F^\spl\setminus\tR\\|\tT|<\infty}}
      \dZ[(L^\tR)_\tT\backslash H(F_\tT)/(L^\tR)_\tT]\otimes\CF_{(L^\tR)^\tT}
      \]
      of $\dT^\tR$.

  \item Suppose that $V$ is \emph{incoherent}, namely, $\prod_{v\in\tV_F}\epsilon_v=-1$. For every $u\in\tV_E\setminus\tV_E^\spl$, we fix a \emph{$u$-nearby space} $\pres{u}{V}$ of $V$, which is a hermitian space over $E$, and an isomorphism $\pres{u}{V}\otimes_F\dA_F^{\ul{u}}\simeq V\otimes_{\dA_F}\dA_F^{\ul{u}}$. More precisely,
      \begin{itemize}
        \item if $u\in\tV_E^{(\infty)}$, then $\pres{u}{V}$ is the hermitian space over $E$, unique up to isomorphism, that has signature $(n-1,1)$ at $u$ and satisfies $\pres{u}{V}\otimes_F\dA_F^{\ul{u}}\simeq V\otimes_{\dA_F}\dA_F^{\ul{u}}$;

        \item if $u\in\tV_E^\fin\setminus\tV_E^\spl$, then $\pres{u}{V}$ is the hermitian space over $E$, unique up to isomorphism, that satisfies $\pres{u}{V}\otimes_F\dA_F^{\ul{u}}\simeq V\otimes_{\dA_F}\dA_F^{\ul{u}}$.
      \end{itemize}
      Put $\pres{u}{H}\coloneqq\rU(\pres{u}{V})$, which is a reductive group over $F$. Then $\pres{u}{H}(\dA_F^{\ul{u}})$ and $H(\dA_F^{\ul{u}})$ are identified.
\end{enumerate}
\end{notation}

\begin{notation}\label{st:g}
Let $m\geq 0$ be an integer. We equip $W_m=E^{2m}$ and $\bar{W}_m=E^{2m}$ the skew-hermitian forms given by the matrices $\tw_m$ and $-\tw_m$, respectively.
\begin{enumerate}[label=(G\arabic*)]
  \item Let $G_m$ be the unitary group of both $W_m$ and $\bar{W}_m$. We write elements of $W_m$ and $\bar{W}_m$ in the row form, on which $G_m$ acts from the right.

  \item We denote by $\{e_1,\dots,e_{2m}\}$ and $\{\bar{e}_1,\dots,\bar{e}_{2m}\}$ the natural bases of $W_m$ and $\bar{W}_m$, respectively.

  \item Let $P_m\subseteq G_m$ be the parabolic subgroup stabilizing the subspace generated by $\{e_{r+1},\dots,e_{2m}\}$, and $N_m\subseteq P_m$ its unipotent radical.

  \item We have
     \begin{itemize}
       \item a homomorphism $m\colon\Res_{E/F}\GL_m\to P_m$ sending $a$ to
          \[
          m(a)\coloneqq
          \begin{pmatrix}
              a &  \\
               & \pres{\rt}{a}^{\tc,-1} \\
          \end{pmatrix}
          ,
          \]
          which identifies $\Res_{E/F}\GL_m$ as a Levi factor of $P_m$.

       \item a homomorphism $n\colon\Herm_m\to N_m$ sending $b$ to
          \[
          n(b)\coloneqq
          \begin{pmatrix}
              1_m & b \\
               & 1_m \\
          \end{pmatrix}
          ,
          \]
          which is an isomorphism.
     \end{itemize}

  \item We define a maximal compact subgroup $K_m=\prod_{v\in\tV_F}K_{m,v}$ of $G_m(\dA_F)$ in the following way:
    \begin{itemize}
      \item for $v\in\tV_F^\fin$, $K_{m,v}$ is the stabilizer of the lattice $O_{E_v}^{2m}$;

      \item for $v\in\tV_F^{(\infty)}$, $K_{m,v}$ is the subgroup of the form
         \[
         [k_1,k_2]\coloneqq\frac{1}{2}
         \begin{pmatrix}
           k_1+k_2   & -ik_1+ik_2 \\
           ik_1-ik_2   & k_1+k_2 \\
         \end{pmatrix}
         ,
         \]
         in which $k_i\in\GL_m(\dC)$ satisfying $k_i\pres{\rt}{k}_i^\tc=1_m$ for $i=1,2$. Here, we have identified $G_m(F_v)$ as a subgroup of $\GL_{2m}(\dC)$ via the embedding $\iota_v$ in Notation \ref{st:f}(F3).
    \end{itemize}

  \item For every $v\in\tV_F^{(\infty)}$, we have a character $\kappa_{m,v}\colon K_{m,v}\to\dC^\times$ that sends $[k_1,k_2]$ to $\dtm k_1/\dtm k_2$.\footnote{In fact, both $K_{m,v}$ and $\kappa_{m,v}$ do not depend on the choice of the embedding $\iota_v$ for $v\in\tV_F^{(\infty)}$.}

  \item For every $v\in\tV_F$, we define a Haar measure $\r{d}g_v$ on $G_m(F_v)$ as follows:
    \begin{itemize}
      \item for $v\in\tV_F^\fin$, $\r{d}g_v$ is the Haar measure under which $K_{m,v}$ has volume $1$;

      \item for $v\in\tV_F^{(\infty)}$, $\r{d}g_v$ is the product of the measure on $K_{m,v}$ of total volume $1$ and the standard hyperbolic measure on $G_m(F_v)/K_{m,v}$.
    \end{itemize}
    Put $\r{d}g=\prod_{v}\r{d}g_v$, which is a Haar measure on $G_m(\dA_F)$.

  \item We denote by $\cA(G_m(F)\backslash G_m(\dA_F))$ the space of both $\cZ(\fg_{m,\infty})$-finite and $K_{m,\infty}$-finite automorphic forms on $G_m(\dA_F)$, where $\cZ(\fg_{m,\infty})$ denotes the center of the complexified universal enveloping algebra of the Lie algebra $\fg_{m,\infty}$ of $G_m\otimes_FF_\infty$. We denote by
      \begin{itemize}
        \item $\cA^{[r]}(G_m(F)\backslash G_m(\dA_F))$ the maximal subspace of $\cA(G_m(F)\backslash G_m(\dA_F))$ on which for every $v\in\tV_F^{(\infty)}$, $K_{m,v}$ acts by the character $\kappa_{m,v}^r$,

        \item $\cA^{[r]\tR}(G_m(F)\backslash G_m(\dA_F))$ the maximal subspace of $\cA^{[r]}(G_m(F)\backslash G_m(\dA_F))$ on which
             \begin{itemize}
               \item for every $v\in\tV_F^\fin\setminus(\tR\cup\tS)$, $K_{m,v}$ acts trivially; and

               \item for every $v\in\tS$, the standard Iwahori subgroup $I_{m,v}$ acts trivially and $\dC[I_{m,v}\backslash K_{m,v}/I_{m,v}]$ acts by the character $\kappa_{m,v}^-$ (\cite{Liu20}*{Definition~2.1}),
             \end{itemize}

        \item $\cA_\cusp(G_m(F)\backslash G_m(\dA_F))$ the subspace of $\cA(G_m(F)\backslash G_m(\dA_F))$ of cusp forms, and by $\langle\;,\;\rangle_{G_m}$ the hermitian form on $\cA_\cusp(G_m(F)\backslash G_m(\dA_F))$ given by the Petersson inner product with respect to the Haar measure $\r{d}g$.
      \end{itemize}
      For a subspace $\cV$ of $\cA(G_m(F)\backslash G_m(\dA_F))$, we denote by
      \begin{itemize}
        \item $\cV^{[r]}$ the intersection of $\cV$ and $\cA^{[r]}(G_m(F)\backslash G_m(\dA_F))$,

        \item $\cV^{[r]\tR}$ the intersection of $\cV$ and $\cA^{[r]\tR}(G_m(F)\backslash G_m(\dA_F))$,

        \item $\cV^\tc$ the subspace $\{\varphi^\tc\res\varphi\in\cV\}$.
      \end{itemize}
\end{enumerate}
\end{notation}

\begin{assumption}\label{st:representation}
In what follows, we will consider an irreducible automorphic subrepresentation $(\pi,\cV_\pi)$ of $\cA_\cusp(G_r(F)\backslash G_r(\dA_F))$ satisfying that
\begin{enumerate}
  \item for every $v\in\tV_F^{(\infty)}$, $\pi_v$ is the (unique up to isomorphism) discrete series representation whose restriction to $K_{r,v}$ contains the character $\kappa_{r,v}^r$;

  \item for every $v\in\tV_F^\fin\setminus\tR$, $\pi_v$ is unramified (resp.\ almost unramified) with respect to $K_{r,v}$ if $\epsilon_v=1$ (resp.\ $\epsilon_v=-1$);

  \item for every $v\in\tV_F^\fin$, $\pi_v$ is tempered.
\end{enumerate}
We realize the contragredient representation $\pi^\vee$ on $\cV_\pi^\tc$ via the Petersson inner product $\langle\;,\;\rangle_{G_r}$ (Notation \ref{st:g}(G8)). By (1) and (2), we have $\cV_\pi^{[r]\tR}\neq\{0\}$, where $\cV_\pi^{[r]\tR}$ is defined in Notation \ref{st:g}(G8).
\end{assumption}

\begin{remark}\label{re:dichotomy}
By Proposition \ref{pr:uniqueness}(2) below, we know that when $\tR\subseteq\tV_F^\spl$, $V$ coincides with the hermitian space over $\dA_E$ of rank $n$ determined by $\pi$ via local theta dichotomy.
\end{remark}

\begin{definition}\label{de:lfunction}
We define the $L$-function for $\pi$ as the Euler product $L(s,\pi)\coloneqq\prod_{v}L(s,\pi_v)$ over all places of $F$, in which
\begin{enumerate}
  \item for $v\in\tV_F^\fin$, $L(s,\pi_v)$ is the doubling $L$-function defined in \cite{Yam14}*{Theorem~5.2};

  \item for $v\in\tV_F^{(\infty)}$, $L(s,\pi_v)$ is the $L$-function of the standard base change $\BC(\pi_v)$ of $\pi_v$. By Assumption \ref{st:representation}(1), $\BC(\pi_v)$ is the principal series representation of $\GL_n(\dC)$ that is the normalized induction of $\arg^{n-1}\boxtimes\arg^{n-3}\boxtimes\cdots\boxtimes\arg^{3-n}\boxtimes\arg^{1-n}$ where $\arg\colon\dC^\times\to\dC^\times$ is the argument character.
\end{enumerate}
\end{definition}

\begin{remark}\label{re:doubling}
Let $v$ be a place of $F$.
\begin{enumerate}
  \item For $v\in\tV_F^{(\infty)}$, doubling $L$-function is only well-defined up to an entire function without zeros. However, one can show that $L(s,\pi_v)$ satisfies the requirement for the doubling $L$-function in \cite{Yam14}*{Theorem~5.2}.

  \item For $v\in\tV_F^\spl$, the standard base change $\BC(\pi_v)$ is well-defined and we have $L(s,\pi_v)=L(s,\BC(\pi_v))$ by \cite{Yam14}*{Theorem~7.2}.

  \item For $v\in\tV_F^\inert\setminus\tR$, the standard base change $\BC(\pi_v)$ is well-defined and we have $L(s,\pi_v)=L(s,\BC(\pi_v))$ by \cite{Liu20}*{Remark~1.4}.

  \item For $v\in\tV_F^\ram\setminus\tR$, the standard base change $\BC(\pi_v)$ is well-defined and we have $L(s,\pi_v)=L(s,\BC(\pi_v))$ by Remark \ref{re:gcd}.
\end{enumerate}
In particular, when $\tR\subseteq\tV_F^\spl$, we have $L(s,\pi)=\prod_{v}L(s,\BC(\pi_v))$.
\end{remark}

Recall that we have the normalized doubling integral
\[
\fZ^\natural_{\pi_v,V_v}\colon\pi_v^\vee\otimes\pi_v\otimes\sS(V_v^{2r})\to\dC
\]
from \cite{LL}*{Section~3}.

\begin{proposition}\label{pr:uniqueness}
Let $(\pi,\cV_\pi)$ be as in Assumption \ref{st:representation}.
\begin{enumerate}
  \item For every $v\in\tV_F^\fin$, we have
     \begin{align*}
     \dim_\dC\Hom_{G_r(F_v)\times G_r(F_v)}(\rI^\Box_{r,v}(0),\pi_v\boxtimes\pi_v^\vee)=1.
     \end{align*}

  \item For every $v\in(\tV_F^\fin\setminus\tR)\cup\tV_F^\spl$, $V_v$ is the unique hermitian space over $E_v$ of rank $2r$, up to isomorphism, such that $\fZ^\natural_{\pi_v,V_v}\neq 0$.

  \item For every $v\in\tV_F^\fin$, $\Hom_{G_r(F_v)}(\sS(V_v^r),\pi_v)$ is irreducible as a representation of $H(F_v)$, and is nonzero if $v\in(\tV_F^\fin\setminus\tR)\cup\tV_F^\spl$.
\end{enumerate}
\end{proposition}

\begin{proof}
This is same as \cite{LL}*{Proposition~3.6} except that in (2) we have to take care of the case where $v\in\tV_F^\ram$, which is a consequence of Proposition \ref{pr:theta}.
\end{proof}

\begin{proposition}\label{pr:eisenstein}
Let $(\pi,\cV_\pi)$ be as in Assumption \ref{st:representation} such that $L(\tfrac{1}{2},\pi)=0$. Take
\begin{itemize}
  \item $\varphi_1=\otimes_v\varphi_{1v}\in\cV_\pi^{[r]\tR}$ and $\varphi_2=\otimes_v\varphi_{2v}\in\cV_\pi^{[r]\tR}$ such that $\langle\varphi_{1v}^\tc,\varphi_{2v}\rangle_{\pi_v}=1$ for $v\in\tV_F\setminus\tR$,\footnote{Strictly speaking, what we fixed is a decomposition $\varphi_1^\tc=\otimes_v(\varphi_1^\tc)_v$ and have abused notation by writing $\varphi^\tc_{1v}$ instead of $(\varphi_1^\tc)_v$.} and

  \item $\Phi=\otimes_v\Phi_v\in\sS(V^{2r})$ such that $\Phi_v$ is the Gaussian function (Notation \ref{st:h}(H3)) for $v\in\tV_F^{(\infty)}$, and $\Phi_v=\CF_{(\Lambda^\tR_v)^{2r}}$ for $v\in\tV_F^\fin\setminus\tR$.
\end{itemize}
Then we have
\begin{align*}
&\quad\int_{G_r(F)\backslash G_r(\dA_F)}\int_{G_r(F)\backslash G_r(\dA_F)}\varphi_2(g_2)\varphi_1^\tc(g_1)E'(0,(g_1,g_2),\Phi)\rd g_1\rd g_2 \\
&=\frac{L'(\tfrac{1}{2},\pi)}{b_{2r}(0)}\cdot C_r^{[F:\dQ]}
\cdot\prod_{v\in\tV_F^\fin}\fZ^\natural_{\pi_v,V_v}(\varphi^\tc_{1v},\varphi_{2v},\Phi_v) \\
&=\frac{L'(\tfrac{1}{2},\pi)}{b_{2r}(0)}\cdot C_r^{[F:\dQ]}
\cdot\prod_{v\in\tS}\frac{(-1)^rq_v^{r-1}(q_v+1)}{(q_v^{2r-1}+1)(q_v^{2r}-1)}
\cdot\prod_{v\in\tR}\fZ^\natural_{\pi_v,V_v}(\varphi^\tc_{1v},\varphi_{2v},\Phi_v),
\end{align*}
where
\[
C_r\coloneqq(-1)^r2^{r(r-1)}\pi^{r^2}\frac{\Gamma(1)\cdots\Gamma(r)}{\Gamma(r+1)\cdots\Gamma(2r)},
\]
and the measure on $G_r(\dA_F)$ is the one defined in Notation \ref{st:g}(G7).
\end{proposition}

\begin{proof}
The proof is same as \cite{LL}*{Proposition~3.7}, with the additional input
\[
\fZ^\natural_{\pi_v,V_v}(\varphi^\tc_{1v},\varphi_{2v},\Phi_v)=1
\]
for $v\in\tV_F^\ram\setminus\tR$ by Proposition \ref{pr:zeta}.
\end{proof}

Suppose that $V$ is incoherent. By \cite{Liu12}*{Section~2B}, we have
\begin{enumerate}
  \item Take an element $u\in\tV_E\setminus\tV_E^\spl$, and $\pres{u}{\Phi}=\otimes_v\pres{u}{\Phi}_v\in\sS(\pres{u}{V}^{2r}\otimes_F\dA_F)$, where we recall from Notation \ref{st:h}(H9) that $\pres{u}{V}$ is the $u$-nearby hermitian space, such that $\supp(\pres{u}{\Phi}_v)\subseteq(\pres{u}{V}^{2r}_v)_\reg$ (Notation \ref{st:h}(H3)) for $v$ in a nonempty subset $\tR'\subseteq\tR$. Then for every $g\in P_r^\Box(F_{\tR'})G_r^\Box(\dA_F^{\tR'})$, we have
      \[
      E(0,g,\pres{u}{\Phi})=\sum_{T^\Box\in\Herm_{2r}^\circ(F)}\prod_{v\in\tV_F} W_{T^\Box}(0,g_v,\pres{u}{\Phi}_v).
      \]

  \item Take $\Phi=\otimes_v\Phi_v\in\sS(V^{2r})$ such that $\supp(\Phi_v)\subseteq(V^{2r}_v)_\reg$ for $v$ in a subset $\tR'\subseteq\tR$ of cardinality at least $2$. Then for every $g\in P_r^\Box(F_{\tR'})G_r^\Box(\dA_F^{\tR'})$, we have
      \[
      E'(0,g,\Phi)=\sum_{w\in\tV_F\setminus\tV_F^\spl}\fE(g,\Phi)_w,
      \]
      where
      \[
      \fE(g,\Phi)_w\coloneqq\sum_{\substack{T^\Box\in\Herm_{2r}^\circ(F)\\\Diff(T^\Box,V)=\{w\}}}W'_{T^\Box}(0,g_w,\Phi_w)
      \prod_{v\in\tV_F\setminus\{w\}} W_{T^\Box}(0,g_v,\Phi_v).
      \]
      Here, $\Diff(T^\Box,V)$ is defined in Notation \ref{st:h}(H4).
\end{enumerate}

\begin{definition}\label{de:analytic_kernel}
Suppose that $V$ is incoherent. Take an element $u\in\tV_E\setminus\tV_E^\spl$, and a pair $(T_1,T_2)$ of elements in $\Herm_r(F)$.
\begin{enumerate}
  \item For $\pres{u}{\Phi}=\otimes_v\pres{u}{\Phi}_v\in\sS(\pres{u}{V}^{2r}\otimes_F\dA_F)$, we put
    \[
    E_{T_1,T_2}(g,\pres{u}{\Phi})\coloneqq\sum_{\substack{T^\Box\in\Herm_{2r}^\circ(F)\\\partial_{r,r}T^\Box=(T_1,T_2)}}
    \prod_{v\in\tV_F} W_{T^\Box}(0,g_v,\pres{u}{\Phi}_v).
    \]

  \item For $\Phi=\otimes_v\Phi_v\in\sS(V^{2r})$, we put
    \[
    \fE_{T_1,T_2}(g,\Phi)_u\coloneqq\sum_{\substack{T^\Box\in\Herm_{2r}^\circ(F)\\\Diff(T^\Box,V)=\{\ul{u}\}\\ \partial_{r,r}T^\Box=(T_1,T_2)}}
    W'_{T^\Box}(0,g_{\ul{u}},\Phi_{\ul{u}})\prod_{v\in\tV_F\setminus\{\ul{u}\}}W_{T^\Box}(0,g_v,\Phi_v).
    \]
\end{enumerate}
Here, $\partial_{r,r}\colon\Herm_{2r}\to\Herm_r\times\Herm_r$ is defined in Notation \ref{st:f}(F2).
\end{definition}

\subsection{Recollection on arithmetic theta lifting}
\label{ss:atl}

From this moment, we will assume $F\neq\dQ$.

Recall that we have fixed a $\bu$-nearby space $\pres{\bu}V$ and an isomorphism $\pres{\bu}{V}\otimes_F\dA_F^{\ul{\bu}}\simeq V\otimes_{\dA_F}\dA_F^{\ul{\bu}}$ from Notation \ref{st:h}(H9). For every open compact subgroup $L\subseteq H(\dA_F^\infty)$, we have the Shimura variety $X_L$ associated to $\Res_{F/\dQ}\pres{\bu}{H}$ of the level $L$, which is a smooth quasi-projective scheme over $E$ (which is regarded as a subfield of $\dC$ via $\biota$) of dimension $n-1$. We remind the readers its complex uniformization
\begin{align}\label{eq:complex_uniformization}
(X_L\otimes_E\dC)^{\r{an}}\simeq\pres{\bu}{H}(F)\backslash\fD\times H(\dA_F)/L,
\end{align}
where $\fD$ denotes the complex manifold of negative lines in $\pres{\bu}{V}\otimes_E\dC$ and the Deligne homomorphism is the one adopted in \cite{LTXZZ}*{Section~3.2}. In what follows, for a place $u\in\tV_E$, we put $X_{L,u}\coloneqq X_L\otimes_EE_u$ as a scheme over $E_u$.

For every $\phi^\infty\in\sS(V^m\otimes_{\dA_F}\dA_F^\infty)^L$ and $T\in\Herm_m(F)$, we put
\[
Z_T(\phi^\infty)_L\coloneqq\sum_{\substack{x\in L\backslash V^m\otimes_{\dA_F}\dA_F^\infty\\ T(x)=T}}\phi^\infty(x)Z(x)_L,
\]
where $Z(x)_L$ is Kudla's special cycle recalled in \cite{LL}*{Definition~4.1}. As the above summation is finite, $Z_T(\phi^\infty)_L$ is a well-defined element in $\CH^m(X_L)_\dC$. For every $g\in G_m(\dA_F)$, Kudla's \emph{generating function} is defined to be
\[
Z_{\phi^\infty}(g)_L\coloneqq\sum_{T\in\Herm_m(F)^+}
\omega_{m,\infty}(g_\infty)\phi^0_\infty(T)\cdot
Z_T(\omega_m^\infty(g^\infty)\phi^\infty)_L
\]
as a formal sum valued in $\CH^m(X_L)_\dC$, where
\[
\omega_{m,\infty}(g_\infty)\phi^0_\infty(T)\coloneqq\prod_{v\in\tV_F^{(\infty)}}\omega_{m,v}(g_v)\phi^0_v(T).
\]
Here, we note that for $v\in\tV_F^{(\infty)}$, the function $\omega_{m,v}(g_v)\phi^0_v$ factors through the moment map $V_v^m\to\Herm_m(F_v)$ (see Notation \ref{st:h}(H1)).

\begin{hypothesis}[Modularity of generating functions of codimension $m$, \cite{LL}*{Hypothesis~4.5}]\label{hy:modularity}
For every open compact subgroup $L\subseteq H(\dA_F^\infty)$, every $\phi^\infty\in\sS(V^m\otimes_{\dA_F}\dA_F^\infty)^L$, and every complex linear map $l\colon\CH^m(X_L)_\dC\to\dC$, the assignment
\[
g\mapsto l(Z_{\phi^\infty}(g)_L)
\]
is absolutely convergent, and gives an element in $\cA^{[r]}(G_m(F)\backslash G_m(\dA_F))$. In other words, the function $Z_{\phi^\infty}(-)_L$ defines an element in $\Hom_\dC(\CH^m(X_L)_\dC^\vee,\cA^{[r]}(G_m(F)\backslash G_m(\dA_F)))$.
\end{hypothesis}

\begin{definition}\label{de:arithmetic_theta}
Let $(\pi,\cV_\pi)$ be as in Assumption \ref{st:representation}. Assume Hypothesis \ref{hy:modularity} on the modularity of generating functions of codimension $r$. For every $\varphi\in\cV_\pi^{[r]}$, every open compact subgroup $L\subseteq H(\dA_F^\infty)$, and every $\phi^\infty\in\sS(V^r\otimes_{\dA_F}\dA_F^\infty)^L$, we put
\[
\Theta_{\phi^\infty}(\varphi)_L\coloneqq\int_{G_r(F)\backslash G_r(\dA_F)}
\varphi^\tc(g)Z_{\phi^\infty}(g)_L\rd g,
\]
which is an element in $\CH^r(X_L)_\dC$ by \cite{LL}*{Proposition~4.7}. It is clear that the image of $\Theta_{\phi^\infty}(\varphi)_L$ in
\[
\CH^r(X)_\dC\coloneqq\varinjlim_L\CH^r(X_L)_\dC
\]
depends only on $\varphi$ and $\phi^\infty$, which we denote by $\Theta_{\phi^\infty}(\varphi)$. Finally, we define the \emph{arithmetic theta lifting} of $(\pi,\cV_\pi)$ to $V$ (with respect to $\biota$) to be the complex subspace $\Theta(\pi,V)$ of $\CH^r(X)_\dC$ spanned by $\Theta_{\phi^\infty}(\varphi)$ for all $\varphi\in\cV_\pi^{[r]}$ and $\phi^\infty\in\sS(V^r\otimes_{\dA_F}\dA_F^\infty)$.
\end{definition}

We recall Beilinson's height pairing for our particular use from \cite{LL}*{Section~6}. We have a map
\begin{align*}
\langle\;,\;\rangle_{X_L,E}^\ell\colon
\CH^r(X_L)^{\langle\ell\rangle}_\dC\times\CH^r(X_L)^{\langle\ell\rangle}_\dC\to\dC\otimes_\dQ\dQ_\ell
\end{align*}
that is complex linear in the first variable, and conjugate symmetric. Here, $\ell$ is a rational prime so that $X_{L,u}$ has smooth projective reduction for every $u\in\tV_E^{(\ell)}$. For a pair $(c_1,c_2)$ of elements in $\rZ^r(X_L)^{\langle\ell\rangle}_\dC\times\rZ^r(X_L)^{\langle\ell\rangle}_\dC$ with disjoint supports, we have
\begin{align*}
\langle c_1,c_2\rangle_{X_L,E}^\ell=\sum_{u\in\tV_E^{(\infty)}}2\langle c_1,c_2\rangle_{X_{L,u},E_u}+
\sum_{u\in\tV_E^\fin}\log q_u\cdot \langle c_1,c_2\rangle_{X_{L,u},E_u}^\ell,
\end{align*}
in which
\begin{itemize}
  \item $q_u$ is the residue cardinality of $E_u$ for $u\in\tV_E^\fin$;

  \item $\langle c_1,c_2\rangle_{X_{L,u},E_u}^\ell\in\dC\otimes_\dQ\dQ_\ell$ is the non-archimedean local index recalled in \cite{LL}*{Section~B} for $u\in\tV_E^\fin$ (see \cite{LL}*{Remark~B.11} when $u$ is above $\ell$), which equals zero for all but finitely many $u$;

  \item $\langle c_1,c_2\rangle_{X_{L,u},E_u}\in\dC$ is the archimedean local index for $u\in\tV_E^{(\infty)}$, recalled in \cite{LL}*{Section~10}.
\end{itemize}

\begin{definition}\label{de:good}
We say that a rational prime $\ell$ is \emph{$\tR$-good} if $\ell$ is unramified in $E$ and satisfies $\tV_F^{(\ell)}\subseteq\tV_F^\fin\setminus(\tR\cup\tS)$.
\end{definition}

\begin{definition}\label{de:tempered_generic}
For every open compact subgroup $L_\tR$ of $H(F_\tR)$ and every subfield $\dL$ of $\dC$, we define
\begin{enumerate}
  \item $(\dS^\tR_\dL)^0_{L_\tR}$ to be the subalgebra of $\dS^\tR_\dL$ (Notation \ref{st:h}(H8)) of elements that annihilate
      \[
      \bigoplus_{i\neq 2r-1}\rH^i_{\dr}(X_{L_\tR L^\tR}/E)\otimes_\dQ\dL,
      \]

  \item for every rational prime $\ell$, $(\dS^\tR_\dL)^{\langle\ell\rangle}_{L_\tR}$ to be the subalgebra of $\dS^\tR_\dL$ of elements that annihilate
      \[
      \bigoplus_{u\in\tV_E^\fin\setminus\tV_E^{(\ell)}}\rH^{2r}(X_{L_\tR L^\tR,u},\dQ_\ell(r))\otimes_\dQ\dL.
      \]
\end{enumerate}
Here, $L^\tR$ is defined in Notation \ref{st:h}(H8).
\end{definition}

\begin{definition}\label{de:kernel_geometric}
Consider a nonempty subset $\tR'\subseteq\tR$, an $\tR$-good rational prime $\ell$, and an open compact subgroup $L$ of $H(\dA_F^\infty)$ of the form $L_\tR L^\tR$ where $L^\tR$ is defined in Notation \ref{st:h}(H8). An \emph{$(\tR,\tR',\ell,L)$-admissible sextuple} is a sextuple $(\phi^\infty_1,\phi^\infty_2,\rs_1,\rs_2,g_1,g_2)$ in which
\begin{itemize}
  \item for $i=1,2$, $\phi^\infty_i=\otimes_v\phi^\infty_{iv}\in\sS(V^r\otimes_{\dA_F}\dA_F^\infty)^L$ in which $\phi^\infty_{iv}=\CF_{(\Lambda^\tR_v)^r}$ for $v\in\tV_F^\fin\setminus\tR$, satisfying that $\supp(\phi^\infty_{1v}\otimes(\phi^\infty_{2v})^\tc)\subseteq(V^{2r}_v)_\reg$ for $v\in\tR'$;

  \item for $i=1,2$, $\rs_i$ is a product of two elements in $(\dS^\tR_{\dQ^\ac})^{\langle\ell\rangle}_{L_\tR}$;

  \item for $i=1,2$, $g_i$ is an element in $G_r(\dA_F^{\tR'})$.
\end{itemize}

For an $(\tR,\tR',\ell,L)$-admissible sextuple $(\phi^\infty_1,\phi^\infty_2,\rs_1,\rs_2,g_1,g_2)$ and every pair $(T_1,T_2)$ of elements in $\Herm_r^\circ(F)^+$, we define
\begin{enumerate}
  \item the global index $I_{T_1,T_2}(\phi^\infty_1,\phi^\infty_2,\rs_1,\rs_2,g_1,g_2)^\ell_L$ to be
     \begin{align*}
     \langle \omega_{r,\infty}(g_{1\infty})\phi^0_\infty(T_1)\cdot\rs_1^*Z_{T_1}(\omega_r^\infty(g_1^\infty)\phi^\infty_1)_L,
     \omega_{r,\infty}(g_{2\infty})\phi^0_\infty(T_2)\cdot\rs_2^*Z_{T_2}(\omega_r^\infty(g_2^\infty)\phi^\infty_2)_L
     \rangle_{X_L,E}^\ell
     \end{align*}
     as an element in $\dC\otimes_\dQ\dQ_\ell$, where we note that for $i=1,2$, $\rs_i^*Z_{T_i}(\omega_r^\infty(g_i^\infty)\phi^\infty_i)_L$ belongs to $\CH^r(X_L)^{\langle\ell\rangle}_\dC$ by Definition \ref{de:tempered_generic}(2);

  \item for every $u\in\tV_E^\fin$, the local index $I_{T_1,T_2}(\phi^\infty_1,\phi^\infty_2,\rs_1,\rs_2,g_1,g_2)^\ell_{L,u}$ to be
     \begin{align*}
     \langle \omega_{r,\infty}(g_{1\infty})\phi^0_\infty(T_1)\cdot\rs_1^*Z_{T_1}(\omega_r^\infty(g_1^\infty)\phi^\infty_1)_L,
     \omega_{r,\infty}(g_{2\infty})\phi^0_\infty(T_2)\cdot\rs_2^*Z_{T_2}(\omega_r^\infty(g_2^\infty)\phi^\infty_2)_L
     \rangle_{X_{L,u},E_u}^\ell
     \end{align*}
     as an element in $\dC\otimes_\dQ\dQ_\ell$;

  \item for every $u\in\tV_E^{(\infty)}$, the local index $I_{T_1,T_2}(\phi^\infty_1,\phi^\infty_2,\rs_1,\rs_2,g_1,g_2)_{L,u}$ to be
     \begin{align*}
     \langle \omega_{r,\infty}(g_{1\infty})\phi^0_\infty(T_1)\cdot\rs_1^*Z_{T_1}(\omega_r^\infty(g_1^\infty)\phi^\infty_1)_L,
     \omega_{r,\infty}(g_{2\infty})\phi^0_\infty(T_2)\cdot\rs_2^*Z_{T_2}(\omega_r^\infty(g_2^\infty)\phi^\infty_2)_L
     \rangle_{X_{L,u},E_u}
     \end{align*}
     as an element in $\dC$.
\end{enumerate}
\end{definition}

Let $(\pi,\cV_\pi)$ be as in Assumption \ref{st:representation}, and assume Hypothesis \ref{hy:modularity} on the modularity of generating functions of codimension $r$.

\begin{remark}\label{re:arithmetic_theta}
In the situation of Definition \ref{de:arithmetic_theta} (and suppose that $F\neq\dQ$), suppose that $L$ has the form $L_\tR L^\tR$ where $L^\tR$ is defined in Notation \ref{st:h}(H8). We have, from \cite{LL}*{Proposition~6.10}, that for every elements $\varphi\in\cV_\pi^{[r]\tR}$ and $\phi^\infty\in\sS(V^r\otimes_{\dA_F}\dA_F^\infty)^L$,
\begin{enumerate}
  \item $\rs^*\Theta_{\phi^\infty}(\varphi)_L=\chi^\tR_\pi(\rs)^\tc\cdot\Theta_{\phi^\infty}(\varphi)_L$ for every $\rs\in\dS^\tR_{\dQ^\ac}$;

  \item $\Theta_{\phi^\infty}(\varphi)_L\in\CH^r(X_L)_\dC^0$;

  \item under \cite{LL}*{Hypothesis~6.6}, $\Theta_{\phi^\infty}(\varphi)_L\in\CH^r(X_L)_\dC^{\langle\ell\rangle}$ for every $\tR$-good rational prime $\ell$.
\end{enumerate}
\end{remark}

We recall the \emph{normalized height pairing} between the cycles $\Theta_{\phi^\infty}(\varphi)$ in Definition \ref{de:arithmetic_theta}, under \cite{LL}*{Hypothesis~6.6}.

\begin{definition}\label{de:natural}
Under \cite{LL}*{Hypothesis~6.6}, for every elements $\varphi_1,\varphi_2\in\cV_\pi^{[r]}$ and $\phi^\infty_1,\phi^\infty_2\in\sS(V^r\otimes_{\dA_F}\dA_F^\infty)$, we define the \emph{normalized height pairing}
\[
\langle\Theta_{\phi^\infty_1}(\varphi_1),\Theta_{\phi^\infty_2}(\varphi_2)\rangle_{X,E}^\natural
\in\dC\otimes_\dQ\dQ_\ell
\]
to be the unique element such that for every $L=L_\tR L^\tR$ as in Remark \ref{re:arithmetic_theta} (with $\tR$ possibly enlarged) satisfying $\varphi_1,\varphi_2\in\cV_\pi^{[r]\tR}$, $\phi^\infty_1,\phi^\infty_2\in\sS(V^r\otimes_{\dA_F}\dA_F^\infty)^L$, and that $\ell$ is $\tR$-good, we have
\[
\langle\Theta_{\phi^\infty_1}(\varphi_1),\Theta_{\phi^\infty_2}(\varphi_2)\rangle_{X,E}^\natural=
\vol^\natural(L)\cdot
\langle\Theta_{\phi^\infty_1}(\varphi_1)_L,\Theta_{\phi^\infty_2}(\varphi_2)_L\rangle_{X_L,E}^\ell,
\]
where $\vol^\natural(L)$ is introduced in \cite{LL}*{Definition~3.8} and $\langle\Theta_{\phi^\infty_1}(\varphi_1)_L,\Theta_{\phi^\infty_2}(\varphi_2)_L\rangle_{X_L,E}^\ell$ is well-defined by Remark \ref{re:arithmetic_theta}(3). Note that by the projection formula, the right-hand side of the above formula is independent of $L$.
\end{definition}

Finally, we review the auxiliary Shimura variety that will \emph{only} be used in the computation of local indices $I_{T_1,T_2}(\phi^\infty_1,\phi^\infty_2,\rs_1,\rs_2,g_1,g_2)_{L,u}$.

\begin{notation}\label{co:auxiliary}
We denote by $\rT_0$ the torus over $\dQ$ such that for every commutative $\dQ$-algebra $R$, we have
$\rT_0(R)=\{a\in E\otimes_\dQ R\res \Nm_{E/F}a \in R^\times\}$.

We choose a CM type $\Phi$ of $E$ containing $\biota$ and denote by $E'$ the subfield of $\dC$ generated by $E$ and the reflex field of $\Phi$. We also choose a skew hermitian space $W$ over $E$ of rank $1$, whose group of rational similitude is canonically $\rT_0$. For a (sufficiently small) open compact subgroup $L_0$ of $\rT_0(\dA^\infty)$, we have the PEL type moduli scheme $Y$ of CM abelian varieties with CM type $\Phi$ and level $L_0$, which is a smooth projective scheme over $E'$ of dimension $0$ (see, \cite{Kot92}, for example). In what follows, when we invoke this construction, the data $\Phi$, $W$, and $L_0$ will be fixed, hence will not be carried into the notation $E'$ and $Y$. For every open compact subgroup $L\subseteq H(\dA_F^\infty)$, we put
\[
X'_L\coloneqq X_L\otimes_EY
\]
as a scheme over $E'$.
\end{notation}

The following notation is parallel to \cite{LL}*{Notation~5.6}.

\begin{notation}\label{st:auxiliary}
In Subsections \ref{ss:split}, \ref{ss:inert}, and \ref{ss:ramified}, we will consider a place $u\in\tV_E^\fin\setminus\tV_F^\heartsuit$ (Definition \ref{de:heart}). Let $p$ be the underlying rational prime of $u$. We will fix an isomorphism $\dC\xrightarrow\sim\ol\dQ_p$ under which $\biota$ induces the place $u$. In particular, we may identify $\Phi$ as a subset of $\Hom(E,\ol\dQ_p)$.

We further require that $\Phi$ in Notation \ref{co:auxiliary} is \emph{admissible} in the following sense: if $\Phi_v\subseteq\Phi$ denotes the subset inducing the place $v$ for every $v\in\tV_F^{(p)}$, then it satisfies
\begin{enumerate}
  \item when $v\in\tV_F^{(p)}\cap\tV_F^\spl$, $\Phi_v$ induces the same place of $E$ above $v$;

  \item when $v\in\tV_F^{(p)}\cap\tV_F^\inert$, $\Phi_v$ is the pullback of a CM type of the maximal subfield of $E_v$ unramified over $\dQ_p$;

  \item when $v\in\tV_F^{(p)}\cap\tV_F^\ram$, the subfield of $\ol\dQ_p$ generated by $E_u$ and the reflex field of $\Phi_v$ is unramified over $E_u$.
\end{enumerate}
To release the burden of notation, we denote by $K$ the subfield of $\ol\dQ_p$ generated by $E_u$ and the reflex field of $\Phi$, by $k$ its residue field, and by $\breve{K}$ the completion of the maximal unramified extension of $K$ in $\ol\dQ_p$ with the residue field $\ol\dF_p$. It is clear that admissible CM type always exists for $u\in\tV_E^\fin\setminus\tV_F^\heartsuit$, and that $K$ is unramified over $E_u$.

We also choose a (sufficiently small) open compact subgroup $L_0$ of $\rT_0(\dA^\infty)$ such that $L_{0,p}$ is maximal compact. We denote by $\cY$ the integral model of $Y$ over $O_K$ such that for every $S\in\Sch'_{/O_K}$, $\cY(S)$ is the set of equivalence classes of quadruples $(A_0,\iota_{A_0},\lambda_{A_0},\eta_{A_0}^p)$ where
\begin{itemize}
  \item $(A_0,\iota_{A_0},\lambda_{A_0})$ is a unitary $O_E$-abelian scheme over $S$ of signature type $\Phi$ (see \cite{LTXZZ}*{Definition~3.4.2 \& Definition~3.4.3})\footnote{Here, our notation on objects is slightly different from \cite{LTXZZ} or \cite{LL} as we, in particular, retrieve the $O_E$-action $\iota_{A_0}$.} such that $\lambda_{A_0}$ is $p$-principal;

  \item $\eta_{A_0}^p$ is an $L_0^p$-level structure (see \cite{LTXZZ}*{Definition~4.1.2} for more details).
\end{itemize}
By \cite{How12}*{Proposition~3.1.2}, $\cY$ is finite and \'{e}tale over $O_K$.
\end{notation}

\subsection{Local indices at split places}
\label{ss:split}

In this subsection, we compute local indices at almost all places in $\tV_E^\spl$. Our goal is to prove the following proposition.

\begin{proposition}\label{pr:index_split}
Let $\tR$, $\tR'$, $\ell$, and $L$ be as in Definition \ref{de:kernel_geometric} such that the cardinality of $\tR'$ is at least $2$. Let $(\pi,\cV_\pi)$ be as in Assumption \ref{st:representation}. For every $u\in\tV_E^\spl$ satisfying $\ul{u}\not\in\tR\setminus\tV_F^\heartsuit$ and $\tV_F^{(p)}\cap\tR\subseteq\tV_F^\spl$ where $p$ is the underlying rational prime of $u$, there exist elements $\rs_1^u,\rs_2^u\in\dS_{\dQ^\ac}^\tR\setminus\fm_\pi^\tR$ such that
\[
I_{T_1,T_2}(\phi^\infty_1,\phi^\infty_2,\rs_1^u\rs_1,\rs_2^u\rs_2,g_1,g_2)^\ell_{L,u}=0
\]
for every $(\tR,\tR',\ell,L)$-admissible sextuple $(\phi^\infty_1,\phi^\infty_2,\rs_1,\rs_2,g_1,g_2)$ and every pair $(T_1,T_2)$ in $\Herm_r^\circ(F)^+$. Moreover, we may take $\rs_1^u=\rs_2^u=1$ if $\ul{u}\not\in\tR$.
\end{proposition}

\begin{proof}
This is simply \cite{LL}*{Proposition~7.1} but without the assumption that $\pi_{\ul{u}}$ is a (tempered) principal series and without relying on \cite{LL}*{Hypothesis~6.6}. The proof is same, after we slightly generalize the construction of the integral model $\cX_m$ to take care of places in $\tV_F^{(p)}\cap\tV_F^\ram$, and use Theorem \ref{th:split_tempered} below which generalizes \cite{LL}*{Lemma~7.3}.
\end{proof}

From now to the end of this section, we assume $\tV_F^{(p)}\cap\tR\subseteq\tV_F^\spl$. We also assume $\ul{u}\in\tV_F^\heartsuit$ and when we need $m\geq 1$ below. We invoke Notation \ref{co:auxiliary} together with Notation \ref{st:auxiliary}. The isomorphism $\dC\xrightarrow\sim\ol\dQ_p$ in Notation \ref{st:auxiliary} identifies $\Hom(E,\dC)$ with $\Hom(E,\dC_p)$. For every $v\in\tV_F^{(p)}$, let $\Phi_v$ be the subset of $\Phi$, regarded as a subset of $\Hom(E,\dC_p)$, of elements that induce the place $v$ of $F$.

For every integer $m\geq 0$, we define a moduli functor $\cX_m$ over $O_K$ as follows: For every $S\in\Sch'_{/O_K}$, $\cX_m(S)$ is the set of equivalence classes of tuples
\[
(A_0,\iota_{A_0},\lambda_{A_0},\eta_{A_0}^p;A,\iota_A,\lambda_A,\eta_A^p,
\{\eta_{A,v}\}_{v\in\tV_F^{(p)}\cap\tV_F^\spl\setminus\{\ul{u}\}},\eta_{A,u,m})
\]
where
\begin{itemize}
  \item $(A_0,\iota_{A_0},\lambda_{A_0},\eta_{A_0}^p)$ is an element in $\cY(S)$;

  \item $(A,\iota_A,\lambda_A)$ is a unitary $O_E$-abelian scheme of signature type $n\Phi-\iota_w+\iota_w^\tc$ over $S$, such that
       \begin{itemize}
         \item for every $v\in\tV_F^{(p)}\setminus\tV_F^\ram$, $\lambda_A[v^\infty]$ is an isogeny whose kernel has order $q_v^{1-\epsilon_v}$;

         \item $\Lie(A[u^{\tc,\infty}])$ is of rank $1$ on which the action of $O_E$ is given by the embedding $\iota_w^\tc$;

         \item for every $v\in\tV_F^{(p)}\cap\tV_F^\ram$, the triple $(A_0[v^\infty],\iota_{A_0}[v^\infty],\lambda_{A_0}[v^\infty])\otimes_{O_K}O_{\breve{K}}$ is an object of $\Exo_{(n,0)}^{\Phi_v}(S\otimes_{O_K}O_{\breve{K}})$ (Remark \ref{re:aexotic}, with $E=E_v$, $F=F_v$, and $\breve{E}=\breve{K}$);\footnote{The sign condition is redundant in our case by \cite{RSZ}*{Remark~5.1(i)}.}
       \end{itemize}

  \item $\eta_A^p$ is an $L^p$-level structure;

  \item for every $v\in\tV_F^{(p)}\cap\tV_F^\spl\setminus\{\ul{u}\}$, $\eta_{A,v}$ is an $L_v$-level structure;

  \item $\eta_{A,u,m}$ is a Drinfeld level-$m$ structure.
\end{itemize}
See \cite{LL}*{Section~7} for more details for the last three items. By \cite{RSZ}*{Theorem~4.5}, for every $m\geq 0$, $\cX_m$ is a regular scheme, flat (smooth, if $m=0$) and projective over $O_K$, and admits a canonical isomorphism
\[
\cX_m\otimes_{O_K}K\simeq X'_{L_{\ul{u},m}L^{\ul{u}}}\otimes_{E'}K
\]
of schemes over $K$. Note that for every integer $m\geq0$, $\dS^{\tR\cup\tV_F^{(p)}}$ naturally gives a ring of \'{e}tale correspondences of $\cX_m$.\footnote{When $m=0$, we do not need $\ul{u}\in\tV_F^\heartsuit$ as the same holds even when $K$ is ramified over $E_u$.}

The following theorem confirms the conjecture proposed in \cite{LL}*{Remark~7.4}, and the rest of this subsection will be devoted to its proof. It is worth mentioning that even in the situation of \cite{LL}*{Lemma~7.3}, the argument below is slightly improved so that \cite{LL}*{Hypothesis~6.6} is not relied on anymore.

\begin{theorem}\label{th:split_tempered}
Let the situation be as in Proposition \ref{pr:index_split} and assume $\ul{u}\in\tV_F^\heartsuit$ and $p\neq\ell$. For every integer $m\geq 0$,
\[
(\rH^{2r}(\cX_m,\dQ_\ell(r))\otimes_\dQ\dQ^\ac)_\fm=0
\]
holds, where $\fm\coloneqq\fm_\pi^\tR\cap\dS^{\tR\cup\tV_F^{(p)}}_{\dQ^\ac}$.
\end{theorem}

We temporarily allow $n$ to be an arbitrary positive integer, not necessarily even. Put $Y_m\coloneqq\cX_m\otimes_{O_K}k$. For every point $x\in Y_m(\ol\dF_p)$, we know that $A_x[u^{\tc,\infty}]$ is a one-dimensional $O_{F_{\ul{u}}}$-divisible group of (relative) height $n$, and we let $0\leq h(x)\leq n-1$ be the height of its \'{e}tale part. For $0\leq h\leq n-1$, let $Y_m^{[h]}$ be locus where $h(x)\leq h$, which is Zariski closed hence will be endowed with the reduced induced scheme structure, and put $Y_m^{(h)}\coloneqq Y_m^{[h]}-Y_m^{[h-1]}$ ($Y_m^{[-1]}=\emptyset$). It is known that $Y_m^{(h)}$ is smooth over $k$ of pure dimension $h$.

Now we suppose that $m\geq 1$. Let $\fS_m^h$ be the set of free $O_{F_{\ul{u}}}/\fp_{\ul{u}}^m$-submodules of $(\fp_{\ul{u}}^{-m}/O_{F_{\ul{u}}})^n$ of rank $n-h$, and put $\fS_m\coloneqq\bigcup_{h=0}^{n-1}\fS_m^h$. For every $M\in\fS_m^h$, we denote by $Y_m^{(M)}\subseteq Y_m^{(h)}$ the (open and closed) locus where the kernel of the Drinfeld level-$m$ structure is $M$. Then we have
\[
Y_m^{(h)}=\coprod_{M\in\fS_m^h}Y_m^{(M)}
\]
for every $0\leq h\leq n-1$. Let $Y_m^{[M]}$ be the scheme-theoretic closure of $Y_m^{(M)}$ inside $Y_m$. Then we have
\begin{align}\label{eq:drinfeld}
Y_m^{[M]}=\bigcup_{\substack{M'\in\fS_m \\ M\subseteq M'}}Y_m^{(M')}
\end{align}
as a disjoint union of strata. Note that Hecke operators away from $\ul{u}$ (of level $L^{\ul{u}}$) preserve $Y_m^{(M)}$ hence $Y_m^{[M]}$ for every $M\in\fS_m$.

We need some general notation. For a sequence $(g_1,\dots,g_t)$ of nonnegative integers with $g=g_1+\cdots+g_t$, we denote by $\rP_{g_1,\dots,g_t}$ the standard upper triangular parabolic subgroup of $\GL_g$ of block sizes $g_1,\dots,g_t$, and $\rM_{g_1,\dots,g_t}$ its standard diagonal Levi subgroup. Moreover, we denote by $C_m^{g_1,\dots,g_t}$ the cardinality of
\[
\GL_{g}(O_{F_{\ul{u}}}/\fp_{\ul{u}}^m)/\rP_{g_1,\dots,g_t}(O_{F_{\ul{u}}}/\fp_{\ul{u}}^m),
\]
which depends only on the partition $g=g_1+\cdots+g_t$. We also put
\[
L_{\ul{u},m}^g\coloneqq\ker\(\GL_g(O_{F_{\ul{u}}})\to\GL_g(O_{F_{\ul{u}}}/\fp_{\ul{u}}^m)\).
\]
For an irreducible admissible representation $\pi$ of $\GL_g(F_{\ul{u}})$ and a positive integer $s$, we have the representation $\Sp_s(\pi)$ of $\GL_{sg}(F_{\ul{u}})$ defined in \cite{HT01}*{Section~I.3}.

\begin{lem}\label{le:middle0}
For $(g_1,\dots,g_t)$ with $g=g_1+\cdots+g_t$ as above and another integer $g'\geq g$, we have
\[
C_m^{g'-g,g} C_m^{g_1,\dots,g_t}=C_m^{g'-g+g_1,g_2,\dots,g_t}.
\]
\end{lem}

\begin{proof}
It follows from the isomorphism
\[
\rP_{g'-g,g}(O_{F_{\ul{u}}}/\fp_{\ul{u}}^m)/\rP_{g'-g+g_1,g_2,\dots,g_t}(O_{F_{\ul{u}}}/\fp_{\ul{u}}^m)
\simeq\GL_g(O_{F_{\ul{u}}}/\fp_{\ul{u}}^m)/\rP_{g_1,\dots,g_t}(O_{F_{\ul{u}}}/\fp_{\ul{u}}^m).
\]
\end{proof}

\begin{lem}\label{le:middle1}
Suppose that $m\geq 1$. Take a sequence $(g_1,\dots,g_t)$ of nonnegative integers with $g=g_1+\cdots+g_t$. Let $\pi_1\boxtimes\cdots\boxtimes\pi_t$ be an admissible representation of $\rM_{g_1,\dots,g_t}(F_{\ul{u}})$. Then we have
\[
\dim\(\Ind_{\rP_{g_1,\dots,g_t}(F_{\ul{u}})}^{\GL_g(F_{\ul{u}})}\pi_1\boxtimes\cdots\boxtimes\pi_t\)^{L_{\ul{u},m}^g}
=C_m^{g_1,\dots,g_t}\prod_{i=1}^t\dim\pi_i^{L_{\ul{u},m}^{g_i}}.
\]
\end{lem}

\begin{proof}
Pick a set $X$ of representatives of the double coset
\[
\rP_{g_1,\dots,g_t}(F_{\ul{u}})\backslash \GL_g(F_{\ul{u}}) / L_{\ul{u},m}^g
\]
contained in $\GL_g(O_{F_{\ul{u}}})$, which is possible by the Iwasawa decomposition. Then an element
\[
f\in\(\Ind_{\rP_{g_1,\dots,g_t}(F_{\ul{u}})}^{\GL_g(F_{\ul{u}})}\pi_1\boxtimes\cdots\boxtimes\pi_t\)^{L_{\ul{u},m}^g}
\]
is determined by $f\res_X$. Since $\GL_g(O_{F_{\ul{u}}})$ normalizes $L_{\ul{u},m}^g$, a function $f'$ on $X$ is of the form $f'=f\res_X$ if and only if $f'$ takes values in $\bigotimes_{i=1}^t\pi_i^{L_{\ul{u},m}^{g_i}}$. As $|X|=C_m^{g_1,\dots,g_t}$, the lemma follows.
\end{proof}

\begin{lem}\label{le:middle2}
Suppose that $m\geq 1$. For every positive integer $g$ and every unramified character $\phi$ of $F_{\ul{u}}^\times$, we have
\[
\sum_{h=0}^g(-1)^h C_m^{g-h,h}\dim\Sp_h(\phi)^{L_{\ul{u},m}^h}=0.
\]
\end{lem}

\begin{proof}
We claim the identity
\begin{align}\label{eq:middle1}
\sum_{h=0}^g(-1)^h
\left[\Ind_{\rP_{h,g-h}(F_{\ul{u}})}^{\GL_g(F_{\ul{u}})}\Sp_h(\phi)\boxtimes\(\phi|\;|_{\ul{u}}^{\frac{g+h-1}{2}}\circ\r{det}_{g-h}\)\right]=0
\end{align}
in $\Groth(\GL_g(F_{\ul{u}}))$. Assuming it, we have
\[
\sum_{h=0}^g(-1)^h\dim\(\Ind_{\rP_{h,g-h}(F_{\ul{u}})}^{\GL_g(F_{\ul{u}})}
\Sp_h(\phi)\boxtimes\(\phi|\;|_{\ul{u}}^{\frac{g+h-1}{2}}\circ\r{det}_{g-h}\)\)^{L_{\ul{u},m}^g}=0.
\]
By Lemma \ref{le:middle1}, the lemma follows.

For the claim, put
\[
\rI(\phi)\coloneqq\Ind_{\rP_{1,\dots,1}(F_{\ul{u}})}^{\GL_g(F_{\ul{u}})}
\phi\boxtimes\phi|\;|_{\ul{u}}\boxtimes\cdots\boxtimes\phi|\;|_{\ul{u}}^{g-1}.
\]
By the transitivity of (normalized) parabolic induction, every irreducible constituent of
\[
\rI(\phi)^{h,g-h}\coloneqq\Ind_{\rP_{h,g-h}(F_{\ul{u}})}^{\GL_g(F_{\ul{u}})}
\Sp_h(\phi)\boxtimes\(\phi|\;|_{\ul{u}}^{\frac{g+h-1}{2}}\circ\r{det}_{g-h}\)
\]
is a constituent of $\rI(\phi)$. By \cite{Zel80}, there is a bijection between the set of irreducible subquotients of $\rI(\phi)$ and the set of sequences of signs of length $g-1$. For such a sequence $\sigma$, we denote by $\rI(\phi)_\sigma$ the corresponding irreducible subquotient. For $0\leq h\leq g-1$, we denote by $\sigma(i)$ the sequence starting from $h$ negative signs followed by $g-1-h$ positive signs. In particular,
\[
\rI(\phi)_{\sigma(g-1)}=\Sp_g(\phi)=\rI(\phi)^{g,0},\qquad
\rI(\phi)_{\sigma(0)}=\phi|\;|_{\ul{u}}^{\frac{2g-1}{2}}\circ\r{det}_g=\rI(\phi)^{0,g}.
\]
By \cite{HT01}*{Lemma~I.3.2}, we have
\[
[\rI(\phi)^{h,g-h}]=[\rI(\phi)_{\sigma(h)}]+[\rI(\phi)_{\sigma(h-1)}]
\]
in $\Groth(\GL_g(F_{\ul{u}}))$ for $0<h<g$. Thus, \eqref{eq:middle1} follows.
\end{proof}

\begin{proposition}\label{pr:middle}
Fix an isomorphism $\ol\dQ_\ell\simeq\dC$. Suppose that $m\geq 1$. For every $0\leq h\leq n-1$ and $M\in\fS_m^h$, we have
\[
\rH^j(Y_m^{[M]}\otimes_k\ol\dF_p,\ol\dQ_\ell)_\fm=0
\]
for every $j\neq h$.
\end{proposition}

This is an extension of \cite{TY07}*{Proposition~4.4}. However, we allow arbitrary principal level at $\ul{u}$ and our case involves endoscopy.

\begin{proof}
In what follows, $h$ will always denote an integer satisfying $0\leq h\leq n-1$. Denote by $\rD_{n-h}$ the division algebra over $F_{\ul{u}}$ of Hasse invariant $\tfrac{1}{n-h}$, with the maximal order $O_{\rD_{n-h}}$.

For a $\fT$-scheme $Y$ of finite type over $k$, and a (finite) character $\chi\colon\rT_0(\dQ)\backslash\rT_0(\dA^\infty)/L_0\to\ol\dQ_\ell^\times$, we put
\[
[\rH_{?,\chi}(Y,\ol\dQ_\ell)]\coloneqq\sum_{j\in\dZ}(-1)^j\rH^j_?(Y\otimes_k\ol\dF_p,\ol\dQ_\ell)[\chi]
\]
as an element in $\Groth(\Gal(\ol\dF_p/k))$ for $?\in\{\;,c\}$.

Let $I^h_m$ be the Igusa variety (of the first kind) introduced in \cite{HT01}*{Section~IV.1} so that $I^h_m$ is isomorphic to $Y_m^{(M)}$ for every $M\in\fS_m^h$ as schemes over $k$ (but not as schemes over $Y_0^{(h)}$). Combining with \eqref{eq:drinfeld}, we obtain the identity
\begin{align}\label{eq:middle5}
[\rH_\chi(Y_m^{[M]},\ol\dQ_\ell)]
&=\sum_{h'=0}^h\sum_{\substack{M'\in\fS_m^{h'} \\ M\subseteq M'}}(-1)^{h-h'}[\rH_{c,\chi}(Y_m^{(M')},\ol\dQ_\ell)] \\
&=\sum_{h'=0}^h(-1)^{h-h'}\cdot\left|\{M'\in\fS_m^{h'} \res M\subseteq M'\}\right|\cdot [\rH_{c,\chi}(I_m^{h'},\ol\dQ_\ell)] \notag \\
&=\sum_{h'=0}^h(-1)^{h-h'}C_m^{h-h',h'}\cdot [\rH_{c,\chi}(I_m^{h'},\ol\dQ_\ell)] \notag
\end{align}
in $\Groth(\Gal(\ol\dF_p/k))$.

Now to compute $[\rH_\chi(I_m^{h'},\ol\dQ_\ell)]$, we use \cite{CS17}*{Lemma~5.5.1} in which the corresponding $J_b(\dQ_p)$ is $\rD_{n-h'}\times\GL_{h'}(F_{\ul{u}})$, and we take $\phi=\phi^{\ul{u}}\phi_{\ul{u}}$ where $\phi^{\ul{u}}$ is the characteristic function of $L^{\ul{u}}$ and $\phi_{\ul{u}}$ is the characteristic function of $O_{\rD_{n-h'}}^\times\times L_{\ul{u},m}^{h'}$. Then we have the identity
\begin{align}\label{eq:middle6}
[\rH_{c,\chi}(I_m^{h'},\ol\dQ_\ell)]=\sum_{\bbn}\sum_{\Pi^{\bbn}}c(\bbn,\Pi^{\bbn})\cdot \Red^{h'}_{\bbn}(\pi^{\bbn}_{\ul{u}})^{O_{\rD_{n-h'}}^\times\times L_{\ul{u},m}^{h'}}
\end{align}
in $\Groth(\rD_{n-h'}^\times/O_{\rD_{n-h'}}^\times)$, where
\begin{itemize}
  \item $\bbn$ runs through \emph{ordered} pairs $(n_1,n_2)$ of nonnegative integers such that $n_1+n_2=n$, which gives an elliptic endoscopic group $G_{\bbn}$ of $\rU(\pres{\bu}{V})$;

  \item $\Pi^{\bbn}$ runs through a finite set of certain isobaric irreducible cohomological (with respect to the trivial algebraic representation) automorphic representations of $\dG_{\bbn}(\dA_F)$, with $\pi^{\bbn}_{\ul{u}}$ the descent of $\Pi^{\bbn}_{\ul{u}}$ to $G_{\bbn}(F_{\ul{u}})\simeq\rM_{n_1,n_2}(F_{\ul{u}})$;

  \item $c(\bbn,\Pi^{\bbn})$ is a constant depending only on $\bbn$ and $\Pi^{\bbn}$ but not on $h'$;

  \item $\Red^{h'}_{\bbn}\colon\Groth(\rM_{n_1,n_2}(F_{\ul{u}}))\to\Groth(\rD_{n-h'}^\times\times\GL_{h'}(F_{\ul{u}}))$ is the zero map if $h'<n_2$, and otherwise is the composition of
      \begin{itemize}
        \item $\Groth(\rM_{n_1,n_2}(F_{\ul{u}}))\to\Groth(\rM_{n-h',h'-n_2,n_2}(F_{\ul{u}}))$, which is the normalized Jacquet functor,

        \item $\Groth(\rM_{n-h',h'-n_2,n_2}(F_{\ul{u}}))\to\Groth(\rM_{n-h',h'}(F_{\ul{u}}))$, which is the normalized parabolic induction, and

        \item $\Groth(\rM_{n-h',h'}(F_{\ul{u}}))\to\Groth(\rD_{n-h'}^\times\times\GL_{h'}(F_{\ul{u}}))$, which is the Langlands--Jacquet map (on the first factor).
      \end{itemize}
\end{itemize}
The image of $[\rH_{c,\chi}(I_m^{h'},\ol\dQ_\ell)]$ in $\Groth(\Gal(\ol\dF_p/k))$ is given by the map
\[
\Groth(\rD_{n-h'}^\times/O_{\rD_{n-h'}}^\times)\to\Groth(\Gal(\ol\dF_p/k))
\]
sending an (unramified) character $\phi\circ\Nm_{\rD_{n-h'}^\times}$ to $\rec(\phi^{-1})\cdot\breve\chi$, where $\breve\chi$ is a finite character of $\Gal(\ol\dF_p/k)$ determined by $\chi$. In what follows, we will regard
\[
\Red^{h'}_{\bbn}(\pi^{\bbn}_{\ul{u}})^{O_{\rD_{n-h'}}^\times\times L_{\ul{u},m}^{h'}}
\]
as an element of $\Groth(\Gal(\ol\dF_p/k))$ via the above map.

Now let us compute for each $\bbn=(n_1,n_2)$,
\begin{align}\label{eq:middle2}
\sum_{h'=0}^h(-1)^{h-h'}C_m^{h-h',h'}\cdot\Red^{h'}_{\bbn}(\pi^{\bbn}_{\ul{u}})^{O_{\rD_{n-h'}}^\times\times L_{\ul{u},m}^{h'}}
\end{align}
in $\Groth(\Gal(\ol\dF_p/k))$, when $\pi^{\bbn}_{\ul{u}}$ is tempered. Write $\pi^{\bbn}_{\ul{u}}=\pi^1\boxtimes\pi^2$ where $\pi^\alpha$ is an tempered irreducible admissible representation of $\GL_{n_\alpha}(F_{\ul{u}})$. In particular, $\pi^1$ is a full induction of the form
\[
\Ind_{\rP_{s_1g_1,\dots,s_tg_t}(F_{\ul{u}})}^{\GL_{n_1}(F_{\ul{u}})}\Sp_{s_1}(\pi_1^1)\boxtimes\cdots\boxtimes\Sp_{s_t}(\pi_t^1),
\]
where $s_1,\dots,s_t$ and $g_1,\dots,g_t$ are positive integers satisfying $s_1g_1+\cdots+s_tg_t=n_1$; and for $1\leq i\leq t$, $\pi_i^1$ is an irreducible cuspidal representation of $\GL_{g_i}(F_{\ul{u}})$ such that $\Sp_{s_i}(\pi_i^1)$ is unitary. Let $\dI$ be the subset of $\{1,\dots,t\}$ such that $\pi_i^1$ is an unramified character (hence $g_i=1$) and $s_i\geq n-h$. Then we have for $h'\geq n_2$,
\begin{align}\label{eq:middle3}
&\Red^{h'}_{\bbn}(\pi^{\bbn}_{\ul{u}})^{O_{\rD_{n-h'}}^\times\times L_{\ul{u},m}^{h'}} \\
&=\sum_{\substack{i\in\dI \\ s_i\geq n-h'}}\dim\(\Ind_{\rP_?(F_{\ul{u}})}^{\GL_{h'}(F_{\ul{u}})}\Sp_{s_i+h'-n}(\pi_i^1)\boxtimes
\(\boxtimes_{j\neq i}\Sp_{s_j}(\pi_j^1)\)\boxtimes\pi^2\)^{L_{\ul{u},m}^{h'}}
[\rec((\pi_i^1)^{-1}|\;|_{\ul{u}}^{\frac{1-n}{2}})\cdot\breve\chi] \notag
\end{align}
in which the suppressed subscript in $\rP_?$ is $(s_i+h'-n,s_1g_1,\dots,\widehat{s_ig_i},\dots,s_tg_t,n_2)$.

We claim that for each $i\in\dI$,
\begin{align}\label{eq:middle4}
\sum_{h'=n-s_i}^h(-1)^{h-h'}C_m^{h-h',h'}\dim\(\Ind_{\rP_?(F_{\ul{u}})}^{\GL_{h'}(F_{\ul{u}})}\Sp_{s_i+h'-n}(\pi_i^1)\boxtimes
\(\boxtimes_{j\neq i}\Sp_{s_j}(\pi_j^1)\)\boxtimes\pi^2\)^{L_{\ul{u},m}^{h'}}=0
\end{align}
if $s_i>n-h$. In fact, by Lemma \ref{le:middle1}, there is a nonnegative integer $D$ independent of $h'$ such that the left-hand side of \eqref{eq:middle4} equals
\begin{align*}
&\quad\sum_{h'=n-s_i}^h(-1)^{h-h'}C_m^{h-h',h'}\cdot C_m^{s_i+h'-n,s_1g_1,\dots,\widehat{s_ig_i},\dots,s_tg_t,n_2}\cdot D\cdot
\dim\Sp_{s_i+h'-n}(\pi_i^1)^{L_{\ul{u},m}^{s_i+h'-n}} \\
&=\sum_{h'=n-s_i}^h(-1)^{h-h'}C_m^{h-h',s_i+h'-n,s_1g_1,\dots,\widehat{s_ig_i},\dots,s_tg_t,n_2}\cdot D\cdot
\dim\Sp_{s_i+h'-n}(\pi_i^1)^{L_{\ul{u},m}^{s_i+h'-n}} \\
&=\sum_{h'=0}^{h+s_i-n}(-1)^{h-h'}C_m^{h+s_i-n-h',h',s_1g_1,\dots,\widehat{s_ig_i},\dots,s_tg_t,n_2}\cdot D\cdot
\dim\Sp_{h'}(\pi_i^1)^{L_{\ul{u},m}^{h'}} \\
&=(-1)^hC_m^{h+s_i-n,s_1g_1,\dots,\widehat{s_ig_i},\dots,s_tg_t,n_2}\cdot D\sum_{h'=0}^{h+s_i-n}(-1)^{h'} C_m^{h+s_i-n-h',h'}\dim\Sp_{h'}(\pi_i^1)^{L_{\ul{u},m}^{h'}}
\end{align*}
in which the last summation vanishes by applying Lemma \ref{le:middle2} with $g=h+s_i-n>0$. Here, we have used Lemma \ref{le:middle0} twice.

By \eqref{eq:middle3} and \eqref{eq:middle4}, we know that \eqref{eq:middle2} is a linear combination of $[\rec((\pi_i^1)^{-1}|\;|_{\ul{u}}^{\frac{1-n}{2}})\cdot\breve\chi]$ with $i\in\dI$ satisfying $s_i=n-h$. Thus, \eqref{eq:middle2} is strictly pure of weight $h$ since $\Sp_{s_i}(\pi_i^1)$ is unitary. By \eqref{eq:middle5}, \eqref{eq:middle6}, and the fact that localization at $\fm$ annihilates all terms in \eqref{eq:middle6} with $\pi^{\bbn}_{\ul{u}}$ not tempered, we know that $[\rH_\chi(Y_m^{[M]},\ol\dQ_\ell)]_\fm$ is strictly pure of weight $h$. Finally, by \cite{Man08}*{Proposition~12}, we know that $Y_m^{[M]}$ is smooth over $k$ of pure dimension $h$. Since $Y_m^{[M]}$ is also proper, we have
\[
\rH^j(Y_m^{[M]}\otimes_k\ol\dF_p,\ol\dQ_\ell)[\chi]_\fm=0
\]
for every $j\neq h$ and every character $\chi\colon\rT_0(\dQ)\backslash\rT_0(\dA^\infty)/L_0\to\ol\dQ_\ell^\times$ from the Weil conjecture. Then the proposition follows.
\end{proof}

\begin{proof}[Proof of Theorem \ref{th:split_tempered}]
We may assume $m\geq 1$ since the morphism $\cX_m\to\cX_0$ is finite and flat. In what follows, $h$ is always an integer satisfying $0\leq h\leq n-1=2r-1$. For a subset $\Sigma\subset\fS_m^h$, we put
\[
Y_m^{(\Sigma)}\coloneqq\bigcup_{M\in\Sigma}Y_m^{(M)}, \qquad
Y_m^{[\Sigma]}\coloneqq\bigcup_{M\in\Sigma}Y_m^{[M]}
\]
in which the first union is disjoint. If $h\geq 1$, we also denote by $\Sigma^\dag$ the subset of $\fS_m^{h-1}$ consisting of $M'$ that contains an element in $\Sigma$.

Fix an arbitrary isomorphism $\ol\dQ_\ell\simeq\dC$. We claim
\begin{itemize}
  \item[($*$)] For every $0\leq h\leq 2r-1$ and every $\Sigma\subset\fS_m^h$,
     \begin{align*}
     \rH^j_c(Y_m^{(\Sigma)}\otimes_k\ol\dF_p,\ol\dQ_\ell)_\fm=
     \rH^j(Y_m^{[\Sigma]}\otimes_k\ol\dF_p,\ol\dQ_\ell)_\fm=0
     \end{align*}
     holds when $j>h$.
\end{itemize}

Assuming the claim, we prove $\rH^{2r}(\cX_m,\ol\dQ_\ell(r))_\fm=0$. By the proper base change theorem and the fact that taking global sections on $\Spec O_K$ is the same as restricting to $\Spec k$ and then taking global sections, the natural map $\rH^{2r}(\cX_m,\ol\dQ_\ell(r))\to\rH^{2r}(Y_m,\ol\dQ_\ell(r))$ is an isomorphism. Thus, it suffices to show that
\[
\rH^0(k,\rH^{2r}(Y_m\otimes_k\ol\dF_p,\ol\dQ_\ell(r)))_\fm=
\rH^1(k,\rH^{2r-1}(Y_m\otimes_k\ol\dF_p,\ol\dQ_\ell(r)))_\fm=0.
\]
The vanishing of $\rH^0(k,\rH^{2r}(Y_m\otimes_k\ol\dF_p,\ol\dQ_\ell(r)))_\fm$ already follows from ($*$) as $Y_m=Y_m^{[2r-1]}$. Now we consider $\rH^1(k,\rH^{2r-1}(Y_m\otimes_k\ol\dF_p,\ol\dQ_\ell(r)))_\fm$. By ($*$), $\rH^{2r-1}(Y_m^{[2r-2]}\otimes_k\ol\dF_p,\ol\dQ_\ell)_\fm=0$ hence the natural map
\[
\rH^{2r-1}_c(Y_m^{(2r-1)}\otimes_k\ol\dF_p,\ol\dQ_\ell)_\fm\to\rH^{2r-1}(Y_m\otimes_k\ol\dF_p,\ol\dQ_\ell)_\fm
\]
is surjective. It suffices to show that $\rH^1(k,\rH^{2r-1}_c(Y_m^{(2r-1)}\otimes_k\ol\dF_p,\ol\dQ_\ell(r)))_\fm=0$. Now we prove by induction on $h$ that for every $M\in\fS_m^h$, $\rH^1(k,\rH^h_c(Y_m^{(M)}\otimes_k\ol\dF_p,\ol\dQ_\ell(r)))_\fm=0$.

The case $h=0$ is trivial. Consider $h>0$ and $M\in\fS_m^h$. Since $Y_m^{[M]}$ is proper smooth over $k$ by \cite{Man08}*{Proposition~12}, we have $\rH^1(k,\rH^h(Y_m^{[M]}\otimes_k\ol\dF_p,\ol\dQ_\ell(r)))_\fm=0$ by the Weil conjecture. By ($*$), we have $\rH^h(Y_m^{[\{M\}^\dag]}\otimes_k\ol\dF_p,\ol\dQ_\ell)_\fm=0$. Thus, it suffices to show that $\rH^1(k,\rH^{h-1}(Y_m^{[\{M\}^\dag]}\otimes_k\ol\dF_p,\ol\dQ_\ell(r)))_\fm=0$. By ($*$) again, we have $\rH^{h-1}(Y_m^{[\{M\}^{\dag\dag}]}\otimes_k\ol\dF_p,\ol\dQ_\ell)_\fm=0$. Thus, the desired vanishing property follows from
\[
\rH^1(k,\rH^{h-1}_c(Y_m^{(\{M\}^\dag)}\otimes_k\ol\dF_p,\ol\dQ_\ell(r)))_\fm=
\bigoplus_{M'\in\{M\}^\dag}\rH^1(k,\rH^{h-1}_c(Y_m^{(M')}\otimes_k\ol\dF_p,\ol\dQ_\ell(r)))_\fm=0,
\]
which holds by the induction hypothesis. We have now proved $\rH^{2r}(\cX_m,\ol\dQ_\ell(r))_\fm=0$ assuming ($*$).

To show the claim ($*$), we use induction on $h$. To ease notation, we simply write $\rH^\bullet_?(-)$ for $\rH^\bullet_?(-\otimes_k\ol\dF_p,\ol\dQ_\ell)_\fm$ for $?\in\{\;,c\}$. The case for $h=0$ is trivial. Suppose that we know ($*$) for $h-1$ for some $h\geq 1$. For every $M\in\fS_m^h$, we have the exact sequence
\[
\cdots\to\rH^{j-1}(Y^{[\{M\}^\dag]}_m)\to\rH^j_c(Y^{(M)}_m)\to\rH^j(Y^{[M]}_m)\to\cdots
\]
By Proposition \ref{pr:middle} and the induction hypothesis, we have $\rH^j_c(Y^{(M)}_m)=0$ for $j>h$. Now take a subset $\Sigma$ of $\fS_m^h$. Then we have $\rH^j_c(Y^{(\Sigma)}_m)=\bigoplus_{M\in\Sigma}\rH^j_c(Y^{(M)}_m)=0$ for $j> h$. By the exact sequence
\[
\cdots\to\rH^j_c(Y^{(\Sigma)}_m)\to\rH^j(Y^{[\Sigma]}_m)\to\rH^j(Y^{[\Sigma^\dag]}_m)\to\cdots
\]
and the induction hypothesis, we have $\rH^j(Y^{[\Sigma]}_m)=0$ for $j>h$. Thus, ($*$) holds for $h$.

The theorem is proved.
\end{proof}

\subsection{Local indices at inert places}
\label{ss:inert}

In this subsection, we compute local indices at places in $\tV_E^\inert$ not above $\tR$.

\begin{proposition}\label{pr:index_inert}
Let $\tR$, $\tR'$, $\ell$, and $L$ be as in Definition \ref{de:kernel_geometric}. Take an element $u\in\tV_E^\inert$ such that its underlying rational prime $p$ is odd and satisfies $\tV_F^{(p)}\cap\tR\subseteq\tV_F^\spl$.
\begin{enumerate}
  \item Suppose that $\ul{u}\not\in\tS$. Then we have
      \[
      \log q_u\cdot \vol^\natural(L)\cdot I_{T_1,T_2}(\phi^\infty_1,\phi^\infty_2,\rs_1,\rs_2,g_1,g_2)^\ell_{L,u}=
      \fE_{T_1,T_2}((g_1,g_2),\Phi_\infty^0\otimes(\rs_1\phi^\infty_1\otimes(\rs_2\phi^\infty_2)^\tc))_u
      \]
      for every $(\tR,\tR',\ell,L)$-admissible sextuple $(\phi^\infty_1,\phi^\infty_2,\rs_1,\rs_2,g_1,g_2)$ and every pair $(T_1,T_2)$ in $\Herm_r^\circ(F)^+$.

  \item Suppose that $\ul{u}\in\tS\cap\tV_F^\heartsuit$ and is unramified over $\dQ$. Recall that we have fixed a $u$-nearby space $\pres{u}V$ and an isomorphism $\pres{u}{V}\otimes_F\dA_F^{\ul{u}}\simeq V\otimes_{\dA_F}\dA_F^{\ul{u}}$ from Notation \ref{st:h}(H9). We also fix a $\psi_{E,\ul{u}}$-self-dual lattice $\Lambda^\star_{\ul{u}}$ of $\pres{u}{V_{\ul{u}}}$. Then there exist elements $\rs_1^u,\rs_2^u\in\dS_{\dQ^\ac}^\tR\setminus\fm_\pi^\tR$ such that
      \begin{multline*}
      \log q_u\cdot \vol^\natural(L)\cdot I_{T_1,T_2}(\phi^\infty_1,\phi^\infty_2,\rs_1^u\rs_1,\rs_2^u\rs_2,g_1,g_2)^\ell_{L,u} \\
      =\fE_{T_1,T_2}((g_1,g_2),\Phi_\infty^0\otimes(\rs_1^u\rs_1\phi^\infty_1\otimes(\rs_2^u\rs_2\phi^\infty_2)^\tc))_u \\
      -\frac{\log q_u}{q_u^r-1}E_{T_1,T_2}((g_1,g_2),\Phi_\infty^0\otimes
      (\rs_1^u\rs_1\phi^{\infty,\ul{u}}_1\otimes(\rs_2^u\rs_2\phi^{\infty,\ul{u}}_2)^\tc)\otimes\CF_{(\Lambda^\star_{\ul{u}})^{2r}})
      \end{multline*}
      for every $(\tR,\tR',\ell,L)$-admissible sextuple $(\phi^\infty_1,\phi^\infty_2,\rs_1,\rs_2,g_1,g_2)$ and every pair $(T_1,T_2)$ in $\Herm_r^\circ(F)^+$.
\end{enumerate}
In both cases, the right-hand side is defined in Definition \ref{de:analytic_kernel} with the Gaussian function $\Phi_\infty^0\in\sS(V^{2r}\otimes_{\dA_F}F_\infty)$ (Notation \ref{st:h}(H3)), and $\vol^\natural(L)$ is defined in \cite{LL}*{Definition~3.8}.
\end{proposition}

\begin{proof}
Part (1) is proved in the same way as \cite{LL}*{Proposition~8.1}. Part (2) is proved in the same way as \cite{LL}*{Proposition~9.1}. Note that we need to extend the definition of the integral model due to the presence of places in $\tV_F^{(p)}\cap\tV_F^\ram$, as we do in the previous subsection. The requirement that $\ul{u}\in\tV_F^\heartsuit$ in (2) is to ensure that $K$ is unramified over $E_u$ (see Notation \ref{st:auxiliary}).
\end{proof}

\subsection{Local indices at ramified places}
\label{ss:ramified}

In this subsection, we compute local indices at places in $\tV_E^\ram$ not above $\tR$.

\begin{proposition}\label{pr:index_ramified}
Let $\tR$, $\tR'$, $\ell$, and $L$ be as in Definition \ref{de:kernel_geometric}. Take an element $u\in\tV_E^\ram$ such that its underlying rational prime $p$ satisfies $\tV_F^{(p)}\cap\tR\subseteq\tV_F^\spl$. Then we have
\[
\log q_u\cdot \vol^\natural(L)\cdot I_{T_1,T_2}(\phi^\infty_1,\phi^\infty_2,\rs_1,\rs_2,g_1,g_2)^\ell_{L,u}=
\fE_{T_1,T_2}((g_1,g_2),\Phi_\infty^0\otimes(\rs_1\phi^\infty_1\otimes(\rs_2\phi^\infty_2)^\tc))_u
\]
for every $(\tR,\tR',\ell,L)$-admissible sextuple $(\phi^\infty_1,\phi^\infty_2,\rs_1,\rs_2,g_1,g_2)$ and every pair $(T_1,T_2)$ in $\Herm_r^\circ(F)^+$, where the right-hand side is defined in Definition \ref{de:analytic_kernel} with the Gaussian function $\Phi_\infty^0\in\sS(V^{2r}\otimes_{\dA_F}F_\infty)$ (Notation \ref{st:h}(H3)), and $\vol^\natural(L)$ is defined in \cite{LL}*{Definition~3.8}.
\end{proposition}

\begin{proof}
The proof of the proposition follows the same line as in \cite{LL}*{Proposition~8.1}, as long as we accomplish the following three tasks. We invoke Notation \ref{co:auxiliary} together with Notation \ref{st:auxiliary}.
\begin{enumerate}
  \item Construct a good integral model $\cX_{\tilde{L}}$ for $X_{\tilde{L}}$ over $O_K$ for open compact subgroups $\tilde{L}\subseteq L$ satisfying $\tilde{L}_v=L_v$ for $v\in\tV_F^{(p)}\setminus\tV_F^\spl$, which is provided after the proof.

  \item Establish the nonarchimedean uniformization of $\cX_{\tilde{L}}$ along the supersingular locus using the relative Rapoport--Zink space $\cN$ from Definition \ref{de:rz}, analogous to \cite{LL}*{(8.2)}, and compare special divisors. This is done in Proposition \ref{pr:uniformization} below.

  \item Show that for $x=(x_1,\dots,x_{2r})\in \pres{\ul{u}}{V}^{2r}$ with $T(x)\in\Herm_{2r}^\circ(F_{\ul{u}})$, we have
      \[
      \chi\(\sO_{\cN(x_1)}\overset{\dL}\otimes_{\sO_\cN}\cdots\overset{\dL}\otimes_{\sO_\cN}\sO_{\cN(x_{2r})}\)=\frac{b_{2r,\ul{u}}(0)}{\log q_u}W'_{T^\Box}(0,1_{4r},\CF_{(\Lambda^\tR_{\ul{u}})^{2r}})
      \]
      if $T(x)=T^\Box$. In fact, this follows from Theorem \ref{th:kr}, Remark \ref{re:whittaker}, and the identity
      \[
      b_{2r,\ul{u}}(0)=\prod_{i=1}^r(1-q_u^{-2i}).
      \]
\end{enumerate}
The proposition is proved.
\end{proof}

Let the situation be as in Proposition \ref{pr:index_ramified}. The isomorphism $\dC\xrightarrow\sim\ol\dQ_p$ in Notation \ref{st:auxiliary} identifies $\Hom(E,\dC)$ with $\Hom(E,\dC_p)$. For every $v\in\tV_F^{(p)}$, let $\Phi_v$ be the subset of $\Phi$, regarded as a subset of $\Hom(E,\dC_p)$, of elements that induce the place $v$ of $F$.

To ease notation, put
\[
\tU\coloneqq\{v\in\tV_F^{(p)}\setminus\tV_F^\spl\res v\neq\ul{u}\}.
\]
In particular, $\tU\cap\tR=\emptyset$.

There is a projective system $\{\cX_{\tilde{L}}\}$, for open compact subgroups $\tilde{L}\subseteq L$ satisfying $\tilde{L}_v=L_v$ for $v\in\tV_F^{(p)}\setminus\tV_F^\spl$, of smooth projective schemes over $O_K$ (see \cite{RSZ}*{Theorem~4.7, AT type (2)}) with
\[
\cX_{\tilde{L}}\otimes_{O_K}K=X'_{\tilde{L}}\otimes_{E'}K
=\(X_{\tilde{L}}\otimes_{E}Y\)\otimes_{E'}K,
\]
and finite \'{e}tale transition morphisms, such that for every $S\in\Sch'_{/O_K}$, $\cX_{\tilde{L}}(S)$ is the set of equivalence classes of tuples
\begin{align*}
(A_0,\iota_{A_0},\lambda_{A_0},\eta_{A_0}^p;A,\iota_A,\lambda_A,\eta_A^p,\{\eta_{A,v}\}_{v\in\tV_F^{(p)}\cap\tV_F^\spl})
\end{align*}
where
\begin{itemize}
  \item $(A_0,\iota_{A_0},\lambda_{A_0},\eta_{A_0}^p)$ is an element in $\cY(S)$;

  \item $(A,\iota_A,\lambda_A)$ is a unitary $O_E$-abelian scheme of signature type $n\Phi-\iota_w+\iota_w^\tc$ over $S$, such that
       \begin{itemize}
         \item for every $v\in\tV_F^{(p)}\setminus\tV_F^\ram$, $\lambda_A[v^\infty]$ is an isogeny whose kernel has order $q_v^{1-\epsilon_v}$;

         \item for every $v\in\tU\cap\tV_F^\ram$, the triple $(A_0[v^\infty],\iota_{A_0}[v^\infty],\lambda_{A_0}[v^\infty])\otimes_{O_K}O_{\breve{K}}$ is an object of $\Exo_{(n,0)}^{\Phi_v}(S\otimes_{O_K}O_{\breve{K}})$ (Remark \ref{re:aexotic}, with $E=E_v$, $F=F_v$, and $\breve{E}=\breve{K}$);

         \item for $v=\ul{u}$, $(A_0[v^\infty],\iota_{A_0}[v^\infty],\lambda_{A_0}[v^\infty])\otimes_{O_K}O_{\breve{K}}$ is an object of $\Exo_{(n-1,1)}^{\Phi_v}(S\otimes_{O_K}O_{\breve{K}})$ (Definition \ref{de:aexotic}, with $E=E_v$, $F=F_v$, and $\breve{E}=\breve{K}$);
       \end{itemize}

  \item $\eta_A^p$ is an $\tilde{L}^p$-level structure;

  \item for every $v\in\tV_F^{(p)}\cap\tV_F^\spl$, $\eta_{A,v}$ is an $\tilde{L}_v$-level structure.
\end{itemize}
In particular, $\dS^\tR$ is naturally a ring of \'{e}tale correspondences of $\cX_L$.

Let $\phi^\infty\in\sS(V\otimes_{\dA_F}\dA_F^\infty)^{\tilde{L}}$ be a $p$-basic element \cite{LL}*{Definition~6.5}. For every element $t\in F$ that is totally positive, we have a cycle $\cZ_t(\phi^\infty)_{\tilde{L}}\in\rZ^1(\cX_{\tilde{L}})$ extending the restriction of $Z_t(\phi^\infty)$ to $X'_{\tilde{L}}$, defined similarly as in \cite{LZ}*{Section~13.3}.

Now we study the nonarchimedean uniformization of $\cX_{\tilde{L}}$ along the supersingular locus. Fix a point $P_0\coloneqq(A_0,\iota_{A_0},\lambda_{A_0},\eta_{A_0}^p)\in\cY(O_{\breve{K}})$. Put
\[
\cX\coloneqq\varprojlim_{\tilde{L}}\cX_{\tilde{L}}
\]
and denote by $\cX_0$ the fiber of $P_0$ along the natural projection $\cX\to\cY$. Let $\cX_0^\wedge$ be the completion along the (closed) locus where $A[u^\infty]$ is supersingular, as a formal scheme over $\Spf O_K$. Also fix a point $\bbP\in\cX_0^\wedge(\ol\dF_p)$ represented by $(P_0\otimes_{O_{\breve{K}}}\ol\dF_p;\bbA,\iota_{\bbA},\lambda_{\bbA},\eta_{\bbA}^p,\{\eta_{\bbA,v}\}_{v\in\tV_F^{(p)}\cap\tV_F^\spl})$.

Put $\bbV\coloneqq\Hom_{O_E}(A_0\otimes_{O_{\breve{E}}}\ol\dF_p,\bbA)\otimes\dQ$. Fixing an element $\varpi\in O_F$ that has valuation $0$ (resp.\ $1$) at places in $\tU\cap\tV_F^\inert$ (resp., $\tU\cap\tV_F^\ram$), we have a pairing
\begin{align*}
(\;,\;)_{\bbV}\colon\bbV\times\bbV\to E
\end{align*}
sending $(x,y)\in\bbV^2$ to the composition of quasi-homomorphisms
\[
A_0\xrightarrow{x}\bbX\xrightarrow{\lambda_{\bbA}}\bbA^\vee\xrightarrow{y^\vee}A_0^\vee\xrightarrow{\varpi^{-1}\lambda_{A_0}^{-1}}A_0
\]
as an element in $\End_{O_E}(A_0)\otimes\dQ$ hence in $E$ via $\iota_{A_0}^{-1}$. We have the following properties concerning $\bbV$:
\begin{itemize}
  \item $\bbV,(\;,\;)_{\bbV}$ is a totally positive definite hermitian space over $E$ of rank $n$;

  \item for every $v\in\tV_F^\fin\setminus(\tV_F^{(p)}\setminus\tV_F^\spl)$, we have a canonical isometry $\bbV\otimes_FF_v\simeq V\otimes_FF_v$ of hermitian spaces;

  \item for every $v\in\tU$, the $O_{E_v}$-lattice $\bbLambda_v\coloneqq\Hom_{O_E}(A_0\otimes_{O_{\breve{E}}}\ol\dF_p,\bbA)\otimes_{O_F}O_{F_v}$ is
      \begin{itemize}
        \item self-dual if $v\in\tU\cap\tV_F^\inert$ and $\epsilon_v=1$,

        \item almost self-dual if $v\in\tU\cap\tV_F^\inert$ and $\epsilon_v=-1$,

        \item self-dual if $v\in\tU\cap\tV_F^\ram$;
      \end{itemize}

  \item $\bbV\otimes_FF_{\ul{u}}$ is nonsplit, and we have a canonical isomorphism
      \[
      \bbV\otimes_FF_{\ul{u}}\simeq\Hom_{O_{E_u}}(A_0[u^\infty]\otimes_{O_{\breve{K}}}\ol\dF_p,\bbA[u^\infty])\otimes\dQ
      \]
      of hermitian spaces over $E_u$.
\end{itemize}

We have a Rapoport--Zink space $\cN$ (Definition \ref{de:rz}, with $E=E_u$, $F=F_{\ul{u}}$, $\breve{E}=\breve{K}$, and $\varphi_0$ the natural embedding) with respect to the object
\[
(\bbX,\iota_{\bbX},\lambda_{\bbX})\coloneqq(\bbA[u^\infty],\iota_{\bbA}[u^\infty],\lambda_{\bbA}[u^\infty])^\rel
\in\Exo_{(n-1,1)}^\rb(\ol\dF_p),
\]
where $-^\rel$ is the morphism \eqref{eq:arz}. We now construct a morphism
\begin{align}\label{eq:uniformization}
\Upsilon^\rel\colon\cX_0^\wedge\to\rU(\bbV)(F)\backslash\(\cN\times\rU(\bbV)(\dA_F^{\infty,\ul{u}})/\prod_{v\in\tU}\bbL_v\)
\end{align}
of formal schemes over $\Spf O_{\breve{K}}$, where $\bbL_v$ is the stabilizer of $\bbLambda_v$ in $\rU(\bbV)(F_v)$, as follows.

We have the Rapoport--Zink space $\cN^{\Phi_u}=\cN^{\Phi_u}_{(\bbA[u^\infty],\iota_{\bbA}[u^\infty],\lambda_{\bbA}[u^\infty])}$ from Definition \ref{de:arz}. We first define a morphism
\[
\Upsilon\colon\cX_0^\wedge\to\rU(\bbV)(F)\backslash\(\cN^{\Phi_u}\times\rU(\bbV)(\dA_F^{\infty,\ul{u}})/\prod_{v\in\tU}\bbL_v\),
\]
and then define $\Upsilon^\rel$ as the composition of $\Upsilon$ with the morphism in Corollary \ref{co:arz}. To construct $\Upsilon$, we take a point
\[
P=(P_0\otimes_{O_{\breve{K}}}S;A,\iota_A,\lambda_A,\eta_A^p,\{\eta_{A,v}\}_{v\in\tV_F^{(p)}\cap\tV_F^\spl})\in\cX_0^\wedge(S)
\]
for a connected scheme $S$ in $\Sch'_{/O_{\breve{K}}}\cap\Sch_{/O_{\breve{K}}}^\rv$ with a geometric point $s$. In particular, $A[p^\infty]$ is supersingular.  By \cite{RZ96}*{Proposition~6.29}, we can choose an $O_E$-linear quasi-isogeny
\[
\rho\colon A\times_S(S\otimes_{O_{\breve{K}}}\ol\dF_p)\to\bbA\otimes_{\ol\dF_p}(S\otimes_{O_{\breve{K}}}\ol\dF_p)
\]
of height zero such that $\rho^*\lambda_{\bbA}\otimes_{\ol\dF_p}(S\otimes_{O_{\breve{K}}}\ol\dF_p)=\lambda_A\times_S(S\otimes_{O_{\breve{K}}}\ol\dF_p)$. We have
\begin{itemize}
  \item $(A[u^\infty],\iota_A[u^\infty],\lambda_A[u^\infty];\rho[u^\infty])$ is an element in $\cN^{\Phi_u}(S)$;

  \item the composite map
     \begin{align*}
     \bbV\otimes_\dQ\dA^{\infty,p}&\xrightarrow{\sim}V\otimes_\dQ\dA^{\infty,p}\xrightarrow{\eta_A^p}
     \Hom_{E\otimes_\dQ\dA^{\infty,p}}(\rH_1(A_{0,s},\dA^{\infty,p}),\rH_1(A_s,\dA^{\infty,p})) \\
     &\xrightarrow{\rho_{s*}\circ}\Hom_{E\otimes_\dQ\dA^{\infty,p}}(\rH_1(A_{0,s},\dA^{\infty,p}),\rH_1(\bbA_s,\dA^{\infty,p}))
     =\bbV\otimes_\dQ\dA^{\infty,p}
     \end{align*}
     is an isometry, which gives rise to an element $h^p\in\rU(\bbV)(\dA_F^{\infty,p})$;

  \item the same process as above will produce an element $h_p^\spl\in\prod_{v\in\tV_F^{(p)}\cap\tV_F^\spl}\rU(\bbV)(F_v)$;

  \item for every $v\in\tU$, the image of the map
     \[
     \rho_{s*}\circ\colon\Hom_{O_{E_v}}(A_{0,s}[v^\infty],A_s[v^\infty])\to\Hom_{O_{E_v}}(A_{0,s}[v^\infty],\bbA_s[v^\infty])\otimes\dQ
     =\bbV\otimes_FF_v
     \]
     is an $O_{E_v}$-lattice in the same $\rU(\bbV)(F_v)$-orbit of $\bbLambda_v$, which gives rise to an element $h_v\in\rU(\bbV)(F_v)/\bbL_v$.
\end{itemize}
Together, we obtain an element
\[
\((A[u^\infty],\iota_A[u^\infty],\lambda_A[u^\infty];\rho[u^\infty]),(h^p,h_p^\spl,\{h_v\}_{v\in\tU})\)
\in\cN^{\Phi_u}(S)\times\rU(\bbV)(\dA_F^{\infty,\ul{u}})/\prod_{v\in\tU}\bbL_v,
\]
and we define $\Upsilon(P)$ to be its image in the quotient, which is independent of the choice of $\rho$.

\begin{remark}
Both $\bbV$ and $\Upsilon^\rel$ depend on the choice of $\bbP$, while the isometry class of $\bbV$ does not.
\end{remark}

\begin{proposition}\label{pr:uniformization}
The morphism $\Upsilon^\rel$ \eqref{eq:uniformization} is an isomorphism. Moreover, for every $p$-basic element $\phi^\infty\in\sS(V\otimes_{\dA_F}\dA_F^\infty)^{\tilde{L}}$ and every $t\in F$ that is totally positive, we have
\begin{align}\label{eq:uniformization1}
\Upsilon^\rel\(\cZ_t(\phi^\infty)_{\tilde{L}}\res_{\cX_0^\wedge}\)=\sum_{\substack{x\in\rU(\bbV)(F)\backslash\bbV \\ (x,x)_{\bbV}=t}}
\sum_{h\in\rU(\bbV^x)(F)\backslash\rU(\bbV)(\dA_F^{\infty,\ul{u}})/\prod_{v\in\tU}\bbL_v}\bbphi(h^{-1}x)\cdot (\cN(x^\rel),h),
\end{align}
where
\begin{itemize}
  \item $\bbV^x$ denotes the orthogonal complement of $x$ in $\bbV$;

  \item $\bbphi$ is a Schwartz function on $\bbV\otimes_F\dA_F^{\infty,\ul{u}}$ such that $\bbphi_v=\phi^\infty_v$ for $v\in\tV_F^\fin\setminus(\tV_F^{(p)}\setminus\tV_F^\spl)$ and $\bbphi_v=\CF_{\bbLambda_v}$ for $v\in\tU$;

  \item $x^\rel$ is defined in \eqref{eq:arz3}; and

  \item $(\cN(x^\rel),h)$ denotes the corresponding double coset in \eqref{eq:uniformization}.
\end{itemize}
\end{proposition}

\begin{proof}
By a similar argument for \cite{RZ96}*{Theorem~6.30}, the morphism $\Upsilon$ is an isomorphism. Thus, $\Upsilon^\rel$ is an isomorphism as well by Corollary \ref{co:arz}.

For \eqref{eq:uniformization1}, by a similar argument for \cite{Liu19}*{Theorem~5.22}, the identity holds with $\cN(x^\rel)$ replaced by $\cN^{\Phi_u}(x)$. Then it follows by Corollary \ref{co:arz1}.

The proposition is proved.
\end{proof}

\subsection{Local indices at archimedean places}
\label{ss:archimedean}

In this subsection, we compute local indices at places in $\tV_E^{(\infty)}$.

\begin{proposition}\label{pr:index_arch}
Let $\tR$, $\tR'$, $\ell$, and $L$ be as in Definition \ref{de:kernel_geometric}. Let $(\pi,\cV_\pi)$ be as in Assumption \ref{st:representation}. Take an element $u\in\tV_E^{(\infty)}$. Consider an $(\tR,\tR',\ell,L)$-admissible sextuple $(\phi^\infty_1,\phi^\infty_2,\rs_1,\rs_2,g_1,g_2)$ and an element $\varphi_1\in\cV_\pi^{[r]\tR}$. Let $K_1\subseteq G_r(\dA_F^\infty)$ be an open compact subgroup that fixes both $\phi^\infty_1$ and $\varphi_1$, and $\fF_1\subseteq G_r(F_\infty)$ a Siegel fundamental domain for the congruence subgroup $G_r(F)\cap g_1^\infty K_1 (g_1^\infty)^{-1}$. Then for every $T_2\in\Herm_r^\circ(F)^+$, we have
\begin{multline*}
\vol^\natural(L)\cdot\int_{\fF_1}\varphi^\tc(\tau_1g_1)
\sum_{T_1\in\Herm_r^\circ(F)^+}I_{T_1,T_2}(\phi^\infty_1,\phi^\infty_2,\rs_1,\rs_2,\tau_1g_1,g_2)_{L,u}\rd\tau_1 \\
=\frac{1}{2}\int_{\fF_1}\varphi^\tc(\tau_1g_1)\sum_{T_1\in\Herm_r^\circ(F)^+}
\fE_{T_1,T_2}((\tau_1g_1,g_2),\Phi_\infty^0\otimes(\rs_1\phi^\infty_1\otimes(\rs_2\phi^\infty_2)^\tc))_u\rd\tau_1,
\end{multline*}
in which both sides are absolutely convergent. Here, the term $\fE_{T_1,T_2}$ is defined in Definition \ref{de:analytic_kernel} with the Gaussian function $\Phi_\infty^0\in\sS(V^{2r}\otimes_{\dA_F}F_\infty)$ (Notation \ref{st:h}(H3)), and $\vol^\natural(L)$ is defined in \cite{LL}*{Definition~3.8}.
\end{proposition}

\begin{proof}
This is simply \cite{LL}*{Proposition~10.1}.
\end{proof}

\subsection{Proof of main results}
\label{ss:proof}

The proofs of Theorem \ref{th:main}, Theorem \ref{th:aipf}, and Corollary \ref{co:aipf} follow from the same lines as for \cite{LL}*{Theorem~1.5}, \cite{LL}*{Theorem~1.7}, and \cite{LL}*{Corollary~1.9}, respectively, written in \cite{LL}*{Section~11}. However, we need to take $\tR$ to be a finite subset of $\tV_F^\spl\cap\tV_F^\heartsuit$ containing $\tR_\pi$ and of cardinality at least $2$, and modify the reference according to the table below.

\begin{center}
\begin{tabular}{|c|c|}
  \hline
  This article & \cite{LL} \\ \hline
  Proposition \ref{pr:uniqueness} & Proposition~3.6 \\ \hline
  Proposition \ref{pr:eisenstein} & Proposition~3.7 \\ \hline
  Proposition \ref{pr:index_split} & Proposition~7.1 \\ \hline
  Proposition \ref{pr:index_inert} & Proposition~8.1 \& Proposition~9.1 \\ \hline
  Proposition \ref{pr:index_ramified} & (not available) \\ \hline
  Proposition \ref{pr:index_arch} & Proposition~10.1 \\
  \hline
\end{tabular}
\end{center}

\begin{remark}\label{re:galois}
When $\tS_\pi=\emptyset$, Theorem \ref{th:main}, Theorem \ref{th:aipf}, and Corollary \ref{co:aipf} can all be proved without \cite{LL}*{Hypothesis~6.6}. In fact, besides Proposition \ref{pr:index_inert}(2) (which we do not need as $\tS_\pi=\emptyset$), the only place where \cite{LL}*{Hypothesis~6.6} is used is \cite{LL}*{Proposition~6.9(2)}. However, we can slightly modify the definition of $(\dS^\tR_\dL)^{\langle\ell\rangle}_{L_\tR}$ in Definition \ref{de:tempered_generic}(2) such that it is the ideal of $\dS^\tR_\dL$ of elements that annihilate
\[
\bigoplus_{u\in\tV_E^\fin\setminus\tV_E^{(\ell)}}\rH^{2r}_\dag(X_{L_\tR L^\tR,u},\dQ_\ell(r))\otimes_\dQ\dL,
\]
where $\rH^{2r}_\dag(X_{L_\tR L^\tR,u},\dQ_\ell(r))\otimes_\dQ\dL$ is the $\dQ_\ell\otimes_\dQ\dL$-submodule of $\rH^{2r}(X_{L_\tR L^\tR,u},\dQ_\ell(r))\otimes_\dQ\dL$ generated by the image of the cycle class map $\CH^r(X_{L_\tR L^\tR,u})\to\rH^{2r}(X_{L_\tR L^\tR,u},\dQ_\ell(r))\otimes_\dQ\dL$. Theorem \ref{th:split_tempered} implies that there exists element in $(\dS^\tR_{\dQ^\ac})^{\langle\ell\rangle}_{L_\tR}\setminus \fm_\pi^\tR$ that annihilates $\rH^{2r}_\dag(X_{L_\tR L^\tR,u},\dQ_\ell(r))\otimes_\dQ\dQ^\ac$ as long as $u$ satisfies $\ul{u}\in\tR\cap\tV_F^\spl\cap\tV_F^\heartsuit$ and $\tV_F^{(p)}\cap\tR\subseteq\tV_F^\spl$ where $p$ is the underlying rational prime of $u$. It follows that with this new definition of $(\dS^\tR_\dL)^{\langle\ell\rangle}_{L_\tR}$, \cite{LL}*{Proposition~6.9(2)} holds when $\tR\subseteq\tV_F^\spl\cap\tV_F^\heartsuit$ without assuming \cite{LL}*{Hypothesis~6.6}.
\end{remark}

\begin{remark}\label{re:final}
Finally, we explain the main difficulty on lifting the restriction $F\neq\dQ$ (when $r\geq 2$). Suppose that $F=\dQ$ and $r\geq 2$. Then the Shimura variety $X_L$ from Subsection \ref{ss:atl} is never proper over the base field. Nevertheless, it is well-known that $X_L$ admits a canonical toroidal compactification which is smooth. However, to run our argument, we need suitable compactification of their integral models at every place finite place $u$ of $E$ as well. As far as we can see, the main obstacle is the compactification of integral models using Drinfeld level structures when $u$ splits over $F$, together with a vanishing result like Theorem \ref{th:split_tempered}.
\end{remark}

\end{document}